\documentclass[reqno]{amsart}
\usepackage{graphicx} 
\usepackage[margin=1.5in]{geometry}
\usepackage{kotex}
\usepackage{amsthm}
\usepackage[dvipsnames]{xcolor}
\usepackage{amssymb}
\usepackage{amsmath}
\usepackage[mathscr]{euscript}
\usepackage{amsfonts}
\usepackage{mathtools}
\usepackage{enumitem}
\usepackage{hyperref}
\usepackage{verbatim}

\newtheorem{theorem}{Theorem}[section]
\newtheorem{lemma}[theorem]{Lemma}
\newtheorem{proposition}[theorem]{Proposition}
\newtheorem{corollary}[theorem]{Corollary}

\newtheorem{definition}[theorem]{Definition}
\newtheorem{remark}[theorem]{Remark}

\newtheorem{thm}[theorem]{Theorem}

\newtheorem{prop}[theorem]{Proposition}

\numberwithin{equation}{section}

\usepackage[normalem]{ulem}

\newcommand{\mc}[1]{{\mathcal #1}}
\newcommand{\mf}[1]{{\mathfrak #1}}

\newcommand{\bb}[1]{{\mathbb #1}}
\newcommand{\bs}[1]{{\boldsymbol #1}}
\newcommand{\ms}[1]{{\mathscr #1}}

\newcommand{\mtt}[1]{{\mathtt #1}}

\newcommand{\rom}[1]{%
  \textup{\uppercase\expandafter{\romannumeral#1}}%
  }

\newcommand{\cb}[1]{{\color{blue} #1}}

\newcommand{\cp}[1]{{\color{purple} #1}}

\def\wt{\widetilde}

\def\build#1_#2^#3{ \mathrel{\mathop{\kern 0pt#1}\limits_{#2}^{#3}}}

\title[Martingale problems for critical zero range processes]{
Dimension-decaying diffusion processes as the scaling limit of
condensing zero-range processes} \author{Johel Beltr\'an,
Kyuhyeon Choi, Claudio Landim}

\address{\noindent PUCP, Av. Universitaria cdra. 18, San Miguel,
Ap. 1761, Lima 100, Per\'u.  \newline e-mail: \rm
\texttt{johel.beltran@pucp.edu.pe} }

\address{\noindent
Massachusetts Institute of Technology, 
Department of Mathematics,
77 Massachusetts Avenue
Cambridge, MA 02139-4307, USA. \newline e-mail: \rm
\texttt{kyuhchoi@mit.edu} } 

\address{IMPA, Estrada Dona Castorina 110, J. Botanico, 22460 Rio de
Janeiro, Brazil and Univ. Rouen Normandie, CNRS,
LMRS UMR 6085,  F-76000 Rouen, France. \\
e-mail: \texttt{landim@impa.br} }

\begin{document}

\begin{abstract}
In this article, we prove that, on the diffusive time scale,
condensing zero-range processes converge to a dimension-decaying
diffusion process on the simplex
\[
\Sigma = \{(x_1,\dots,x_S) : x_i \ge 0,\; \sum_{i\in S} x_i = 1\},
\]
where $S$ is a finite set. This limiting diffusion has the distinctive
feature of being absorbed at the boundary of the simplex. More
precisely, once the process reaches a face
\[
\Sigma_A = \{(x_1,\dots,x_S) : x_i \ge 0,\; \sum_{i\in A} x_i = 1\},
\qquad A \subset S,
\]
it remains confined to this set and evolves in the corresponding
lower-dimensional simplex according to a new diffusion whose
parameters depend on the subset $A$. This mechanism repeats itself,
leading to successive reductions of the dimension, until one of the
vertices of the simplex is reached in finite time. At that point, the
process becomes permanently trapped.

The proof relies on a method to extend the domain of the associated
martingale problem, which may be of independent interest and useful
in other contexts.
\end{abstract}

\maketitle

\section{Introduction}

Metastability is a physical phenomenon that is ubiquitous in
first-order phase transitions. Early attempts at a precise description
can be traced back, at least, to Maxwell \cite{M75}. Following the
seminal work of Cassandro, Galves, Olivieri, and Vares \cite{cgov84},
and building on the foundational contributions of Lebowitz and Penrose
\cite{LP71}, numerous rigorous mathematical theories have been
developed to describe metastable phenomena \cite{BL10, M22, RES,
B25}. We refer to \cite{OV05, BH15, l-review} for recent monographs.

These theories have been applied and further developed in a wide range
of contexts, including statistical mechanics \cite{LPN22, KS25},
neural networks \cite{BN13, LM22}, molecular dynamics \cite{LRR24},
finance \cite{MS23}, population dynamics \cite{VA24}, and dynamical
systems \cite{DW12}, to mention a few.

Many of these theories were developed from the study of specific
examples of stochastic dynamics, with the goal of formalizing the
metastable behavior observed in such systems. The model reduction
approach \cite{BL10, BL12B, RES}, for instance, originated from the
analysis of condensing zero-range processes.

This stochastic dynamics, introduced by Evans \cite{E00}, describes
the evolution of particles on a finite set $S$ and is conservative, in
the sense that the total number of particles is preserved. Its
stationary states—one for each fixed number of particles—exhibit a
peculiar structure known in the physics literature as
\emph{condensation}.  Mathematically, this means that, under the
stationary distribution and above a certain critical density, a
macroscopic number of particles concentrates on a single site
\cite{E00, JMP00}. This phenomenon has been observed and studied in a
variety of contexts, including shaken granular systems, growing and
rewiring networks, traffic flows, and wealth condensation in
macroeconomics. We refer to \cite{EH05} for a comprehensive review.

Condensation in a class of zero-range dynamics was first established
rigorously in \cite{gss} by proving the equivalence of ensembles in
the thermodynamic limit. This result was later refined in \cite{al,
al2, agl} and shown to remain valid even when the total number of
sites is fixed while the number of particles diverges \cite{BL12}.

Once the presence of a condensate in the stationary state is
established, it is natural to investigate its time evolution
\cite{gl}. This problem was addressed in \cite{BL12} for
super-critical reversible dynamics, where the authors proved that, on
an appropriate time scale, the position of the condensate evolves as a
random walk whose jump rates are proportional to the capacities of the
underlying random walks. This result was subsequently extended to
super-critical totally asymmetric dynamics on a finite discrete torus
in \cite{Landim2014}, to the general super-critical case in
\cite{SZRP}, and to the critical symmetric case in \cite{MMC, RES}.

To describe the problem we examine in this article, fix a finite set
$S$, and jump rates $r\colon S\times S\to \bb R_+$.  Assume that the
continuous-time $S$-valued Markov chain associated to the jump rates
$r(i,j)$ is irreducible. Denote by $(m_i :i\in S)$ its unique
stationary state. Fix the jump rate $g\colon\bb N\to [0, \infty)$
given by $g(0)=0$, $g(n) = 1 + (b/n)$, $n\ge 1$, where $b>0$ is a
fixed parameter.  The zero-range dynamics associated to the pair
$(r,g)$ is the $\bb N^S$-valued Makov chain in which a particle at
site $i$ jumps to $j$ at rate $g(p)\, r(i,j)$ if there are $p$
particles at site $i$.

In this article, we investigate the mechanism through which a
condensate is formed.  In this model, two relevant nucleation time
scales arise.  Starting from an initial configuration with positive
particle density at each site, the process evolves on a linear time
scale according to the solution of an ordinary differential equation,
converging to a state in which all particles concentrate on the sites
where the invariant measure $m$ attains its maximum \cite{ABCJ}.

We consider here the next time-scale.  The dynamics is said to be
super-critical if $b>1$, and critical if $b=1$ because in the first
case the condensate evolves is the time-scale $N^{1+b}$ \cite{BL12,
Landim2014, SZRP}, while it in the second one it evolves in the
time-scale $N^2\ln N$ \cite{MMC, RES}. If $b<1$ there is no
condensation. 

Assume, for simplicity, that the stationary state
of the the underlying $S$-valued Markov chain is the uniform measure,
and that the process starts from an initial configuration with
positive particle density at each site.

In \cite{BJL}, the authors showed that in the super-critical case on
the diffusive time scale, the evolution of the particle density
— defined as the number of particles at each site divided by the total
number of particles — converges to an absorbed diffusion process on the
simplex
$\Sigma = \{ \, (x_1, \dots , x_{S} ) \in \bb R^{S} : x_i\geq 0 \,,
\,\; \sum_{i\in S}x_i = 1 \, \}$. The generator $\mf L$ of this
diffusion is given by
\begin{equation*}
(\mf LF) (x) \, =\,   \sum_{i\in S} (\nabla_{\bs b_i} F)(x)
+ \frac{1}{2} \sum_{i,j\in S} m_i \, r(i,j)\,
[\, (\partial_{x_j} - \partial_{x_i})^2 F]\, (x) \,, \quad x\in \Sigma\,,
\end{equation*}
where 
\begin{equation*}
(\nabla_ { {\bs b}_i}  F) (x)  \,= \, b\, \mtt 1\{x_i > 0\}
\left( \frac{{m}_i}{x_i}\right) \sum_{j\in S}
r(i,j)\, [ (\partial_{x_j} - \partial_{x_i}) F]\, (x) \,, 
\quad x\in \Sigma\, .
\end{equation*}
Note that the drift diverges as the diffusion approaches the boundary,
and that the parameter $b$ appears only as a multiplicative constant
of the drift.

As mentioned in the abstract, this limiting diffusion has the
distinctive feature of being absorbed at the boundary of the
simplex. More precisely, once the process reaches a face
$ \Sigma_A = \{(x_1,\dots,x_S) : x_i \ge 0,\; \sum_{i\in A} x_i =
1\}$, $A \subset S$, it remains confined to this set and evolves in
the corresponding lower-dimensional simplex according to a new
diffusion whose parameters depend on the subset $A$. This mechanism
repeats itself, leading to successive reductions of the dimension,
until one of the extreme points of the simplex is reached in finite
time. At that instant, the process becomes permanently trapped.  We
named such a process a {\it dimension-decaying} diffusion process.

One might be tempted to attribute the absorption at the boundary to
the divergence of the drift. This interpretation, however, is
incorrect, since for $b<1$ the process is expected to be reflected at
the boundary. Thus, the multiplicative parameter $b$, which may appear
innocuous at first glance, plays a fundamental role in determining the
qualitative behavior of the diffusion.

Although multidimensional diffusions with boundaries have been
extensively studied since the seminal work of Wentzell \cite{W60} (see
also \cite{SU65} and \cite[Section V.6]{IW14}), we are not aware of
examples in the literature exhibiting this type of behavior, nor of a
theoretical framework that adequately accounts for it. A distinctive
feature of the process considered here is the divergence of the drift
at the boundary, which implies that $\mf Lu$ fails to be continuous up
to the boundary even for smooth functions $u$, a regularity assumption
that is typically imposed in the study of diffusions with boundaries.

The proof presented in \cite{BJL} is divided into two main steps. First,
the authors showed that any limiting distribution of the process
solves an associated martingale problem. Second, they proved that this
martingale problem admits a unique solution. While the argument in the
second step is fairly general and extends to the critical case, the
first step relies on the construction of a superharmonic function
belonging to the domain of the generator. This part of the argument is
specific to the supercritical case and, in fact, contains a flaw, which
is corrected in the present work.

In this article, we propose a method for extending the domain of
generators that is particularly useful for the study of
dimension-decaying diffusions. This extension allows the inclusion in
the generator’s domain of functions that are discontinuous at the
boundary of the simplex. More precisely, functions for which
$\lim_{n} F(x^{(n)} )$ may be different from $F(x)$ for sequences
$x^{(n)} = (x^{(n)} _1, \dots, x^{(n)} _S)$ such that $x^{(n)}_i>0$,
$\lim_{n} x^{(n)}=x$, $x_i=0$.  Such an extension considerably
simplifies the construction of functions with prescribed properties in
the domain of the generator, in particular superharmonic
functions. Therefore, the interest of this article lies both in the
specific result on nucleation for critical condensing zero-range
processes and in the general methodology proposed to address such
problems.

To conclude this introduction, let us reinforce that for $b<1$, there
is no condensation, and one expects reflection at the boundary with a
positive local time at the boundary. Proving the convergence of
zero-range dynamics to the corresponding diffusion is an interesting
open problem.

\section{Model and Main Results}
\label{sec02}

We present in this section the main results of the article, and
introduce the notation used throughout the article.  Let
$S= \{1, \dots, \mtt p \}$ be a finite set with at least two elements,
$|S| = \mtt p \geq 2$.  Elements of $S$ are represented by the letters
$i$, $j$, $k$.  Denote by $\mc L_S$ the generator of a $S$-valued,
continuous time Markov chain
\begin{equation}
\label{23}
\cb{ (\mc L_S f)(i) } \,:=\, \sum_{j\in S} r(i,j)\, \big[\, f(j) -
f(i)\,\big]\,, \quad f\colon S\to \bb R\,.
\end{equation}
Assume that the Markov chain is irreducible and that $r(i,i)=0$, for
all $i\in S$. Denote by $\lambda_i$ the holding rates:
$\cb{\lambda_i}: = \sum_{j\neq i} r(i,j)$, and by $\cb{ (m_i)_{i\in S}
}$ the unique stationary state.

Let $\cb{(\bs e_i)_{i\in S}}$ stand for the canonical vectors in
$\bb R^S$ and define
\begin{equation}
\label{17}
{\color{blue} \bs v_i} \,: = \,\sum_{j\in S} r(i,j)\, 
(\bs e_j - \bs e_i), \quad i \in S\,.
\end{equation}
Throughout this work, we adopt the convention that vectors are denoted
in boldface, while scalars are typeset in standard font.  As
$m(\cdot)$ is the stationary state,
\begin{equation}
\label{eq:invar}
\sum_{i\in S} {m}_i {\bs v}_i = \bs 0\,.
\end{equation}

\subsection*{Condensing Zero-range process}

Denote by $\cb{\eta = (\eta_i)_{i\in S}}$ the elements of $\bb N^S$,
called configurations, and by $\mc H_N\subset \bb N^S$, $N\in \bb N$,
the set of configurations with $N$ particles:
\begin{equation*}
\cb{\mc H_N}  \,:=\,  \{ \eta = (\eta_i)_{i\in S}\in\bb N^S:
\sum_{i\in S} \eta_i = N\}\,.
\end{equation*}

For each $i\in S$, let $g_i\colon\bb N\to [0, \infty)$ be the jump
rate of particles at site $i$, so that $g_i(0)=0$, $g_i(n)>0$ for
$n>0$. Assume that
\begin{equation}
\label{34}
\lim_{n\to \infty} n\left(\frac{g_i(n)}{m_i}-1\right) \,=\, b  \,,
\quad \text{for some $b\ge 1$}\,.
\end{equation}

The zero-range process $\{\eta_N(t) : t\geq 0\}$ is the
$\mc H_N$-valued continuous-time Markov chain induced by the generator
$L_N$ given by
\begin{equation*}
\cb{ (L_N f)(\eta) } \,:=\,
\sum_{i,j\in S} g_i(\eta_i)\,
r(i,j) \, (f(\eta^{i,j})-f(\eta))\;,  \quad \eta \in \mc H_N \,,
\quad f\colon  \mc H_N \rightarrow \bb R \,.
\end{equation*}
In this formula, $\eta^{i,j}\in \mc H_N$ is the configuration obtained
from $\eta$ by moving a particle from site $i$ to site $j$. More
precisely, if $\eta_i = 0$,  then $\eta^{i,j} = \eta$ , and if
$\eta_i\geq 1$, 
$$
\cb{ (\eta^{i,j})_k } \,:=\, 
\begin{cases}
    \eta_k-1 & \text{if } k=i \\
    \eta_k+1 & \text{if } k=j \\
    \eta_k & \text{if } k \in S \setminus \{i,j\} \,.
\end{cases}
$$

Let $\Sigma \subset \bb R^S$ be the set of non-negative coordinates
whose sum is $1$:
$$
\cb{ \Sigma } \, :=\,
\Big\{ \, (x_i)_{i\in S} \in \bb R^S : x_i\geq 0
\text{ for all }i\in S \text{ and } \sum_{i\in S}x_i = 1 \, \Big\}.
$$
Since $\mc H_N$ consists of vectors summing to $N$, we may embed
$\mc H_N$ into $\Sigma$. Let $\iota_N:\mc H_N \rightarrow \Sigma$ be
the projection given by
\begin{equation*}
\cb{ \iota_N(\eta)_i } \,:=\,  \frac{\eta_i}{N} \,, \quad i\in S\,.
\end{equation*}
Let $\Sigma_N$ be the image of $\mc H_N$ under $\iota_N$, that is,
$$
\cb{ \Sigma_N}  \,:=\,  \iota_N(\mc H_N)\,.
$$
Consequently, $\Sigma_N$ becomes a subset of $\Sigma$ consisting of
vectors whose coordinates are rational numbers with denominator $N$.

Let $X^N_t$ denote the $\Sigma_N$-valued Markov chain obtained by
projecting the speeded-up process $\eta_N(t)$ via the map $\iota_N$:
$$
\cb{X^N_t} \,:=\,  \iota_N(\eta_N(tN^2)), \; t\geq 0 \,.
$$
This defines the rescaled zero-range process on ​$\Sigma_N$, a
continuous-time Markov chain $\{X^N_t:t\geq 0\}$ taking values in
$\Sigma_N$ associated with the generator
\begin{equation*}
\cb{ (\mc L_N f)(x)  } \,:=\,  N^2\sum_{\substack{i,j\in S, \\ x_i>0}}
g_i(Nx_i) \, r(i,j)\Big[ \, f \Big( x+
\frac{\bs e_j- \bs e_i}{N}\Big)-f(x) \, \Big ] \;,
\; x\in \Sigma_N.
\end{equation*}

Denote by $\cb{D(\bb R_+, \Sigma)}$ the space of $\Sigma$-valued
right-continuous trajectories with left limits equipped with the
Skorokhod topology. For each $x\in \Sigma_N$, let $\cb{\bb P^N_{x}}$
be the probability measure on $D(\bb R_+, \Sigma)$ induced by the
Markov chain $X^N_t$ starting from $x$.

Consider a sequence $(x_N:N\ge 1)$ that converges to some
$x\in \Sigma$ as $N\to\infty$. The main result of this article states
that the sequence of measures $\bb P^N_{x_N}$ ​ converges in
distribution to a measure $\bb P_x$, which is induced by a
dimension-decaying diffusion on $\Sigma$ and characterized as the
unique solution to a corresponding martingale problem.

\subsection*{Martingale Problem} \label{sec:mp}

To introduce the martingale problem, we first define its domain.  Let
$\bs 1$ be the vector with all coordinates equal to $1$:
$\cb{\bs 1 = \sum_{i\in S} \bs e_i}$, and let $\mathring{\Sigma}$ be
the interior of the set $\Sigma$, defined as
$$
\cb{\mathring{\Sigma}} \,: =\,
\{x \in \Sigma :   x_i > 0 \; \forall i \in S \,\} \,.
$$
Since $\mathring{\Sigma}$ is an open subset of the hyperplane
$$ \{x \in \bb R^S : \sum_{i\in S} x_i = 1\}, $$
a vector $\bs V=(V_i)_{i\in S} \in \bb R^S$ is said to be a tangent
vector to $\mathring{\Sigma}$ if it is orthogonal to $\bs 1$. Denote
by $T_\Sigma$ the linear space of formed by these vectors:
\begin{equation}
\label{19}
\cb{T_\Sigma} \,:=\, \big\{ \bs V \in \bb R^S : 
\bs V \cdot \bs 1 \,=\, \sum_{i\in S} V_i = 0\, \big\}\,,
\end{equation}
where $\cdot$ denotes the standard inner product in $\bb R^S$.  Note
that each vector $\bs v_i$ introduced in \eqref{17} belongs to
$T_\Sigma$.

Denote by $\cb{C(\mathring{\Sigma}) }$ the space of continuous
functions $f\colon \mathring{\Sigma} \to \bb R$, and by
$\cb{C(\Sigma)} $ the elements of $C(\mathring{\Sigma})$ which can be
continuously extended to $\Sigma$.  A function
$f\in C(\mathring{\Sigma})$ is said to be differentiable at
$x\in \mathring{\Sigma}$ if there exists a vector in $T_\Sigma$,
denoted by $\cb{\nabla^\Sigma f(x)}$, such that
\begin{equation}
\label{eq:differentiability}
\lim_{ \bs V \in T_{\Sigma} ,  |\bs V| \to 0}
\frac{ f(x + \bs V) - f(x) -\bs V \cdot \nabla^{\Sigma} f(x)}{|\bs V|}
= 0 \,.
\end{equation}

For convenience, we often abbreviate
$\bs V \cdot \nabla^\Sigma f$ as $\cb{\nabla_{\bs V} f}$. Denote by
$\cb{C^1(\mathring{\Sigma})}$ the space of functions
$f\in C(\mathring{\Sigma})$ which are differentiable at every
$x\in \mathring{\Sigma}$ and such that the map
$x \mapsto \nabla_{\bs V} f (x)$ belongs to
$C(\mathring{\Sigma})$ for any $\bs V \in T_\Sigma$.

Finally, let $\cb{C^2(\mathring{\Sigma})}$ be the space of functions
$f \in C^1(\mathring{\Sigma})$ such that
$\nabla_{\bs V} f \in C^1(\mathring{\Sigma})$ for any
$\bs V \in T_\Sigma$.  Denote by $C^1(\Sigma)$, $C^2(\Sigma)$ the
elements of $C^1(\mathring{\Sigma})$, $C^2(\mathring{\Sigma})$ which
can be continuously extended to $\Sigma$, respectively:
\begin{gather*}
\cb{ C^1(\Sigma)} \,:=\, \big \{ \, f  \in C(\Sigma) \cap 
C^1(\mathring{\Sigma}) : 
\nabla_{\bs V} f \text{ continuously extends to }  \Sigma
\text{ for all } \bs V \in T_\Sigma \, \big \},
\\
\cb{ C^2(\Sigma)} \,:=\, \big \{  \, f  \in C^1(\Sigma) 
\cap C^2(\mathring{\Sigma}) : 
\nabla_{\bs V} (\nabla_{\bs W} f) \text{ continuously extends to } 
\Sigma \text{ for all } \bs V, \bs W \in T_\Sigma \,\big\} \,.
\end{gather*}
We denote by the same symbol $\nabla_{\bs V} f$, $\nabla_{\bs V}
(\nabla_{\bs W} f)$ the continuous extension to $\Sigma$ of these
functions. 

By Whitney's theorem \cite{W34}, for any function $F$ in
$C^1(\Sigma)$, there exists an extension
$\hat F\colon \bb H_1 := \{x\in \bb R^S : \sum_{i\in S} x_i = 1\} \to
\bb R$ of class $C^1(\bb H_1)$ such that
$(\nabla_{\bs V} \hat F)(x) = (\nabla_{\bs V} F)(x)$ for all
$x\in \mathring{\Sigma}$. We may further extend $\hat F$ to $\bb R^S$
by setting the extension, denoted by $\tilde F$, to be constant along
the orthogonal direction to $\Sigma$:
$\tilde F(x+k \bs 1) = \hat F(x)$ for all $x\in \bb H_1$, $k\in\bb R$,
where, recall, $\bs 1$ is the vector with all coordinates equal to
$1$. This procedure provides a function
$\tilde F\colon \bb R^S\to \bb R$ of class $C^1(\bb R^S)$ such that
\begin{equation*}
\bs V \cdot \nabla^\Sigma F (x) \,=\,  \bs V \cdot (\nabla \tilde{F} )
(x) \,=\,
\sum_{i\in S} V_i \, \partial_{x_i} \tilde{F} (x) \;\;
\text{for all}\;\;  x \in \mathring{\Sigma} \,,\;\; \bs V  \in T_\Sigma\,,
\end{equation*}
where $\bs V = \sum_{i\in S} V_i \, \bs e_i$, and $\nabla \tilde{F}$
is the gradient of $\tilde{F}$ with respect to the standard inner
product in $\bb R^S$.

Analogously, any function $F$ in $C^2(\Sigma)$ can be extended to
a function in $C^2(\bb R^S)$: there exists a function $\tilde F\colon
\bb R^S\to \bb R$ of class $C^2(\bb R^S)$ such that
\begin{equation*}
\nabla_{\bs V} (\nabla_{\bs W} F) (x) \,=\,
\nabla_{\bs V} (\nabla_{\bs W} \tilde F) (x)\;\;
\text{for all}\;\;  x\in
\mathring{\Sigma} \,,\; \bs V\,, \bs W \in T_\Sigma\,.
\end{equation*}

\begin{definition}
\label{def:dom}
For $i\in S$, define the vector field
$\bs b_i\colon \Sigma\to T_\Sigma$ by
\begin{equation*}
{\color{blue} \bs b_i(x)} \,:=\,  b \, \mtt 1 \{x_i > 0\} \, 
\left( \frac{{m}_i}{x_i}\right) \bs v_i\,  \quad x\in \Sigma\,,.
\end{equation*}
where $\cb{\mtt 1\{A\}}$ represents the indicator function of the set
$A$.  The associated differential operator, denoted by
$\nabla_{\bs b_i}$ acts on $F\in C^1(\Sigma)$ as
\begin{equation*}
\cb{ (\nabla_ { {\bs b}_i}  F) (x)}  \,:= \, b \,  \mtt 1\{x_i > 0\}
\left( \frac{{m}_i}{x_i}\right) (\nabla_{\bs v_i} F) (x) \,,
\quad x\in \Sigma\, .
\end{equation*}
In addition, let
\begin{equation*}
{\color{blue} \mc D_i}  \,:=\, \left\{ \, F\in C^2(\Sigma) \; :  \; \textrm{$
\nabla_{\bs b_i }  \, F$ is continuous on $\Sigma$} \,\right\}
\quad {\rm and} \quad \cb{\mc D_A} \,:= \, \bigcap_{i\in A}\mc D_i
\end{equation*}
for any nonempty subset $A$ of $S$.
\end{definition}

Denote by $(a_{ij} :  i,j \in S)$ the matrix whose entries are given by
\begin{equation}
\label{22}
\cb{a_{ij}}  \, :=\,  -\, m_i \, r(i,j)  \,=\, -\,
{m}_i \, {\bs v}_i\cdot {\bs e_j}, \;\;
i\neq j\in S\,, \quad
\cb{a_{ii}}  \, :=\, m_i \, \lambda_i \,, \quad i\in S\,.
\end{equation}

\begin{definition}
Denote by $\mf L\colon \mc D_S \to C(\Sigma)$ the
differential operator defined by
\begin{equation}
\label{eq:L}
\cb{ (\mf LF) (x)} \,:=\,   (\nabla_{\bs b} F)(x)
+ \frac{1}{2} \sum_{i,j} m_i \, r(i,j)\,
[\nabla_{\bs e_i - \bs e_j} (\nabla_{\bs e_i - \bs e_j} F)](x), \quad x\in \Sigma\,,
\end{equation}
where $\bs b\colon \Sigma \to \bb R^S$ is the vector field, and
$\nabla_{\bs b} F$ the derivative defined by
\begin{equation*}
\cb{ {\bs b}(x) } \,:=\,  \sum_{i\in S} {\bs b}_i(x) \,, \quad
\cb{ (\nabla_{\bs b} F)(x) } \,:=\,
\sum_{i\in S} (\nabla_{{\bs b}_i} F)(x) \,, \quad x\in \Sigma\,,
\end{equation*}
respectively.
\end{definition}

Clearly, for any $C^2$-extension $\tilde{F} \colon \bb R^S \to \bb R$
of $F\in C^2(\Sigma)$,
\begin{gather*}
(\mf LF)(x) \,=\,  (\nabla_{\tilde{{\bs b}} } \, \tilde F) (x) \,+\,
(D_{\tilde{ {\bs a} }  } \, \tilde F) (x) \quad \forall\, x\in
\Sigma\,,
\end{gather*}
where
\begin{equation}
\label{eq:ba}
\begin{gathered}
( \nabla_{{\tilde{\bs b} }} \, F)(x) \,=\,  - \, b \, \sum_{i\in S}
\mtt 1\{x_i > 0\}\, \frac{1}{x_i}\, \sum_{k\in S}
a_{ik} \, [\, (\partial_{x_k} -  \partial_{x_i}) \tilde F\,] \,  (x)
\\
\quad {\rm and}
\quad ( D_{\tilde{{\bs a} }}  \, \tilde F)(x)
= \sum_{i,j\in S}  a_{ij}\, \partial_{x_i} \partial_{x_j}
\tilde{F}(x)\,.
\end{gathered}
\end{equation}

Denote by $\cb{C(\bb R_+, \Sigma)}$ the space of continuous
trajectories $\omega\colon \bb R_+\to \Sigma$ equipped with the
topology of uniform convergence on bounded intervals, and its
corresponding Borel $\sigma$-field $\ms F$. Denote by
$\cb{X_t}\colon C(\bb R_+, \Sigma)\to \Sigma$, $t\ge 0$, the process
of coordinate maps and by $(\ms F_t)_{t\ge 0}$ the generated
filtration $\cb{\ms F_t:=\sigma(X_s:s\le t)}$, $t\ge 0$. A probability
measure $\bb P$ on $C(\bb R_+, \Sigma)$ is said to start at
$x\in \Sigma$ when $\bb P[X_0=x]=1$.

\begin{definition}
\label{def1}
A probability measure $\bb P$ on $C(\bb R_+, \Sigma)$ is a solution
for the $(\mf L, \mc D_S)$-martingale problem if, for any
$H\in \mc D_S$,
\begin{equation}
\label{f06}
H(X_{t}) - \int_{0}^{t} (\mf L H)(X_s) \, ds\;,\quad t\ge 0
\end{equation}
is a $\bb P$-martingale with respect to the filtration $(\ms F_t)_{t\ge 0}$.
\end{definition}

We are now ready to state the main theorem.

\begin{thm}
\label{thm:mainthm}
For each $x\in \Sigma$, there exists a unique probability measure on
$C(\bb R_+, \Sigma)$, denoted by $\bb P_x$, which starts at $x$ and is
a solution of the $(\mf L,\mc D_S)$-martingale problem.  Furthermore,
let $\bb P^N_{x_N}$ be the probability measure on $D(\bb R_+, \Sigma)$
induced by the Markov chain $X^N_t$ starting from $x_N \in
\Sigma_N$. If $x_N$ converges to $x\in \Sigma$, then, $\bb P^N_{x_N}$
converges to $\bb P_x$ in the Skorohod topology.
\end{thm}

\begin{remark}
\label{rm10}
In Section \ref{sec8} we introduce and alternative martingale problem
and show in Theorem \ref{thm:altmart} that any solution of the
$(\mf L,\mc D_S)$-martingale problem is also a solution of the
alternative one. The uniqueness part in Theorem
\ref{thm:mainthm} is proved through the alternative martingale problem.
\end{remark}

The statement of the theorem is identical to \cite[Theorem 2.2]{BJL}
and \cite[Theorem 2.6]{BJL}, and we adopt a similar strategy.  We
begin by showing that the solution to the martingale problem is
boundary dimension-decaying, or, equivalently, absorbing at the
boundary (Theorem \ref{thm:abs}).  Using this property, we then
establish uniqueness, following the argument in \cite[Section 6]{BJL}.
Finally, we prove that the sequence $\bb P^N_{x_N}$ is tight and
converges to the unique solution of the martingale problem, adapting
the method from \cite[Section 7]{BJL}.

Although the overall structure of the proof follows that of
\cite{BJL}, the case \(b=1\) introduces a difficulty in establishing
the absorbing property of the solution to the martingale problem. In
particular, this step requires constructing a super-harmonic function
that lies in the domain of the generator, and this construction
becomes delicate when \(b=1\).

To construct such a function, we introduce a method, presented in
Section~\ref{sec:edmp}, which we refer to as the \emph{extension of
the domain}. The core idea of this approach is Theorem~\ref{thm:emp},
whose proof is given in Section~\ref{sec:proofthmemp}. This
result states that any solution to the $(\mf L,\mc D_S)$-martingale
problem is also a solution to a $(\mf L^{\mc E},\mc E_S)$-martingale
problem, where the domain $\mc E_S$ contains $\mc D_S$ and
$\mf L F = \mf L^{\mc E} F$ for all $F\in \mc D_S$. This result
therefore extends the domain of the generator $\mf L$.  and allows to
construct super-harmonic functions which do not belong to $\mc D_S$,
but only to $\mc E_S$.

The following two subsections summarize the additional results that
can be obtained.

\subsection*{A boundary dimension-decaying diffusion}
\label{sec:An absorbed diffusion}

For each $x\in \Sigma$, denote
$$
\ms A(x) \,:=\, \{j\in S: x_j = 0\},\;\; \ms B(x) \,:=\,
S\setminus \ms A(x).
$$
For all nonempty subset $B \subset S$, define
$h_B\colon C(\bb R_+, \Sigma) \to \bb R_+$ as the first time one of
the coordinates in $B$ vanishes
$$
\cb{ h_B(x)} \,:=\,   \inf\{t\geq 0: \prod_{j\in B} X_t(j) = 0 \}\,.
$$

Let $\cb{(\theta_t)_{t\geq 0}}$ be the semigroup of time translation
in $C(\bb R_+, \Sigma)$. Define a sequence of pairs of stopping times
and sets $(\sigma_n, \ms B_n)_{n\geq 0}$ as follows. Set
$\cb{\sigma_0 = 0}$ and $\cb{\ms B_0 = \ms B(X_0)}$.  For $n\geq 1$,
we define
$$
\cb{\sigma_n}\,:=\, \sigma_{n-1} + h_{\ms B_{n-1}} \circ
\theta_{\sigma_{n-1}},\;\;
\cb{\ms B_{n} }\,:=\,
\{j\in S: X_{\sigma_n}(j) > 0\}
$$
on $\{\sigma_{n-1} < \infty$\} and $\sigma_n = \infty$ on
$\{\sigma_{n-1} = \infty\}$.

We say that a probability measure $\bb P$ on $C(\bb R_+, \Sigma)$ is
absorbing if

$$
\bb P\{ \ms B_n \supseteq B(X_t) \text{ for all } t\ge \sigma_n \} =
1, \text{ for every } n\ge 0.
$$

If $\bb P$ is absorbing then $\bb P-$a.s., $(\ms B_n)_{n\geq 0}$ is
decreasing and
$$
\exists \;
1\leq n_0 \leq |\ms B_0| \text{ such that } \sigma_{n_0} = \infty
\;\text{ and }\; \ms B_{n-1} \supsetneq \ms B_n \text{ for all } 1\leq
n< n_0.
$$

As an intermediate step in proving the uniqueness of the martingale
problem solution, we prove the following theorem, which gives an
interesting property of the process itself.

\begin{thm}
\label{thm:abs}
For each $x\in \Sigma$, the probability measure $\bb P_x$ is
absorbing.
\end{thm}

\begin{remark}
In Section \ref{sec8}, we present further properties of the process.
Propositions \ref{prop:feller} and \ref{prop:strongmarkov} state that
the process has the Feller property. Proposition \ref{ft} states that
the time it takes to reach the boundary has finite expectation. In
particular, the time needed to reach the set of extremal points of the
simplex $\Sigma$ has finite expectation.
\end{remark}

\subsection*{Behavior after absorption}
\label{sec:BAA}

Similarly to \cite[Section 2.5]{BJL}, we expect the process $X_t$ to
have a recursive absorbing structure in the sense that, after
absorption, the process again follows the same dynamics with $\bs r$
replaced by the jump rates of the trace process.

For each $B\subset S$ with $|B|\geq 2$, consider the simplex
$\Sigma_B$ and its interior $\mathring{\Sigma}_B$:
$$
\cb{\Sigma_B} \,:=\,  \{x \in \Sigma : \sum_{j\in B} x_j =1 \}, \quad
\cb{\mathring{\Sigma}_B} \,:=\,
\{x \in \Sigma_B : x_j > 0, \forall j\in B\}.
$$
Mind that $\Sigma_B$, $\mathring{\Sigma}_B$ are subsets of $\bb R^S$
and not $\bb R^B$.

Since $\mathring{\Sigma}_B$ is an open subset of the  affine subspace
$$ \{x \in \bb R^S : \sum_{j\in B} x_j =1, x_i = 0 \text{ for all } i\in S\setminus B \} $$
of $\bb R^S$, a vector $\bs V$ in $\bb R^S$ will be said to be
tangent to $\mathring{\Sigma}_B$ if
$$
\sum_{j\in B} V_j = 0, \text{ and } V_i = 0 \text{ for all } i\in
S\setminus B.
$$
We denote by $\cb {T_{\Sigma_B}}$ the linear space of all vectors
tangent to $\mathring{\Sigma}_B$.

We extend the notion of differentiability introduced in
\eqref{eq:differentiability}.  Denote by $\cb{C(\mathring{\Sigma}_B) }$
the space of continuous functions
$f\colon \mathring{\Sigma}_B \to \bb R$, and by $\cb{C(\Sigma_B)} $ the
elements of $C(\mathring{\Sigma}_B)$ which can be continuously extended
to $\Sigma_B$.  A function $f\in C(\mathring{\Sigma}_B)$ is said to be
differentiable at $x\in \mathring{\Sigma}_B$ if there exists a vector in
$T_{\Sigma_B}$, denoted by $\nabla^{\Sigma_B} f(x)$, such that
\begin{equation*}
\lim_{ \bs V \in T_{\Sigma_B} , |\bs V| \to 0}  \frac{ f(x + \bs V) - f(x)
-\bs V \cdot \nabla^{\Sigma_B} f(x)}{|\bs V|} \,=\,0\,.
\end{equation*}

For convenience, we often abbreviate
$\bs V \cdot \nabla^{\Sigma_B} f$ as $\cb{\nabla_{\bs V} f}$. Denote
by $\cb{C^1(\mathring{\Sigma}_B)}$ the space of functions
$f\in C(\mathring{\Sigma}_B)$ which are differentiable at every
$x\in \mathring{\Sigma}_B$ and such that the map
$x \mapsto \nabla_{\bs V} f (x)$ belongs to $C(\mathring{\Sigma}_B)$
for any $\bs V \in T_{\Sigma_B}$.

Finally, let $\cb{C^2(\mathring{\Sigma}_B)}$ be the space of functions
$f \in C^1(\mathring{\Sigma}_B)$ such that
$\nabla_{\bs V} f \in C^1(\mathring{\Sigma}_B)$ for any
$\bs V \in T_{\Sigma_B}$.  Denote by $C^1(\Sigma_B)$, $C^2(\Sigma_B)$
the elements of $C^1(\mathring{\Sigma}_B)$, $C^2(\mathring{\Sigma}_B)$
which can be continuously extended to $\Sigma_B$, respectively:
\begin{align}
\label{eq:C2}
\cb{ C^1(\Sigma_B)} := \big \{ \, f  \in C(\Sigma_B) \cap 
C^1(\mathring{\Sigma}_B) : \!
\nabla_{\bs V} f \text{ continuously extends to }  \Sigma_B
\text{ for all } \bs V \in T_{\Sigma_B} \, \big \},\,
\end{align}
\begin{align*}
\cb{ C^2(\Sigma_B)} \,:=\, \big \{  \, f  \in C^1(\Sigma_B) 
\cap C^2(\mathring{\Sigma}_B) : 
\nabla_{\bs V} (\nabla_{\bs W} f) \text{ continuously extends to } 
\Sigma_B \text{ for all } \bs V, \bs W \in T_{\Sigma_B} \,\big\} \,.
\nonumber
\end{align*}
We denote by the same symbol $\nabla_{\bs V} f$, $\nabla_{\bs V}
(\nabla_{\bs W} f)$ the continuous extension to $\Sigma_B$ of these
functions.

Denote by
$$
\cb{ \bs{r^B} }\, :=\, \{r^B(x,y):x,y\in B\}
$$
the jump rates of the trace on $B$ of the Markov process generated by
$\mc L_S$. Detailed explanation of this process is given in Section
\ref{sec:trace}.  Let $\{\bs v_j^B:j\in B\}$ be the vectors in
$T_{\Sigma_B}$ defined by
\begin{equation}
\label{eq:vjB}
\cb{ \bs v_j^B } \,:=\, \sum_{k\in B} r^B(j,k)(\bs e_k-\bs
e_j) \,.
\end{equation}
where $\{\bs e_j:j\in B\}$ stands for the subset of the canonical
basis of $\bb R^S$ indexed by $B$,  and let
$\bs b^B:\Sigma_B\rightarrow \bb R^S$ be the vector field defined by
$$
\cb{ \bs b^B(x)} \,:=\, b \, \sum_{j\in B}
\frac{m_j}{x_j}\, \bs v_j^B \, \mtt  1\{x_j>0\},\quad x\in
\Sigma_B\,.
$$

Similar to $\mc D_S$ from Definition \ref{def:dom}, for $j\in B$, let
$\cb{\mc D^B_j}$ be the space of functions $H$ in $C^2(\Sigma_B)$ for
which the map
$x \mapsto \mtt 1\{x_j>0\} \, (m_j/x_j)\, (\nabla_{\bs v_j^B} H)
(x)$ is continuous on $\Sigma_B$, and let
\begin{equation}
\label{18}
\cb{\mc D^B_A} \,: =\,  \bigcap_{j\in A}\mc D^B_j\,,
\quad \text{for}\;\; \varnothing \subsetneq A \subset B\,.
\end{equation}

Let $\mf L^B \colon \mc D^B_B \to C(\Sigma_B)$ be the second
order differential operator which acts on functions in $\mc D^B_B$ as
\begin{equation}
\label{eq:LB}
\cb{(\mf L^B F)(x)} \,:=\,  (\nabla_{\bs b^B} F)(x)
+ \frac{1}{2}\sum_{j,k\in B} m_j \, r^B(j,k)
\, [\nabla_{\bs e_j - \bs e_k} (\nabla_{\bs e_j - \bs e_k} F)] (x),
\end{equation} 
where
\begin{equation*}
\cb{ (\nabla_{\bs b^B} F)(x)}  \,:=\,
b \,  \sum_{j\in B}
\frac{m_j}{x_j}\, \mtt 1\{x_j>0\}\, (\nabla_{\bs v_j^B} F) (x)
\end{equation*}
for $x \in \Sigma_B$ and $F \in \mc D^B_B$.

Fix $x$ in $\Sigma$ and assume that
$\ms A(x) = \{j \in S: x_j = 0\} \neq \varnothing$. Let
$B = \ms A(x)^c$.  Take a measure $\bb P_{x}$ which is a solution of
the $(\mf L, \mc D_S)$-martingale problem starting at $x$. By
Theorem \ref{thm:abs}, $\bb P_{x}$ is concentrated on trajectories
which belong to $C(\bb R_+,\Sigma_B)$.  Let $\bb P_{x}^B$ be the
restriction of $\bb P_{x}$ to $C(\bb R_+,\Sigma_B)$:
$$
\cb{ \bb P_{x}^B(\Xi) } \,:=\,
\bb P_{x}(\Xi) \,, \quad  \Xi \subset C(\bb R_+,\Sigma_B) \,.
$$
which is a probability measure on $C(\bb R_+,\Sigma_B)$.  Then the
following proposition holds analogously to \cite[Proposition
2.4]{BJL}.

\begin{prop}
\label{prop:restriction}
Fix $x$ in $\Sigma$. Assume that
$\ms A(x) = \{j \in S: x_j = 0\} \neq \varnothing$, and set
$B = \ms A(x)^c$. Let $\bb P_{x}$ be the unique solution of the
$(\mf L, \mc D_S)$-martingale problem with starting point $x$.  Denote
by $\bb P_{x}^B$ the restriction of $\bb P_{x}$ to
$C(\bb R_+,\Sigma_B)$. Then, the measure $\bb P_{x}^B$ solves the
$(\mf L^B, \mc D^B_B)$-martingale problem.
\end{prop}

{
\begin{remark}
\label{rem:sodo}
We may wish to apply the differential operators $\nabla_{\bs b^B}$,
$(\nabla_{\bs e_i - \bs e_j})^2 (i,j\in B)$ to functions that do not
belong to $C^2(\Sigma_B)$ but are locally smooth.  Accordingly, for
each \( x \in \Sigma \), we define local quantities
$(\nabla_{\bs b} F)(x)$,
$[\nabla_{\bs e_i - \bs e_j} (\nabla_{\bs e_i - \bs e_j} F)] (x)$ when
$F \in C^2(U)$ for some open neighborhood $U$ (in $\Sigma_B$) of
\( x \). In particular, we may interpret the equation \eqref{eq:L}
locally. This viewpoint is crucial in extending the domain of the
generator $\mf L$ (see the Definition \ref{def:ed}), which plays a
fundamental role in the proof of the martingale problem uniqueness.
In case of any ambiguity, we clarify that $\mf L$ refers to a local
second-order differential operator by explicitly stating it as such.
\end{remark}

\noindent{\bf Organization:} The article is organized as follows. In
the next section we introduce the trace on
$\varnothing \subsetneq B\subsetneq S$ of the $S$-valued Markov chain
induced by the generator $\mc L_S$. This process provides the
diffusion coefficient and the drift of the diffusion when it evolves
on $\Sigma_B$. In Section \ref{sec:edmp}, we introduce a relaxed
version of the martingale problem, extending the domain of the
generator. The main result of this section, Theorem \ref{thm:emp},
states that a solution of the orignal martingale problem is also a
solution of the extended martingale problem. The proof of this result,
presented in Section \ref{sec:proofthmemp}, appeals to maps
$J_A\colon \bb R_+^A \to \bb R_+$, $A\subset S$, introduced in Section
\ref{sec:aux}, which mimics the norm on $\Sigma_A$ but are adapted to
the domain of the generator.  In Section \ref{sec:proofthmemp} we
prove that the solution of the martingale problem is absorbing. The
proof is based on the construction of a superharmonic functions in the
domain of the maringale problem. This is the part of the argument
which requires the extension of the domain.  Finally, in Section
\ref{sec8}, we prove the convergence of the condensed zero-range
process to the boundary dimension-decaying diffusion process. In this
last section we present further properties of this process. We show,
for example, that it reaches one of the vertices of the simplex in a
time which has finite expectation.

\section{The trace process}
\label{sec:trace}

We introduce in this section the trace of the $S$-valued Markov chain
induced by generator $\mc L_S$ on a nonempty proper subset $B$ of $S$.
We also define a projection map
\begin{equation*}
\gamma_B \colon \bb R^S \to \{x\in\bb R^S : x_i =0 \;
\forall\; i\in S\setminus B\}
\end{equation*}
which plays an important role in the extension of the martingale
problem.

Recall from \cite[Section~6]{BL10} the definition of trace process,
and from \eqref{23} that $\mc L_S$ stands for the generator of the
$S$-valued, continuous time, irreducible Markov chain induced by the
jump rates $r$ over $S$.  Denote by $\cb{\mc D(\bb R_+,S)}$ the space
of $S$-valued, right-continuous trajectories with left-limits
$x\colon \bb R_+\to S$ equipped with the Skorohod topology and its
associated Borel $\sigma$-field.  Denote by $\cb{\bb P_j}$ the
probability measure on $\mc D(\bb R_+, S)$ induced by the Markov
process with generator $\mc L_S$, starting from state $j\in S$. For a
nonempty, proper subset $B$ of $S$, let $T_B$, $T^+_B$ be the hitting
time of $B$, and the return time to $B$, respectively:
$$
\cb{T_B} \,:=\,  \inf \{t\geq 0: x_t \in B\},
\quad \cb{T^+_B}\,:=\,  \inf \{t\geq \tau_1: x_t \in B\},
$$
where $\tau_1$ represents the time of the first jump:
\begin{equation*}
\tau_1 \,=\, \inf \{t\geq 0: x_t \neq x_{0}\}\,.
\end{equation*}

Assuming $|B|\geq 2$, let $(x^B_t)_{t \geq 0}$ denote the trace of the
process $(x_t)_{t \geq 0}$ on $B$ (for details, see \cite{BL10}). This
trace process is an irreducible, $B$-valued Markov chain with jump
rates $\bs r^B = r^B(j, k)$ given by
\begin{equation}
\label{eq:rB}
\cb{r^B(j,k) } \, = \,\lambda_j \,
\bb P_j[T_k = T_B^+]\,, \;\;  j\neq k\in B \;, \quad 
\cb{ r^B(j,j) } \,= \,0\,, 
\end{equation}
where, recall, $\lambda_j$, $j\in S$, represent the holding rates of
the Markov chain induced by the generator $\mc L_S$.  Denote by
$\cb{\mc L^B_S}$ the generator of the $B$-valued Markov chain with
jump rates $r^B(\cdot \, , \cdot)$, and by $\lambda^B(j)$ the holding
times: $\cb{\lambda^B(j) } := \sum_{k\in B\setminus\{j\}} r^B(j,k) $.

Fix a subset $B$ of $S$ with at least two elements.  For each
$i\in B$, let $\cb{u^B_i \colon S\rightarrow \bb [0,1]}$ be the
$\mc L_S$-harmonic extension to $S$ of the indicator function of $i$
on $B$.  In other words, $u^B_i$ is the unique solution to
\begin{equation}
\label{10}
\begin{cases}
u^B_i(j) = \delta_{i,j} & \text{for } j\in B \\
(\mc L_S u^B_i)(j) = 0 & \text{for } j\in S\setminus B.
\end{cases}
\end{equation}
It is well-known that the solution $u^B_k(\cdot)$ has a stochastic
representation given by
\begin{equation}
\label{eq:probrep}
u^B_k(j) \,=\, \bb P_j[T_k = T_B] \,, \quad j\in S\,.
\end{equation}

We turn to the formula \eqref{eq:rB} of the jump rates $r^B(j,k)$.  By
the strong Markov property applied to the identity \eqref{eq:probrep} at
the time of the first jump, for all $i\not\in B$, $k\in B$.
\begin{equation}
\label{24}
\lambda_i\, u^B_k(i) \,=\, \lambda_i\, \bb P_i[T_k = T_B] 
\,=\, \sum_{j\in S} r(i,j)\, \bb P_j [T_k = T_B]
\,=\, \sum_{j\in S} r(i,j)\, u^B_k(j)\,.
\end{equation}
Fix $j$, $k\in B$, $k\neq j$.  By the strong Markov property applied
to the identity \eqref{eq:rB} at the time of the first jump,
\begin{equation}
\label{20}
r^B(j,k) = r(j,k) + \sum_{l\in B^c}
r(j,l) \, \bb P_l[T_k = T_B] = \sum_{l\in S} r(j,l) \, u^B_k(l), \;\;
\text{for}\;\; k\neq j\in B.
\end{equation}
As $u^B_k(j)=0$, we may subtract $u^B_k(j)$ from $u^B_k(l)$ to obtain that
\begin{equation}
\label{01}
r^B(j,k) = (\mc L_S \, u^B_k)(j)\,.
\end{equation}

Similarly, for $k\in B$,
\begin{equation*}
\lambda_k\, \bb P_k[\, T^+_k = T^+_B\,]
\,=\, \sum_{\ell\in B^c} r(k,\ell)\, \bb P_\ell[\, T_k = T_B\,]
\,=\, \sum_{\ell\in B^c} r(k,\ell)\, u^B_k(\ell)\,.
\end{equation*}
On the other hand, by \eqref{eq:rB}, and the definition of
$\lambda^B(k)$,
\begin{equation*}
\lambda_k\, \bb P_k[\, T^+_k = T^+_B\,]
\,=\, \lambda_k\, \big\{\, 1 \,-\, \sum_{j\in B\setminus\{k\}}
\bb P_k[\, T^+_j = T^+_B\,]\,\big\}
\,=\, \lambda_k \,-\, \sum_{j\in B\setminus\{k\}}
r^B(k,j) \,=\,
\lambda_k  \,-\, \lambda^B(k) \,.
\end{equation*}
Therefore, by the two previous identities 
\begin{equation}
\label{25}
\sum_{\ell\in B^c} r(k,\ell)\, u^B_k(\ell) \,=\,
\lambda_k  \,-\, \lambda^B(k) \,.
\end{equation}
Moreover, since $\sum_{k\in B} u^B_k(l) =1$ for all $l\in S$,
summing \eqref{20} over $k\in B\setminus \{j\}$ yields that
\begin{equation*}
\lambda^B(j) = \sum_{k\in B\setminus \{j\}} r^B(j,k) =
\sum_{l\in S} r(j,l)  \sum_{k\in B\setminus \{j\}} u^B_k(l)
=
\sum_{l\in S} r(j,l) \, [ 1 - u^B_j(l) ]\,.
\end{equation*}
As $u^B_j(j)=1$, we conclude that
\begin{equation}
\label{02}
\lambda^B(j) \,=\, -\, (\mc L_S u^B_j)(j)\,.
\end{equation}

\subsection{The projection map}
    
Let $\cb{A = S\setminus B}$.  Define the linear projection map
$\color{blue} \gamma_B:\bb R^S\rightarrow \{ x\in \bb R^S : x_i = 0 \;
\forall i\in A\}$ by
\begin{equation}
\label{eq:proj}
[\gamma_B(x)]_j = u^B_j \cdot x = x_j
+  \sum_{k\in A} u^B_j(k) x_k, \;\; j\in B,
\;\;\text{ and }\;\; [\gamma_B(x)]_i = 0, \;\; i\in A.
\end{equation}
Note that the restriction of $\gamma_B$ to $\Sigma$ maps into
$\Sigma_B$, so we may write
$\gamma_B \colon \Sigma\rightarrow \Sigma_B$.  By the definition of
$\gamma_B$, $\bs v_j$, \eqref{01}, and \eqref{02},
\begin{equation}
\label{12}
\begin{gathered}
[\gamma_B(\bs v_j)]_k = u^B_k \cdot \bs v_j = \mc L_S
u^{B}_k(j) = r^B(j,k),\;\; j\,,\, k\in B\,, \;\; k\neq j.
\\
[\gamma_B(\bs v_j)]_j = u^B_j \cdot \bs v_j =
(\mc L_S u^{B}_j) (j) \,=\,  -\, \lambda^B(j),\;\; j\in B\,.
\end{gathered}
\end{equation}
Thus, the vectors $\gamma_B(\bs v_j)$, $j \in B$, relate to the
generator $\mc L^B_S$ in the same way that the vectors $\bs v_i$,
$i\in S$, relate to the generator $\mc L_S$.

On the other hand, as $u^{B}_k$, $k\in B$, is $\mc L_S$-harmonic on
$A$,
\begin{equation*}
[\gamma_B(\bs v_i)]_k \,=\, u^{B}_k \cdot \bs v_i = \mc L_S
u^{B}_k(i) = 0 \;\;\text{ for all $i\in A = S\setminus B$}\,.
\end{equation*}
In conclusion,
\begin{equation}
\label{eq:projofvect}
\gamma_B(\bs v_j) = \bs v^B_j,\;\; j\in B \;\;
\text{ and } \;\; \gamma_B(\bs v_i) = 0\,,
\;\; i\in A=S\setminus B\, .
\end{equation}

\begin{lemma}
\label{lem:lincomb}
The following properties hold:

\begin{enumerate}[leftmargin=*]
\item[\textup{(1)}] For $\varnothing \neq B\subsetneq S$,
$\{\bs v_i, i\in B\}$ are linearly independent.
\item[\textup{(2)}] For all $x\in \bb R^S$, $\gamma_B(x) - x$ is a
linear combination of $\bs v_k$ for $k\in A=S\setminus B$.
\end{enumerate}
\end{lemma}

\begin{proof}
For the first part, fix a vector $\bs f\in \bb R^S$. By definition of
$\bs v_i$, $\bs v_i \cdot \bs f = 0$ if and only if
$(\mc L_S \bs f)_i =0$. As the chain is irreducible,
$\bs v_i \cdot \bs f = 0$ for all $i\in S$ implies that $\bs f$ is
constant. Hence,
$ \text{dim}(\text{span}\{\bs v_k : k\in S\}) = |S|-1$. Therefore,
together with the fact that
\begin{equation*}
\sum_{k\in S} m_k \, \bs v_k = 0\,,
\end{equation*}
for any $i\in S$ the vectors $\{ \bs v_j,\; j\in S\setminus \{i\}\}$ are
linearly independent. This proves the first part.

We turn to the second assertion of the lemma.
Recall that
$$[\gamma_B(x)]_i = x_i + \sum_{j\in A} u^B_i(j) x_j, \;\;i\in B.$$
Thus, $\gamma_B( \gamma_B(x)) = \gamma_B(x)$ so that 
$$\gamma_B(\gamma_B(x) - x) = \gamma_B(x) - \gamma_B(x) = 0.$$
Therefore, to prove the assertion  we need
to analyze the kernel of $\gamma_B$.
    
We claim that
$\text{ker}(\gamma_B) = \text{span}\{\bs v_k : k\in A\}$.  On the one
hand, by \eqref{eq:projofvect},
$\text{span}\{\bs v_k : k\in A\} \subset \text{ker}(\gamma_B)$, and by
the first part of the lemma,
$\text{dim\,} ( \text{span}\{\bs v_k : k\in A\} ) = |A|$.  On the
other hand, since $\gamma_B$ preserves
$\bb R^B \times \{\bs 0\} \subset \bb R^S$,
$\text{dim}(\text{Im}(\gamma_B))\geq |B|$.  Therefore,
$\text{dim}(\text{ker}(\gamma_B)) \leq |S|-|B| = |A|$, so that
$$\text{span}\{\bs v_k : k\in A\} = \text{ker}(\gamma_B)\,,$$
which completes the proof of the lemma.
\end{proof}

The next result  is a direct consequence of Lemma \ref{lem:lincomb}.

\begin{lemma}
\label{lem:RAlinmap}
Fix $\varnothing \neq B \subsetneq S$. Let $A = S \setminus B$. For
$x\in \bb R^S$, let $\cb{x_A}$ be the cannonical projection of $x$ to
$\bb R^A$ defined by $[x_A]_i = x_i$ for $i\in A$. Then, there exists
a linear map $L_A : \bb R^A \to \bb R^A$ such that
\begin{equation}
\label{eq:RAlinmap}
\gamma_B(x) - x = \sum_{i\in A} [L_A(x_A)]_i \bs v_i.
\end{equation}
\end{lemma}

\begin{proof}
From Lemma \ref{lem:lincomb}(2), there exists a linear map
$L: \bb R^S \to \bb R^A$ such that
$$
\gamma_B(x) - x = \sum_{i\in A} [L(x)]_i \, \bs v_i\,.
$$
Denote by $\pi_A: \bb R^S \to \bb R^A$ the canonical projection
defined by $\pi_A(x) = x_A$, we need to show that there exists a
factorization map $L_A: \bb R^A \to \bb R^A$ such that
$L = L_A \circ \pi_A$. This is equivalent to
$$
\text{ker} \, \pi_A \subset \text{ker} L\,.
$$
It is clear that $\text{ker} \, \pi_A = \bb R^B \times \{\bs 0\}$.  By
definition, $\gamma_B$ preserves $\bb R^B \times \{\bs 0\}$, which
implies that $\bb R^B \times \{\bs 0\} \subset \text{ker} L$
because, by Lemma \ref{lem:lincomb}.(1), the vectors $\{\bs v_i, i\in A\}$ are
linearly independent.
\end{proof}

An important property of the map $\gamma_B$ is that the function maps
an interior of a subsimplex of $\Sigma$ into the one of
$\Sigma_B$. This is stated as follows:

\begin{lemma}
\label{lem:splxrest}
Let $B, C \subset S$. For each $i\in B$, either $[\gamma_B(x)]_i=0$
for all $x\in \mathring{\Sigma}_C$ or $[\gamma_B(x)]_i\neq 0$ for all
$x\in \mathring{\Sigma}_C$.  Therefore, there exists $D\subset B$ such
that $\gamma_B(\mathring{\Sigma}_C) \subset \mathring{\Sigma}_D$.
\end{lemma}

\begin{proof}
Let $ A = S \setminus B$.  By \eqref{eq:proj}, for every $i\in B$, 
$x\in \mathring{\Sigma}_C$,
\begin{equation}
\label{21}
[\gamma_B(x)]_i = x_i + \sum_{k\in A} u_i^B(k) \, x_k
= x_i + \sum_{k\in A \cap C} u_i^B(k) \, x_k \,.
\end{equation}
Since $x_j>0$ for all $j\in C$ and $u_\ell^B(k)\ge 0$ for
all $\ell\in B$, $k\in S$,
$[\gamma_B(x)]_i=0$ if, and only if, $x_i=0$ (that is $i\not \in C$)
and $u_i^B(k)=0$ for all $k\in A \cap C$. This condition does not
depend on the point $x\in \mathring{\Sigma}_C$, but only on the
indices $i$ and  $u^B_i(k)$. Therefore, either it holds for all points in
$\mathring{\Sigma}_C$ or it holds for none. This is the first
assertion of the lemma.

Let
\begin{gather*}
G\,=\, \{i\in B : u_i^B(k)=0 \text{ for all } k\in A \cap C \}  \,,
\\
D  \,=\,  B \setminus
\big\{\, [B\setminus C] \,\cap \, G \,\big\}
\,=  \,  [ B \cap C ]\,\cup\, [ B \setminus G ]\,,
\end{gather*}
so that $D\subset B$.

We claim that
$\gamma_B(\mathring{\Sigma}_C) \subset \mathring{\Sigma}_D$.  Fix
$x\in \mathring{\Sigma}_C$. To prove the assertion, we have to show
that $[\gamma_B(x)]_i >0$ if $i\in D$ and $[\gamma_B(x)]_i =0$
otherwise. 

Consider first the case $i\in D$. If $i\in B\cap C$, then $x_i>0$
because $x\in \mathring{\Sigma}_C$. Thus, by \eqref{21},
$[\gamma_B(x)]_i >0$. If $i\in B \setminus G$, then, $u_i^B(k)>0$ for
some $k\in A \cap C$. As $x\in \mathring{\Sigma}_C$, $x_k>0$.  Thus,
by \eqref{21}, $[\gamma_B(x)]_i \ge u_i^B(k)\, x_k >0$. In conclusion,
$[\gamma_B(x)]_i >0$ for all $i\in D$.

Fix $i\not\in D$. If $i\not\in B$, by \eqref{eq:proj},
$[\gamma_B(x)]_i =0$. Suppose that $i\in B\setminus D$, so that
$i\in (B \cap G)\setminus C$. As $i\not\in C$ and
$x\in \mathring{\Sigma}_C$, $x_i=0$. Since $i\in G$, $u_i^B(k)=0$ for
all $k\in A \cap C$. Thus, by \eqref{21}, $[\gamma_B(x)]_i =0$, as
claimed. This completes the proof of the lemma.  
\end{proof}

We conclude this section by presenting a formula for the composition
of projection maps associated with distinct simplices. This identity
plays a central role in establishing the absorbing structure of the
process, as stated in Proposition  \ref{prop:restriction}.

\begin{lemma}
\label{lem:recursive}
Fix $\varnothing \neq B \subset C \subset S$. Then,
$\gamma_B \circ \gamma_C = \gamma_B$.
\end{lemma}
    
\begin{proof}
Fix $x\in \Sigma$, and recall that $A=S\setminus B$. By
\eqref{eq:proj}, for $j\in B$,
\begin{align*}
[\gamma_B(\gamma_C(x))]_j
& = [\gamma_C(x)]_j + \sum_{k\in A\cap C} u^B_j(k)\, [\gamma_C(x)]_k
\\
&= x_j + \sum_{i\in S\setminus C} u^C_j(i) \,x_i
+ \sum_{k\in A\cap C} u^B_j(k)\,
\Big( x_k + \sum_{i\in S\setminus C} u^C_k(i) \, x_i \Big)\,.
\end{align*}
We restricted the second sum in the first line to $k\in A\cap C$
because $[\gamma_C(x)]_k =0$ for $k\not\in C$.
    
On the other hand,
$$
[\gamma_B(x)]_j = x_j + \sum_{k\in A} u^B_j(k) \, x_k\,.
$$
Hence, deleting the common terms, to complete the proof it remains to
show that
$$
\sum_{i \in S\setminus C} u^C_j(i) \, x_i + \sum_{k\in A\cap C}
u^B_j(k) \sum_{i\in S\setminus C} u^C_k(i) \, x_i = \sum_{k\in
A\setminus C}u^B_j(k) \, x_k\,.
$$
Since $B \subset C$, in the last sum we may rewrite $A\setminus C$ as
$S\setminus C$. Comparing the coefficient of
$x_i, \; i\in S\setminus C$, the equation is equivalent to
$$
u^C_j(i) + \sum_{k \in A\cap C} u^B_j(k) \, u^C_k(i) = u^B_j(i)\,.
$$
By \eqref{eq:probrep}, this identity can be rewritten as
$$
\bb P_i[T_B=T_j] = \bb P_i[T_C = T_j] + \sum_{k\in A\cap C} \bb
P_i[T_C=T_k] \, \bb P_k[T_B = T_j]\,.
$$
which can be directly verified probabilistically.
    
We just proved that $[\gamma_B(\gamma_C(x))]_j = [\gamma_B(x)]_j$ for
all $x\in \Sigma$, $j\in B$. By \eqref{eq:proj}, this identity holds
trivially for $j\not \in B$. This completes the proof of the lemma.
\end{proof}

\section{Extension of the martingale problem}
\label{sec:edmp}

In this section, we introduce a relaxed version of the martingale
problem. This modification is motivated by the need to construct an
appropriate superharmonic function, which is essential for
establishing the absorbing property and, in turn, for proving the
uniqueness of the solution to the martingale problem.

Unlike in \cite{BJL}, the construction of the superharmonic function
in Section~\ref{sec:proofthmabs} cannot be carried out using only
functions from the original domain $\mc D_S$ of the martingale
problem. Consequently, it becomes necessary to enlarge the domain to
include functions with weaker regularity. The purpose of this section
is to develop this extended framework. 

To motivate this extension, we begin by examining the operator
$\mf {L}$ from a different perspective.  Recall from \eqref{18} and
\eqref{eq:LB} the definition of the domains $\mc D^{B}_A$ for
$\varnothing \subsetneq A \subset B \subset S$, $|B|\ge 2$ and the
operator $\mf L^B$. The next result provides an alternative formula
for the value of $\mf L F$ on the set $\Sigma_B$ in terms of the
generator $\mf L^B$.

\begin{proposition}
\label{prop:domrest}
Fix $F \in \mc D_S$. Then, for every subset $B$ of $S$ with at least
two elements, the restriction of $F$ to $\Sigma_B$, denoted by
$\cb{F|_{\Sigma_B}}$, belongs to the domain $\mc D^B_B$.  Moreover,
$(\mf L F) |_{\Sigma_B} = \mf L^B (F|_{\Sigma_B})$.
\end{proposition}

To prove this second-order identity, we first establish a few auxiliary
lemmata controlling the second-order derivatives of $F$.

\begin{lemma}
\label{l31}
For any $k\in S\setminus B$, $F\in \mc D_k$, $\bs w \in \Sigma_B$,
$$
\nabla_{\bs e_k - \bs w}
\nabla_{\bs v_k} F = 0 \;\; \text{ \rm on } \;\;\Sigma_B\,.
$$
\end{lemma}

\begin{proof}
Fix $k\in S\setminus B$, $F\in \mc D_k$.  As $F\in \mc D_k$,
$\nabla_{\bs b_k} F$ is a continuous function. Therefore,
$(\nabla_{\bs v_k} F)(x) = (x_k/m_k)\, H(x)$ for some continuous
function $H\in C(\Sigma)$ which vanishes if $x_k=0$. As $k\not\in B$,
$H|_{\Sigma_B} = 0$.

For $x\in \Sigma_B$, since $(\nabla_{\bs v_k} F)(x)=0$, 
\begin{align*}
(\nabla_{\bs e_k - \bs w} \nabla_{\bs v_k} F)(x)
\,=\, \lim_{t\to 0}  \frac{ (\nabla_{\bs v_k}F)
(x + t(\bs e_k - \bs w)) }{t}\,.
\end{align*}
As $\bs w \in \Sigma_B$, $[\bs w]_k =0$ so that
$[\bs e_k - \bs w]_k=1$. Hence, By definition of
$H$, as $H$ is continuous and $x_k=0$, this expression is equal to
\begin{align*}
\lim_{t\to 0} \frac{1}{m_k}\,
H(x + t(\bs e_k - \bs w)) \,=\,
\frac{1}{m_k}\,  H(x) \,=\, 0\,,
\end{align*}
as claimed.
\end{proof}

The following lemma describes how the trace process and the projection
map are related through the Hessian term.

\begin{lemma}
\label{l32}
For any $F \in C^2(\bb R^S)$, 
\begin{equation} \label{eq:l32}
\sum_{i, j\in S} a_{ij} \nabla_{\bs e_i} \nabla_{\gamma_B(\bs e_j)} F
= \frac{1}{2}\sum_{i,j\in B} m_i \, r^B(i,j)
\, (\partial_{x_i}- \partial_{x_j})^2 F\,.
\end{equation}
\end{lemma}

\begin{proof}
By the definition \eqref{eq:proj} of $\gamma_B(\bs e_j)$,
\begin{equation}
\label{26}
\begin{aligned}
\sum_{i, j\in S} a_{ij} \nabla_{\bs e_i} \nabla_{\gamma_B(\bs e_j)} F
& = \sum_{i, j\in S}\sum_{k\in B} a_{ij} \,  \nabla_{\bs e_i} \nabla_{\bs e_k} F
\Big\{ \delta_{j,k} + \sum_{\ell\in A} u^B_k(\ell )\,
\delta_{j,\ell}\, \Big\}
\\
& = \sum_{i\in S}\sum_{k\in B} a_{ik}\,  \nabla_{\bs e_i} \nabla_{\bs e_k} F
\,+\, 
\sum_{i\in S}\sum_{k\in B} \nabla_{\bs e_i} \nabla_{\bs e_k} F \sum_{\ell\in A}
a_{i\ell} \,  u^B_k(\ell ) \,,
\end{aligned}
\end{equation}
where $ A=S\setminus B$,

If $i\in A$, by the definition \eqref{22} of $a_{i\ell}$, and
\eqref{24},
\begin{equation*}
\sum_{\ell\in A} a_{i\ell} \,  u^B_k(\ell )\,=\, -\, a_{ik}\,.
\end{equation*}
The right-hand side of \eqref{26} is thus equal to
\begin{equation*}
\sum_{i, k\in B} a_{ik} \,  \nabla_{\bs e_i} \nabla_{\bs e_k} F
\,+\, 
\sum_{i,k\in B} \nabla_{\bs e_i} \nabla_{\bs e_k} F \sum_{\ell\in A}
a_{i\ell}  \,  u^B_k(\ell ) \,.
\end{equation*}
On the other hand, by the first identity in \eqref{20} and \eqref{25},
if $i\in B$, $i\neq k$, 
\begin{equation*}
\sum_{\ell\in A} a_{i\ell} \,  u^B_k(\ell )\,=\, -\, a_{ik}\,-\,
m_i\, r^B(i,k) \quad\text{and}\quad
\sum_{\ell\in A} a_{k\ell} \,  u^B_k(\ell )\,=\, -\,
m_{k}\,[\, \lambda_k - \lambda^B(k)\,]\,.
\end{equation*}
The right-hand side of \eqref{26} can be further simplified to
\begin{equation*}
\sum_{k\in B} m_k\, \lambda^B(k)\, \nabla^2_{\bs e_k} F
\,-\, 
\sum_{i\in B} \sum_{k\in B\setminus\{i\}} m_i\, r^B(i,k) \, 
\nabla_{\bs e_i} \nabla_{\bs e_k} F \,.
\end{equation*}
This completes the proof of the lemma.
\end{proof}

A direct use of Lemmata \ref{l31} and \ref{l32} does not yield full
control of \cp{$D_{\tilde{a}}$}, since the left-hand side of
\eqref{eq:l32} is not symmetric in $i$ and $j$.  In effect, we need to
swap the indices, and this is achieved by introducing the adjoint
generators $\mc L_S^\dagger$ and $(\mc L_S^B)^\dagger$.

Let $r^\dagger $ be the adjoint of $r$ with respect to
stationary measure $m$, explicitly given by
$$
\cb{r^\dagger (i,j)} \,:=\, m_j\, r(j,i) / m_i\,, \quad i\neq j\in S\,.
$$
Denote by $\cb{\mc L^\dagger_S}$ the generator of the Markov chain
associated to the jump rates $r^\dagger (i,j)$, and let
\begin{equation*}
\cb{\bs v^\dagger_i} \,: =\,
\sum_{k\in S} r^\dagger(i,k) \, (\bs e_k - \bs e_i)\,, \quad
\cb{a^\dagger_{ij}}
\,:=\, -\, m_i \, \bs v^\dagger_i \cdot \bs e_j, \;\; i,j\in S\,,
\end{equation*}
so that
\begin{equation*}
a^\dagger_{ij} \,=\, -\, m_i \, r^\dagger(i,j) \;\;\text{for}\;\;
i\neq j\in S\,,\;\;\text{and}\;\;
a^\dagger_{ii}  \,=\, m_i \, \lambda^\dagger(i)\,,\;\;
\text{where}\;\;
\cb{\lambda^\dagger(i)}\,:=\, \sum_{j\neq i}
r^\dagger(i,j)  \,.
\end{equation*}

For a subset $B$ of $S$ with at least two elements, denote by
$\{ \cb{r^{B,\dagger} (j,k)} : j,k\in B\}$ the jump rates of the trace
on $B$ of the Markov chain with generator $\mc L^\dagger_S$.  As in
\eqref{10}, \eqref{eq:proj}, introduce the equilibrium potentials
$\cb{u^{B,\dagger}_i}\colon S\to [0,1]$ and the projection maps
$\cb{\gamma_B^\dagger}\colon\Sigma\rightarrow \Sigma_B$ replacing the
generator $\mc L_S$ by its adjoint $\mc L^\dagger_S$.

\begin{corollary}
\label{l33}
For any $F \in C^2(\bb R^S)$, 
$$
\sum_{i, j\in S} a_{ij} \nabla_{\bs e_j}
\nabla_{\gamma_B^\dagger(\bs e_i)} F
\,=\, \frac{1}{2}\sum_{i,j\in B} m_i \, r^{B,\dagger} (i,j)
(\partial_{x_i} - \partial_{x_j})^2 F.
$$
\end{corollary}

\begin{proof}
Since  $a^\dagger_{ij} = a_{ji}$, 
\begin{align*}
\sum_{i, j\in S} a_{ij} \nabla_{\bs e_j}
\nabla_{\gamma_B^\dagger (\bs e_i)} F
= \sum_{i, j\in S} a^\dagger_{ij} \nabla_{\bs e_i}
\nabla_{\gamma_B^\dagger(\bs e_j)} F \,.
\end{align*}
By the previous lemma with the adjoint rates in place of the rates,
this expression is equal to
\begin{equation*}
\frac{1}{2}
\sum_{i,j\in B} m_i \,  r^{B,\dagger} (i,j) \, (\partial_{x_i}- \partial_{x_j})^2 F\,,
\end{equation*}
as claimed.
\end{proof}

Since $m_i \, r^{B,\dagger} (i,j) = m_j \, r (j,i) $, the right-hand
side of Lemma \ref{l32} and Corollary \ref{l33} are the same. This is
because the adjoint of the trace process is the trace of the adjoint
process.

\begin{proof} [Proof of Proposition \ref{prop:domrest}]
We first claim that $F|_{\Sigma_B} \in \mc D^B_B$. Recall that
$A=S\setminus B$. Fix $i\in B$ and $x^n\in \mathring{\Sigma}_{B}$ with
$(x^n)_i \to 0$.  It is enough to show that
$$ \frac{\nabla_{\bs v^B_i} F(x^n)}{(x^n)_i} \to 0\,. $$
By \eqref{eq:projofvect} and as $F$ belongs to $\mc D_S$, 
$$ \nabla_{\bs v^B_i} F(x^n) = \nabla_{\gamma_B(\bs v_i)} F(x^n)
= \nabla_{\bs v_i} F(x^n) - \nabla_{\bs v_i - \gamma_B(\bs v_i)}
F(x^n)\,.
$$
By definition of $\mc D_S$, $(\nabla_{\bs v_j} F)(y) = 0$ for
$y\in \Sigma_{B}$, $j\in A$. Thus, by Lemma \ref{lem:lincomb},
$$ \nabla_{\bs v_i - \gamma_B(\bs v_i)} F(x^n) = 0\,, $$
so that
$$
\lim_{n\to\infty} \frac{\nabla_{\bs v^B_i} F(x^n)}{(x^n)_i} \,=\,
\lim_{n\to\infty} \frac{\nabla_{\bs v_i} F(x^n)}{(x^n)_i} \,=\, 0$$
because $F\in \mc D_S$.  This completes the proof of the first part of
the proposition.

It remains to show that
$\mf L^B (F|_{\Sigma_B}) = (\mf L F)|_{\Sigma_B}$.  Fix
$x\in \Sigma_B$. Let $\tilde{A} = \{i\in S : x_i = 0\}$ and
$\tilde{B} = S\setminus \tilde{A}$. By definition,
$A\subset \tilde{A}$ and $\tilde{B} \subset B$.

We start from the definition of $\mf L$:
\begin{align*}
\mf L F(x) &= \nabla_{\bs b} F(x) + \frac{1}{2} \sum_{i,j} m_i \, r(i,j)\,
[\nabla_{\bs e_i - \bs e_j} (\nabla_{\bs e_i - \bs e_j} F)](x).
\end{align*}
By definition of the operator $\nabla_{\bs b} $, the first term is
equal to
\begin{align*}
\nabla_{\bs b} F(x) &= \sum_{j\in \tilde{B}} \nabla_{\bs b_j} F(x)
\,=\, b \,  \sum_{j\in \tilde{B}} m_j \,
\frac{\nabla_{\bs v_j}F(x)}{x_j}
\,=\, b \,  \sum_{j\in \tilde{B}} m_j
\frac{\nabla_{\bs v_j - \gamma_B(\bs v_j)}F(x)}{x_j}
\,+\, b \,  \sum_{j\in \tilde{B}} m_j \frac{\nabla_{\gamma_B(\bs v_j)}F(x)}{x_j}\,\cdot
\end{align*}
As in the first part of the proof, by definition of $\mc D_S$,
$\nabla_{\bs v_i} F(y) = 0$ for $i\in A$, $y\in \Sigma_{B}$. Thus,
since $x\in \Sigma_{B}$, by Lemma
\ref{lem:lincomb}, the first term on the right-hand side vanishes, so
that 
\begin{equation}
\label{eq:SB1}
\nabla_{\bs b} F(x) \,=\, b \, 
\sum_{j\in \tilde{B}} m_j \frac{\nabla_{\gamma_B(\bs v_j)}F(x)}{x_j}
\,=\, \nabla_{\bs b^B} F(x)\,.
\end{equation}

We turn to the second term. Let $\tilde{F}\in C^2(\bb R^S)$ be an
extension of $F$.  By \eqref{eq:ba} and Corollary \ref{l33}, the
second term is equal to
\begin{align*}
\sum_{i,j\in S} a_{ij} \partial_{x_i} \partial_{x_j} \tilde{F}(x)
&= \sum_{i,j\in S} a_{ij} \nabla_{\bs e_j} \nabla_{\bs e_i}
\tilde{F}(x)
\\
&= \frac{1}{2}\sum_{i,j\in B} m_i \, r^{B,\dagger} (i,j)\,
(\partial_{x_i} - \partial_{x_j})^2 \tilde{F}(x)
\,+\, \sum_{i,j\in S} a_{ij} \nabla_{\bs e_j}
\nabla_{\bs e_i - \gamma^\dagger_B(\bs e_i)} \tilde{F}(x)\,. 
\end{align*}
Since $\sum_j a_{ij} \bs e_j = -m_i\, \bs v_i$ and
$\gamma^\dagger_B(\bs e_i) = \bs e_i$ for $i\in B$, the second term
is equal to
\begin{align*}
- \, \sum_{i\in S}  m_i \, \nabla_{\bs v_i}\, 
\nabla_{\bs e_i - \gamma^\dagger_B(\bs e_i)} \, \tilde{F}(x)
\,=\,
- \, \sum_{i\in A}  m_i \, \nabla_{\bs v_i}\, 
\nabla_{\bs e_i - \gamma^\dagger_B(\bs e_i)} \, \tilde{F}(x)
\end{align*}
Since $x\in \Sigma_B$, by Lemma \ref{l31}, this expression vanishes.

Therefore,
\begin{equation}
\label{eq:SB2}
\sum_{i,j\in S} a_{ij} \, \partial_{x_i} \,
\partial_{x_j}  \tilde{F}(x)
\,=\, \frac{1}{2}\sum_{i,j\in B} m_i\,
r^{B, \dagger} (i,j) \, (\partial_{x_i}- \partial_{x_j})^2 \tilde{F}(x)\,.
\end{equation}
As $m_i\, r^{B, \dagger} (i,j) = m_j\, r^{B} (j,i) $, combining
\eqref{eq:SB1} and \eqref{eq:SB2} yields that
\begin{equation}
\label{eq:SB}
\mf L F(x) = \mf L^B F(x)\,,
\end{equation}
which  completes the proof of the proposition.
\end{proof}

In light of Proposition \ref{prop:domrest}, we redefine the operator
$\mf L$ on the domain $\mc D_S$, then further extend the domain of the
operator.  To do so, we need the space $\cb{C_{\rm pc}(\Sigma)}$ (pc
for piecewise), the space of functions $F\colon \Sigma \to \bb R$ such
that, for each subset $B$ of $S$ with at least two elements, the
function
$F|_{\mathring{\Sigma}_B} \colon \mathring{\Sigma}_B \to \bb R$ is
continuous. To differentiate the new operator (with extended domain)
from the original one, represented by $\mf L$, we denote it by
$\mf L^{\mc E}$.

\begin{definition}
\label{def:ed}
For $F\in C(\Sigma)$ satisfying
$F|_{\mathring{\Sigma}_C}\in C^2(\mathring{\Sigma}_C)$ for all
nonempty $C\subset S$, $|C|\geq 2$, define the operator
$\mf L^{\mc E}$ as follows: For $x\in \Sigma$, let
$B = \{ i \in S : x_i \neq 0 \}$. Then
$\mf L^{\mc E} F \in C_{\rm pc}(\Sigma)$ is defined by
\begin{equation}
\label{eq:eop}
\mf L^{\mc E} F(x) =\begin{cases}
[\mf L^B ( F|_{\mathring{\Sigma}_B} )] (x),
& \text{ if } B = \{ i \in S : x_i \neq 0 \},\; |B|\geq 2 \\
0, & \text{ otherwise.}
\end{cases}
\end{equation}
Here, $\mf L^B$ is the operator defined in \eqref{eq:LB}, considered
as a local second order differential operator on
$\mathring{\Sigma}_B$, as mentioned in Remark \ref{rem:sodo}.
\end{definition}

Note that $\mf L^{\mc E}$ equals $\mf L$ on the domain $\mc D_S$ by
Proposition \ref{prop:domrest}. The functions introduced in Definition
\ref{def:ed} can have pathological behavior near the boundary of the
domain, as the value of $\mf L^{\mc E} F$ may diverge as $x$
approaches the boundary of $\Sigma_B$, causing the martingale problem
to be ill-defined.  For this reason, it is necessary to introduce
further conditions.

\smallskip\noindent
\textbf{Condition $\mathfrak{E1}$:} For each $i\in S$,
$F\in C^1(\Sigma)$ satisfies condition $\mathfrak{E1}(i)$ if the map
\begin{equation}
\label{eq:condC1}
x \mapsto \frac{1}{x_i} \nabla_{\bs v_i} F(x),
\end{equation}
is bounded on $\{x\in \Sigma : x_i>0\}$. If $F$ satisfies condition
$\mathfrak{E1}(i)$ for all $i\in S$, then we say that $F$ satisfies
condition $\mathfrak{E1}$.

Similar to the definition in \eqref{eq:C2}, for $A\subset S$ with
$|A|\geq 2$, let
\begin{equation*}
\cb{ C^2_b(\mathring{\Sigma}_A)} \,:=\,
\big\{\, f \in C^2(\mathring{\Sigma}_A) :
f \text{ has bounded
second derivatives} \,\big\} \,.
\end{equation*}
The expression``$f$ has bounded second derivatives'' means that for
all tangent vectors $\bs V$ and $\bs W$ in $T_{\Sigma_A}$, the second
derivative $\nabla_{\bs V} (\nabla_{\bs W} f)$ is bounded on
$\mathring{\Sigma}_A$.  In other words,
there exists a finite constant $C_0>0$ such that
\begin{equation}
\label{27}
\big| \, [ \nabla_{\bs V} (\nabla_{\bs W} f) ] (x) \, \big| \,\le\,
C_0 \, \Vert \bs V\,\Vert\, \Vert \bs W\,\Vert
\end{equation}
for all $\bs V, \bs W \in T_{\Sigma_A}$ and
$x\in \mathring{\Sigma}_A$.

\smallskip\noindent \textbf{Condition $\mathfrak{E2}$:} We say that
$F\in C^1(\Sigma)$ satisfies condition $\mathfrak{E2}(A)$, 
$A \subset S$ with $|A|\geq 2$, if
$F|_{\mathring{\Sigma}_A}\in C^2_b(\mathring{\Sigma}_A)$.  If $F$
satisfies condition $\mathfrak{E2}(A)$ for all $A\subset S$ with
$|A|\geq 2$, then we say that $F$ satisfies condition $\mathfrak{E2}$.

\begin{remark}
\label{rem:C1}
By Lemma \ref{app:6}, a function $F\in C(\Sigma)$ such that
$F|_{\mathring{\Sigma}}\in C^2_b(\mathring{\Sigma})$ belongs to
$C^1(\Sigma)$, and, for any $B\subset S$ with $|B|\geq 2$,
$F|_{\Sigma_B}$ belongs to $C^1(\Sigma_B)$.
\end{remark}

\begin{definition}
For each $i\in S$, denote by $\mc E_i$ the collection of functions
$F\in C^1(\Sigma)$ satisifying condition
$\mathfrak{E1}(i)$ and condition $\mathfrak{E2}$.  In addition, let
\begin{equation*}
\cb{\mc E_A} \, :=\, \bigcap_{i\in A} \mc E_i \quad
\textrm{ for each nonempty $A\subseteq S$.}
\end{equation*}
\end{definition}

The next result is a consequence of Proposition \ref{prop:domrest}.

\begin{proposition}
\label{prop:naturalextend}
It holds that $\mc D_S\subset \mc E_S$. Moreover,  $\mf L^{\mc E} F =
\mf L F$ for all $F\in \mc D_S$.
\end{proposition}

\begin{proof}
Fix $F\in \mc D_S$. By Definition \ref{def:dom}, $F\in C^1(\Sigma)$,
and $F$ satisfies condition $\mf E_1$. Fix a subset $A$ of $S$ with at
least two elements. By the same definition,
$F|_{\mathring{\Sigma}_A} \in C^2_b(\mathring{\Sigma}_A)$. This proves
that $F$ belongs to $\mc E_S$.  On the other hand, by \eqref{eq:eop}
and Proposition \ref{prop:domrest}, $\mf L^{\mc E} F = \mf L F$, which
completes the proof of the proposition..
\end{proof}



We finally define the extended martingale problem.  Let
$\cb{C_{\rm pc}^b(\Sigma)}$ be the space of bounded functions in
$C_{\rm pc}(\Sigma)$.  It is clear that the operator $\mf L^{\mc E}$
maps $\mc E_S$ to $C_{\rm pc}^b(\Sigma)$.  Since
$C_{\rm pc}^b(\Sigma)$ is a subset of bounded Borel functions, we can
consider the following martingale problem.
\begin{definition} A probability measure $\bb P$ on
$C(\bb R_+, \Sigma)$ is a solution for the
$(\mf L^{\mc E}, \mc E_S)$-martingale problem if, for any
$H\in \mc E_S$,
\begin{equation*}
\mf L^{\mc E} H(X_{t}) - \int_{0}^{t} (\mf L^{\mc E} H)(X_s) \, ds\;,\quad t\ge 0
\end{equation*}
is a $\bb P$-martingale with respect to the filtration
$(\ms F_t)_{t\ge 0}$, the same as in Definition \ref{f06}.
\end{definition}

In the next section, we prove the following theorem.

\begin{thm}
\label{thm:emp}
For each $x\in \Sigma$, denote by $\bb P_{x}$ a probability measure on
$D(\bb R_+, \Sigma)$ which starts at $x$ and is a solution of the
$(\mf L,\mc D_S)$-martingale problem.  Then $\bb P_{x}$ also solves
the $(\mf L^{\mc E}, \mc E_S)$-martingale problem.
\end{thm}

\subsection{Extension map}
\label{sec:extmap}

Since our extended domain is defined by combining information from
subsimplices, we may want to extend functions defined on subsimplices
to the entire simplex $\Sigma$ in a way that the extended function
belongs to the domain of the generator $\mf L^{\ms E}$. This approach
plays a crucial role in Sections \ref{sec:proofthmemp} and
\ref{sec:proofthmabs}.

Given a function $f\colon \Sigma_B \to \bb R$, define the function
$ \gamma^*_B f\colon \Sigma \to \bb R$ by
\begin{equation}
\label{eq:pullback}
\cb{(\gamma^*_B f)(x)} \,:=\,  f(\gamma_B(x))\,, \;\; x\in \Sigma\,. 
\end{equation}

\begin{lemma}
\label{lem:extregularity}
Suppose that $F$ satisfies condition $\mathfrak{E2}$.  Fix
$\varnothing \neq B\subsetneq S$, and let $F_B$ be the restriction of
$F$ to $\Sigma_B$: $\cb{F_B = F|_{\Sigma_B}}$. Then, $\gamma_B^*F_B$
belongs to $\mc E_A$, where $A = S \setminus B$.
\end{lemma}

\begin{proof}
We need to check condition $\mathfrak{E1}(A)$ and condition
$\mathfrak{E2}$. First, note that $F \in C^1(\Sigma)$.
By Remark \ref{rem:C1}, $F_B = F|_{\Sigma_B} \in C^1(\Sigma_B)$.

\noindent \textbf{(1) $\gamma_B^*F_B$ satisfies condition
$\mathfrak{E1}(A)$.}
    
\noindent
Fix $i\in A$. By Lemma \ref{app:4}, as $F_B \in C^1(\Sigma_B)$,
$$
\frac{1}{x_i} \, [\nabla_{\bs v_i} (\gamma_B^*F_B) ] (x) \,=\,
\frac{1}{x_i}\, 
(\nabla_{\gamma_B (\bs v_i)} F_B)  (\gamma_B (x)) \;\;  \text{ on the
set} \;\; \{x\in \Sigma : x_i >0\} \,.
$$
By \eqref{eq:projofvect}, as $i\in A$, $\gamma_B (\bs v_i)=0$. The
right-hand side thus vanishes. In particular, it is bounded as
required.

\noindent \textbf{(2) $\gamma_B^*F_B$ satisfies Condition
$\mathfrak{E2}$.}

\noindent
We need to check condition $\mathfrak{E2}(C)$ for all sets
$C \subset S$, $|C|\ge 2$. As $F$ belongs to $\mc E_A$,
it satisfies condition $\mf E\mf 2 (B)$. Hence,
$F|_{\mathring{\Sigma}_B}\in C^2_b(\mathring{\Sigma}_B)$, so that
$F_B|_{\mathring{\Sigma}_B} \in C^2_b(\mathring{\Sigma}_B)$.

On the other hand, by Lemma \ref{lem:splxrest},
$\gamma_B(\mathring{\Sigma}_C) \subset \mathring{\Sigma}_D$ for some
$D\subset B$.  Thus $(\gamma_B^*F_B)|_{\mathring{\Sigma}_C}$ belongs
to $C^2_b(\mathring{\Sigma}_C)$, as claimed.
\end{proof}

The next Lemma states  that some particular directional
derivatives of the extension map vanish. 

\begin{lemma}
\label{lem:resvan}
Fix $\varnothing \neq B \subset C \subset S$ and
$G \in C^1(\Sigma_B)$. Then,
$[\, \nabla_{\bs v^C_i} (\gamma_B^* G)] (x) = 0$ for all
$i\in C\setminus B$, $x\in \mathring{\Sigma}$.
\end{lemma}

\begin{proof}
Fix $i\in C$. By Lemma \ref{app:4}, and since, by \eqref{eq:projofvect},
$\bs v^C_i = \gamma_C (\bs v_i)$,
\begin{equation*}
[ \, \nabla_{\bs v^C_i} \, (\gamma_B^*G)\,] \,  (x)\,=\,
(\nabla_{\gamma_B (\bs v^C_i)} G) \, (\gamma_B (x))
\,=\, (\nabla_{\gamma_B \circ \gamma_C (\bs v_i)} G) \,
(\gamma_B (x))
\end{equation*}
for all $x\in \mathring{\Sigma}$.  By Lemma \ref{lem:recursive}, this
expression is equal to
$(\nabla_{\gamma_B (\bs v_i)} G)\, (\gamma_B (x))$. As $i\not\in B$,
by \eqref{eq:proj}, $\gamma_B (\bs v_i) =0$, which completes the proof
of the lemma.
\end{proof}

\section{A norm map and related Lemmas}
\label{sec:aux}

This section is devoted to the construction of maps
$J_A\colon \bb R^A_+\to \bb R_+$, $A$ a nonempty subset of $S$, which
mimic the norm of $\bb R^A$ and whose directional derivatives at the
boundary of the simplex $\Sigma$ behave nicely (see Lemmata
\ref{lem:normmapbdd}, \ref{lem:normmapvanish}). These functions play a
crucial role in Section \ref{sec:proofthmemp} to approximate functions
in $\mc E_S$ by functions in $\mc D_S$ (Lemma \ref{lem:FFe}) and in
the construction of a superharmonic function in the domain $\mc E_S$,
see Lemma \ref{lem:auxfunc}.

For each nonempty $A\subseteq S$, endow $\bb R^A$ with the norm
$\color{blue} \Vert x \Vert_A := \sum_{i\in A} |x_i|$ and consider the
cone $\bb R_+^A:=\{x\in \bb R^A : x_i\ge 0, \forall i\in A\}$ with
vertex ${\bs 0}\in \bb R^A$.

\begin{lemma}
\label{johp}
For each nonempty subset $A$ of $S$ there exists a function
$J_A \colon \bb R_+^A \to [0,\infty)$ such that
\begin{enumerate} [leftmargin=*]
\item[$a)$] $J_A(\alpha x)= \alpha \,J_A(x)$, for all $\alpha\ge 0$
and $x\in \bb R_+^A$.
\item[$b)$] There exist constants $0<c_1\le c_2<\infty$ so that
\begin{equation*}
c_1 \Vert x \Vert_A \le J_A(x) \le
c_2 \Vert x \Vert_A, \quad \forall x\in \bb R_+^A.
\end{equation*}
\item[$c)$] $J_A$ admits a $C^{\infty}$ extension on a open subset of
$\bb R^A$ containing $\bb R_+^A\setminus \{{\bs 0}\}$.

\item[$d)$] Assume that $|A|\ge 2$. For all
$x\in \bb R_+^A\setminus \{{\bs 0}\}$, and $i\in A$,
\begin{equation*}
x_i=0 \quad \implies \quad \nabla_{\bs w_i} J_A \;\;
\textrm{vanishes on a neighborhood of $x$},
\end{equation*}
where each $\bs w_i$ is the canonical projection (restriction) of $\bs
v_i\in \bb R^S$ on $\bb R^A$, i.e.,
$$
[\bs w_i]_j = [\bs v_i]_j,\;\; j\in A\, 
$$
\end{enumerate}
\end{lemma}

\begin{proof}
For $A\not = S$, this follows from the proof of Lemma 4.1 in
\cite[Section 8]{BJL}. For $A=S$, it is clear that
$J_S(x) = \sum_{i\in S} x_i$, $x\in \bb R ^S_+$, satisfies all the
stated properties.
\end{proof}

We shall use each function $J_A$ as a suitable perturbation of
$\|\cdot\|_A$ satisfying the above properties. When $A$ is a
singleton, properties $a)$ and $b)$ imply that $J_A(x)=\lambda x$, for
some $\lambda>0$.

Let us now add some properties derived from Lemma \ref{johp}. Fix a
nonempty set $A\subseteq S$. It follows from property $a)$ that for
every $\alpha>0$ and $x\in \bb R_+^A\setminus \{\bs 0\}$,
\begin{equation}
\label{joh1}
(\nabla J_A) (\alpha x) \,=\,  (\nabla J_A) (x) \quad \textrm{and}
\quad ({\rm Hess} \,J_A) (\alpha x)=\alpha^{-1} ({\rm Hess}  \,J_A) (x).
\end{equation}
Since $\nabla J_A$ and ${\rm Hess} \, J_A$
are continuous on $\bb R_+^A\setminus \{0\}$, then, by \eqref{joh1},
\begin{equation}\label{joh2}
\sup_{x\in \bb R_+^A\setminus \{\bs 0\}} \|\nabla J_A (x) \|_A <\infty
\;\;\;\text{ and } \sup_{x\in \bb R_+^A\setminus \{\bs 0\}}
\|x\|_A \|\text{Hess} \, J_A (x) \|_A <\infty.
\end{equation}
Lastly, assume that $|A|\ge 2$, fix an arbitrary $i\in A$ and recall
the vector ${\bs w}_i$ from property $d)$. In virtue of this property,
\begin{equation}
\label{03}
\mtt 1\{x_i>0\} \, \frac{\nabla_{\bs w_i} J_A(x)}{x_i} , \quad x\in
\bb R_+^A\setminus \{\bs 0\} \,, \;\;
\text{ is continuous.} 
\end{equation}
Then, by \eqref{joh1},
\begin{equation}
\label{johl}
\sup_{x\in \bb R_+^A\setminus \{\bs 0\}, \;x_i>0} \; \|x\|_A
\;\frac{|\nabla_{\bs w_i} J_A(x)|}{x_i}  <\infty\,.
\end{equation}

Let $\mf J_A\colon \bb R^S_+ \to \bb R_+$ be given by
\begin{equation}
\label{mfi}
\cb{ \mf I_A(x)} \,:=\, J_A( x_A )\,.
\end{equation}
To keep the notation simple, we define that $\|x\|_A= \|x_A\|_A $,
where $x_A$ stands for the canonical projection of $x$ on $\bb R^A$,
$A$ is a nonempty subset of $S$ and $x\in \Sigma$.  The next result is
the estimate \eqref{johl} stated in terms of this new notation
$\mf I_A$.

\begin{lemma}
\label{lem:normmapbdd}
For all nonempty $A\subseteq S$ and $i\in A$ we have
$$
\sup_{x\in \bb R^S_+ : \|x\|_A >0, \, x_i>0} \; \|x\|_A \;
\frac{| \nabla_{\bs v_i} \mf I_A(x)|}{x_i}  <\infty. $$
\end{lemma}
    
Next result is a  corollary of Lemma \ref{lem:normmapbdd}.

\begin{lemma}
\label{lem:normmapvanish}
Fix $x\in \Sigma$ and let $C = \{i\in S : x_i \neq 0\}$.  Suppose
$B\subsetneq C$. For $A = S\setminus B$, $\mf I_A(x) > 0 $ and
    $$ \nabla_{\bs v_k} \mf I_A(x) = 0, \quad \forall k \notin C.$$
\end{lemma}

\begin{proof}
We first show that $\|x\|_A > 0$.  Suppose not. Then, $\|x\|_A = 0$
implies $x_i = 0$ for all $i\in A$.  This implies $A \subset C^c$,
therefore $A^c = B \supset C$. This contradicts the assumption that
$B\subsetneq C$.

Fix $k\not \in C$, and let $(x^n)_{n\in \bb N}\subset \Sigma$ be a
sequence such that
    $$ x^n \to x \;\; \text{ and } \;\; (x^n)_k >0,\;\; \forall n\in \bb N.$$
    By Lemma \ref{lem:normmapbdd}, 
    $$ |\nabla_{\bs v_k} \mf I_A(x^n)| \to 0,$$
    which implies the assertion of the lemma.
\end{proof}

Using the auxiliary function $\mf I_A$, we derive estimates of
functions in $\mc E_A$.  For $U \subset \Sigma$, let
\begin{equation*}
\cb{BC(\Sigma,U) } \,:=\,  \big\{ f: \Sigma \to \bb R \text{ is
Borel measurable, bounded, 
and } f|_{U} \text{ is continuous} \big\}\,.
\end{equation*}

\begin{lemma}
\label{lem:estEA}
Fix $\varnothing \neq B \subsetneq S$, and let $A = S \setminus B$.
For any $G\in \mc E_A$ with $G|_{\Sigma_B}=0$,
\begin{enumerate}[leftmargin=*]

\item[\textup{(1)}] $G = H \mf I_A^2$ for some
$H\in BC(\Sigma,\Sigma\setminus \Sigma_B)$.

\item[\textup{(2)}] For all $i, j \in S$,
$(\partial_{x_i} -\partial_{x_j} ) G = K \mf I_A$ for some
$K\in BC(\Sigma,\Sigma\setminus \Sigma_B)$.
\end{enumerate}
\end{lemma}

\begin{proof}
[Proof of (1)]
By condition $\mathfrak{E1} (i)$, $i\in A$, $G$ belongs to
$C^1(\Sigma)$ and there exists a bounded Borel
function $h_i\colon \Sigma \to \bb R$, such that
\begin{equation*}
\frac{\nabla_{\bs v_i} G(x)}{x_i} = h_i(x) \quad
\text{for}\;\; \{x_i>0 \}\,.
\end{equation*}
Let $C>0$ be an upper bound of the functions $|h_i|$, that is,
$\max_{i \in A}\sup_{x\in \Sigma} |h_i(x)| \leq C$, so that
\begin{equation}
\label{04}
|\, \nabla_{\bs v_i} G(x)\,|\,\le\, C \,x_i \quad 
\text{for all}\;\; x \in \Sigma \,.
\end{equation}

We claim that there exists a finite constant $C_0$ such
that $|G(x)|\le C_0 \mf I^2_A(x)$ for all $x\in \Sigma$. This
inequality trivially holds on $\Sigma_B$ because both functions vanish
on this set. Fix $x\in \Sigma\setminus \Sigma_B$.  Since
$G$ is of class $C^1 (\Sigma)$, $G|_{\Sigma_B}=0$ and $\gamma_B(x) \in
\Sigma_B$, by Lemma \ref{lem:RAlinmap} and Lemma \ref{app:1},
\begin{align*}
G(x) & = \int_0^1 \nabla_{x-\gamma_B(x)} G(\gamma_B(x) +
t(x-\gamma_B(x))) \, dt
\\
&= \, - \, 
\sum_{i\in A}\int_0^1 [L_A(x_A)]_i \nabla_{\bs v_i}
G(\gamma_B(x) + t(x-\gamma_B(x))) \, dt\,.
\end{align*}
By definition of the constant $C$, the absolute value of the previous
expression is bounded by
\begin{align*}
&  C \sum_{i\in A}
\big | \, [L_A(x_A)]_i \,\big| \, \int_0^1
\big|\, [\gamma_B(x) + t(x-\gamma_B(x))]_i \,\big|\; dt \\
&\quad \le\, C \sum_{i\in A} \big | \, [L_A(x_A)]_i \,\big| \,
\Big\{ \, |\, [\gamma_B(x) ]_i  \,| \,+\,  |\,x_i \,| \,  \Big\}  
\, \leq \, C'\,  \|x_A\|^2 
\end{align*}
for some new finite constant $C'$. Thus,
$|G(x)| \leq C_0 \mf I_A^2(x)$ for some finite constant $C_0>0$, as
claimed.

To complete the proof of assertion (1), it remains to
define $H(x)$ as $G(x)/\mf I_A^2(x)$ for
$x\in \Sigma \setminus \Sigma_B$ and $H(y) =0$ for $y\in
\Sigma_B$. Clearly, $G = H \mf I^2_A$, $H$ is  bounded and
measurable, and $H$ restricted to $\Sigma \setminus \Sigma_B$ is
continuous because so are $G(\cdot)$ and $\mf I_A(\cdot)$ on this set.

\smallskip \noindent \textit{Proof of (2).} Fix
$i \neq j\in S$. We claim that there exists a finite constant $C$ such
that
\begin{equation}
\label{30}
|\, (\partial_{x_i} - \partial_{x_j}) G (x) \,| \le C\, \mf I_A(x)
\end{equation}
for all $x\in \mathring{\Sigma}$.

Rewrite $(\partial_{x_i} - \partial_{x_j}) G$ as
$\nabla_{\bs e_i - \bs e_j} G$, so that
\begin{equation}
\label{05}
|\, (\partial_{x_i} - \partial_{x_j}) G(x)\,|  \,\le\,
|\, \nabla_{\gamma_B (\bs e_i - \bs e_j) } G (x)\,|
\,+\,  |\, \nabla_{\gamma_B (\bs e_i - \bs e_j)  - (\bs e_i - \bs e_j) } G (x)\,|
\end{equation}
By Lemma \ref{lem:lincomb}.(2), \eqref{04}, and the definition
\eqref{mfi} of $\mf I_A$,
the second term is less than or equal to
\begin{equation*}
C \, \sum_{k\in A} |\, \nabla_{\bs v_k} G(x) \, |
\, \leq \,  C \, \mf I_A(x) 
\end{equation*}
for some finite constant $C$, which may change from line to line. 

We turn to the first term on the right-hand side of \eqref{05}.  As
$x\in \mathring{\Sigma}$, the interior of the line segment  between
$x$ and $\gamma_B(x)$ is contained in $\mathring{\Sigma}$.

For convenience, let $\bs w = \gamma_B(\bs e_i - \bs e_j)$.  By
\eqref{eq:proj}, $\bs w_k =0$ for $k\in A$. Thus, as $G$ vanishes on
$\Sigma_B$, and $\gamma_B(x) \in \mathring{\Sigma}_B$,
$\nabla_{\bs w} G(\gamma_B(x)) =0$. By condition
$\mf E_2(S)$,
\begin{align*}
|\nabla_{\bs w} G(x)| = |\nabla_{\bs w} G(x)
- \nabla_{\bs w} G(\gamma_B(x)) |
= \Big|\, \int_0^1 \nabla_{x-\gamma_B(x)} \nabla_{\bs w} G(\gamma_B(x)
+ t(x-\gamma_B(x))) \, dt\, \Big|\, .
\end{align*}
As $G$ satisfies condition $\mathfrak{E2} (S)$, $\bs w$,
$x-\gamma_B(x)$ belong to $T_{\Sigma}$, and
$\gamma_B(x) + t\, (x-\gamma_B(x))$ to $\mathring\Sigma$, by
\eqref{27}, there exists a finite constant $C>0$ such that
\begin{equation*}
\big| \, \nabla_{x-\gamma_B(x)} \nabla_{\bs w} G(\gamma_B (x)
+ t(x-\gamma_B(x)))\, \big| \,\le\,  C\,  \Vert\, x-\gamma_B(x)\Vert
\, \Vert \bs w \Vert \, .
\end{equation*}
By Lemma \ref{lem:RAlinmap}, this expression is equal to
\begin{equation*}
C \, \Vert \bs w\Vert\, \Big\Vert\sum_{i\in A} \, [L_A(x_A)]_i \,
\bs v_i\, \Big\Vert \, \le C' \, \|x_A\|\, .
\end{equation*}
This proves the claim \eqref{30}.

We may extend the estimate \eqref{30} to $x\in \Sigma$, Consider a
sequence $x^n \to x$ such that $x^n \in \mathring{\Sigma}$. Since
$ | (\nabla_{\bs e_i - \bs e_j} G) (x^n)| \leq C\, \mf I_A(x^n)$ for
all $n\ge 1$, as $G\in C^1(\Sigma)$ and $\mf I_A$ is continuous on
$\Sigma$, letting $n\rightarrow \infty$, yields that \eqref{30} holds
for $x\in \Sigma$.

To complete the proof of assertion (2), it remains to define $K(x)$ as
$\nabla_{\bs e_i - \bs e_j} G(x)/\mf I_A(x)$ for
$x\in \Sigma \setminus \Sigma_B$ and $K(y) =0$ for $y\in
\Sigma_B$. Clearly, $K$ is bounded and measurable, and $K$ restricted
to $\Sigma \setminus \Sigma_B$ is continuous because so are
$\nabla_{\bs e_i - \bs e_j} G(\cdot)$ and $\mf I_A(\cdot)$ on this
set.
\end{proof}

We conclude this section by constructing a set of cutoff
functions $\Phi_A$ in the domain of the generator which are equal to
$1$ if $\min_{i\in A}^{} x_i\ge \epsilon$ and equal to $0$ if
$\min_{i\in A}^{} x_i \le \delta$ for some $0<\delta<\epsilon$. This
is the content of the next result, which is an adaptation of
\cite[Lemma 3.5]{ABCJ}.

\begin{lemma} 
\label{phi}
Let $A$ be a nonempty subset of $S$ and
$ B = S\setminus A$. Given $\epsilon>0$,
there exist $\Phi: \Sigma \to [0,1]$ and $\delta\in (0,\epsilon)$ such
that $\Phi \in \mathcal D_{S}$,
\begin{enumerate} [leftmargin=*]
\item[$i)$] $\build\min_{i\in A}^{} x_i\ge \epsilon$ \; $\implies$ \;
$\Phi(x)=1$,\quad  and 
\item[$ii)$] $\build\min_{i\in A}^{} x_i \le \delta$ \; $\implies$ \; $\Phi(x)=0$.
\end{enumerate}
\end{lemma}

\begin{proof}
Fix a nonempty subset $A$ of $S$.  Let
$\cb{\phi } \colon \bb R \to [0,1]$ be a smooth function such that
\begin{align*}
\phi(x)=0, \;\textrm{for} \; x\le 2/3
\quad \textrm{and} \quad \phi(x)=1, \;\textrm{for} \; x\ge 1.
\end{align*}
Let $(J_D)_{\varnothing \varsubsetneq D \subseteq S}$ be functions
satisfying all properties in Lemma \ref{johp}. Since each $J_D$ can be
rescaled, without loss of generality we may assume that property $b)$
in Lemma \ref{johp} is satisfied for $1=c_1\le c_2=C$. For each
nonempty $D\subseteq S$, we define
\begin{equation*}
\cb{\varepsilon_D}\, := \,\epsilon \,(3C)^{|D|-1-|W|}
\end{equation*}
and $\varphi_D \in C^{\infty}(\Sigma)$ as
\begin{equation*}
\cb{ \varphi_D} ( x ) \,:=\,  \phi\left(  \frac{J_D( x_D
)}{\varepsilon_D} \right),
\quad x\in \Sigma,
\end{equation*}
where $x_D$ is the canonical projection of $x$ on $\bb R ^D$. To keep
notation simple we set $\cb{\Vert x \Vert_D=\Vert x_D \Vert_D}$, for
$x\in \Sigma$. It is clear that ${\varphi}_D$ can be extended smoothly
to an open set containing $\Sigma$. By \eqref{03},
${\varphi}_D \in {\mathcal D}_D$, and by definition of $\phi$,
\begin{equation}\label{jh1}
{\varphi}_D(x) = \left\{
\begin{array}{cl}
0, & \textrm{if $\|x\|_D\le 2\varepsilon_D(3C)^{-1}$}, \medskip\\
1, & \textrm{if $\|x\|_D\ge \varepsilon_D$}.
\end{array}
\right.
\end{equation}

Define for each $k\in A$,
\begin{align}
\label{a24}
\cb{ \Phi_k} \,:=\, 
\prod_{W\subseteq B} {\varphi}_{W\cup\{k\}}\,,
\end{align}
so that $\Phi_k={\varphi}_{\{k\}}$ if $W$ is empty. Let us finally
check that
\begin{align*}
\cb{\Phi} \,:=\, \prod_{k\in A}\Phi_k
\end{align*}
fulfills the conditions of the lemma. For each $k\in A$, and
$W\subseteq B$, by \eqref{jh1}
\begin{equation*}
x_k\ge \epsilon \quad \implies \quad \Vert x \Vert_{W \cup
\{k\}} \ge \epsilon \ge \varepsilon_{W \cup \{k\}} \quad
\implies \quad \varphi_{W \cup \{k\}}  (x) = 1\,.
\end{equation*}
Thus, by \eqref{a24}, $\Phi_k(x)= 1$.  Hence, $i)$ holds.

By \eqref{jh1}, for each $k\in A$ we have
\begin{align*}
x_k \le \frac{2 \varepsilon_{\{k\}}}{3C}  \quad \implies
\quad {\varphi}_{\{k\}}(x)=0 \quad \implies \quad \Phi(x)=0.
\end{align*}
Therefore, $ii)$ holds by choosing
\begin{align*}
\delta := \frac{2 \epsilon}{(3C)^{|W|+1}} < \epsilon.
\end{align*}

By $ii)$, it is obvious that $\Phi\in \mathcal D_A$. It remains to
prove that $\Phi\in \mathcal D_B$ if $B$ is nonempty. This property
holds if we show that
\begin{align*}
\Phi_k \in \mathcal D_{j}, \quad \textrm{for all $k\in A$ and $j\in B$.}
\end{align*}
Fix $k\in A$ and $j\in B$, and  write
\begin{align*}
\Phi_k = \prod_{E\subseteq B\setminus \{j\}}
\left( \varphi_{E\cup \{k\}} \,\varphi_{E\cup \{j,k\}}  \right).
\end{align*}
It is therefore enough to verify that
\begin{align*}
\varphi_{E\cup \{k\}} \,\varphi_{E\cup \{j,k\}} \in \mathcal{D}_j,
\quad \textrm{for each $E\subseteq B\setminus\{j\}$}.
\end{align*}
By \eqref{03},
$\varphi_{E\cup \{j,k\}} \in \mc D_{E\cup \{j,k\}} \subset
\mathcal{D}_{j}$, it remains to show that 
\begin{equation}
\label{c}
x  \; \mapsto \;  \mtt  1 \{x_j>0\} \, \left(\frac{1}{x_j}\right)
\varphi_{E\cup \{j,k\}}(x) \; \nabla_{{\bs v}_j} \varphi_{E\cup
\{k\}}(x)
\quad \textrm{is continuous on $\Sigma$}\,.
\end{equation}
Fix some $x\in \Sigma$ such that $x_j=0$. On the one hand, by \eqref{jh1},
\begin{equation*}
\Vert x \Vert_{E\cup \{k\}} > \varepsilon_{E\cup\{k\}}
\quad \implies \quad  \varphi_{E\cup\{k\}} \equiv 1 \;\;
\textrm{on a neighborhood of $x$}.
\end{equation*}
On the other hand, by \eqref{jh1},
\begin{align*}
\Vert x \Vert_{E\cup \{j,k\}} = \Vert x \Vert_{E\cup \{k\}}
< 2 \varepsilon_{E\cup\{k\}} = \frac{2 \varepsilon_{E\cup\{j,k\}}}{3C}
\quad \implies \quad  \varphi_{E\cup\{j,k\}} \equiv 0 \;\;
\textrm{on a neighborhood of $x$}.
\end{align*}
Therefore,
$\varphi_{E\cup \{j,k\}}(x) \; \nabla_{{\bs v}_j} \varphi_{E\cup
\{k\}}(x)$ vanishes on a neighborhood of $\{x_j=0\}$, as claimed in
\eqref{c}. This completes the proof of the lemma.
\end{proof}

\section{Proof of Theorem \ref{thm:emp}}
\label{sec:proofthmemp}

The proof of Theorem \ref{thm:emp} is divided in a few steps.  First,
for each finite signed measure $\mu$ on $\Sigma$, we introduce a
topology on the space $\mc E_S$, called the $\mu$-topology. This
topology is tailor-made for martingale problems. More precisely, fix
$H\in \mc E_S$. Suppose that for each finite signed measure $\mu$ on
$\Sigma$, there exists a sequence of functions $(H_n)_{n\ge 1}$ in
$\mc E_S$ converging to $H$ in the $\mu$-topology and such that, for
each $n\ge 1$,
\begin{equation*}
H_n(X_{t}) - H_n(X_{0}) - \int_{0}^{t} (\mf L^{\mc E} H_n)(X_s) \,ds
\end{equation*}
is a martingale in the canonical space
$(C([0,T], \Sigma), \bb P, (\ms F_t)_{t\ge 0})$. Then the previous
expression with $H$ replacing $H_n$ is also a martingale.

The main result of this section, Proposition \ref{prop:mudense},
states that the space $\mc D_S$ is dense in $\mc E_S$ in this
topology, in the sense that for each $H\in \mc E_S$ and finite signed
measure $\mu$ on $\Sigma$, there exists a sequence of functions
$(H_n)_{n\ge 1}$ in $\mc D_S$ converging to $H$ in the
$\mu$-topology. Theorem \ref{thm:emp} is a simple consequence of this
result. This is the content of Subsection \ref{sec6.1}.

In Subsection \ref{sec6.2}, we prove Proposition \ref{prop:mudense} in
three steps. We first define functional spaces $\mc K^{\mc G}$,
$\mc G \subset 2^S$, such that $\ms E_S \subset \mc K^{\varnothing}$,
and $\mc K^{\mc G} \supset \mc K^{\mc G'}$ if $\mc G\subset \mc G'$.
Lemma \ref{lem:2SsubsetD} asserts that $\mc K^{2^S} \subset \mc D_S$
and Proposition \ref{prop:piledense} that for any finite signed
measure $\mu$ on $\Sigma$, $\mc K^{\mc G \cup \{B\}}$ is $\mu$-dense
in $\mc K^{\mc G}$. Proposition \ref{prop:mudense} follows from the
previous results.

\subsection{A $\mu$-topology on $\mc E_S$}
\label{sec6.1}

Let $\cb{\mc M(\Sigma)}$ be the space of finite signed Borel
measures on $\Sigma$.  Fix $\mu \in \mc M(\Sigma)$.  For
$H\in \mc E_S$, we define the norm $\|H\|_{\mu}$ by
$$
\cb{ \|H\|_{\mu}}\,:=\, \|H\|_{\infty} + \Big|\,
\int_{\Sigma} \mf L^{\mc E} H \; d\mu\,\Big|\,.
$$

This norm induces a metric and a topology on $\mc E_S$. This topology
will be called the $\mu$-topology. It is clearly first
countable. Therefore, for any set $C \subset \mc E_S$, $x\in \overline{C}$
if and only if there exists a sequence $x_n\in C$ converging to $x$.
For any $C\subset D \subset \mc E_S$, we say $C$ is $\mu$-dense in $D$
if $D \subset \overline{C}$ in the $\mu$-topology.  We prove the following
proposition in the next section.

\begin{proposition}
\label{prop:mudense}
For each $\mu\in \mc M(\Sigma)$ and $f\in \mc E_S$, there exists
a sequence $f_n\in \mc D_S$ converging to $f$ in the $\mu$-topology.
\end{proposition}

Assuming Proposition \ref{prop:mudense}, we prove Theorem
\ref{thm:emp}

\begin{proof} [Proof of Theorem \ref{thm:emp}]
Suppose that $\bb P$ solves the $(\mf L, \mc D_S)$-martingale problem.
Fix $H\in \mc E_S$.  It is enough to show that for any $n\ge 1$,
continuous function $G\colon \Sigma^n \rightarrow \bb R$, and
$0\leq s_1 \leq \cdots \leq s_n \leq t_1 <t_2$,
\begin{equation}
\label{martcond}
\bb E\left[G(X_{s_1},\cdots, X_{s_n})\left\{H(X_{t_2})
- H(X_{t_1}) - \int_{t_1}^{t_2} (\mf L^{\mc E} H)(X_s)ds\right\}\right] = 0\,,
\end{equation}
where $\bb E$ represents the expectation with respect to $\bb P$.

For $\phi \in C(\Sigma)$, observe that
$$\phi \mapsto \bb E\left[G(X_{s_1},\cdots, X_{s_n})
\int_{t_1}^{t_2} \phi(X_s)ds\right]$$
is a bounded linear functional on $C(\Sigma)$.  Therefore, there
exists $\mu\in \mc M(\Sigma)$ such that
\begin{equation} \label{martcond2} \bb E\left[G(X_{s_1},\cdots,
X_{s_n})\int_{t_1}^{t_2} \phi(X_s)ds\right] = \int_{\Sigma}
\phi(x)\;d\mu(x) \, .
\end{equation}
Hence, \eqref{martcond} is equivalent to
$$ \bb E\left[G(X_{s_1},\cdots, X_{s_n})\left\{H(X_{t_2}) - H(X_{t_1})
\right\} \right] = \int_{\Sigma} (\mf L^{\mc E} H)(x) d\mu (x)\,.
$$

By Proposition \ref{prop:mudense}, we may take a sequence
$H^{\mu}_n\in \mc D_S$ converging to $H$ in the $\mu$-topology.  Since
$H^{\mu}_n\in \mc D_S$, and $\bb P$ solves the
$(\mf L, \mc D_S)$-martingale problem, by Proposition
\ref{prop:naturalextend}, 
$$
\bb E\big[\, G(X_{s_1},\cdots, X_{s_n}) \, \{H^{\mu}_n(X_{t_2}) -
H^{\mu}_n(X_{t_1}) \} \,\big]
= \int_{\Sigma} (\mf L H^{\mu}_n)(x) \; d \mu(x)
= \int_{\Sigma} (\mf L^{\mc E} H^{\mu}_n)(x) \; d\mu(x)\,.
$$
Taking $n\rightarrow \infty$ completes the proof.
    
\end{proof}

\subsection{Proof of Proposition \ref{prop:mudense}}
\label{sec6.2}

In this section, we prove that $\mc D_S$ is $\mu$-dense in $\mc E_S$ for
all $\mu \in \mc M(\Sigma)$.

\begin{definition}
Let $\ms F$ be a collection of subsets of $S$, that is,
$\ms F \subset 2^S$. We say $\ms F$ is a pile if for any $A\in \ms F$,
$B\subset A$ implies $B\in \ms F$.  For any pile $\ms F$, we define a
domain $\ms K^{\ms F}$ by
$$
\cb{\ms K^{\ms F} } \,:=\,  \big\{ \, F \in \mc E_S :
\forall B\in \ms F, \exists\,
\epsilon_B > 0 \text{ such that } \nabla_{\bs v_j} F(x) = 0\;\;
\forall j \in S\setminus B\, ,\,
x \text{ with } \|x\|_{S \setminus B} <\epsilon_B\, \big\},
$$ 
where we treat $\|\cdot \|_{\varnothing} = 0$.
\end{definition}

By definition, $\ms K^{\varnothing} = \mc E_S$. 
Intuitively, $\ms K^{\ms F}$ is the set of functions that are
fiberwise constant near the boundary $\Sigma_B$ for each $B\in \ms F$
with respect to the projection $\gamma_B\colon \Sigma \to \Sigma_B$.
Recall from Section \ref{sec:extmap} the definition of the function
$ \gamma^*_B F\colon \Sigma \to \bb R$ for a function
$F\colon \Sigma_B \to \bb R$.

\begin{lemma}
\label{lem:pile}
Let $\ms F$ be a pile, and fix $F\in \ms K^{\ms F}$. For
$\varnothing \neq B \in \ms F$, let $A=B^c$. Then,
$$
F(x) = F(\gamma_B(x)) =  [\gamma^*_B (F |_{\Sigma_B}) ] (x)
\text{ for all }
x \in \Sigma \text{ such that } \|x\|_A < \epsilon_B.
$$
\end{lemma}

\begin{proof}
Fix $x \in \Sigma$ such that $\|x\|_A < \epsilon_B$. Convexity of
$\|\cdot\|_A$ implies that the line segment between $x$ and
$\gamma_B(x)$ is contained in the set $ \{y\in \Sigma : \|y\|_A <
\epsilon_B\}$. 

Write
$$F(x) - F(\gamma_B(x)) = \int_0^1 \nabla_{x-\gamma_B(x)}
F(\, \gamma_B(x)+t\, (x-\gamma_B(x))\, ) \, dt \,.
$$
Since $\| \gamma_B(x) + t (x-\gamma_B(x)) \|_A < \epsilon_B$ for all
$0\le t\le 1$, and because $F$ belongs to $\ms K^{\ms F}$,
$\nabla_{\bs v_j}F(\gamma_B(x) + t (x-\gamma_B(x))) = 0$ for $j\in A$
and $0\le t\le 1$. Thus, since by Lemma \ref{lem:lincomb},
$\gamma_B(x) - x$ is a linear combination of $\bs v_k$ for $k\in A$,
the previous integral vanishes. This proves the lemma.    
\end{proof}

Let $2^S$ be the collection of all subsets of $S$. The next result
asserts that $\ms K^{2^S} \subset \mc D_S$:

\begin{lemma}
\label{lem:2SsubsetD}
It holds that $\ms K^{2^S} \subset \mc D_S$.
\end{lemma}

\begin{proof}
Fix $F\in \ms K^{2^S}$.

\smallskip\noindent{\bf Claim 1:} $F \in C^2(\Sigma)$. \smallskip

For any $\bs V, \bs W \in T_\Sigma$, we need to find a function $G$ in
$C(\Sigma)$ such that
$ \nabla_{\bs V} \nabla_{\bs W} F = G \text{ on } \mathring{\Sigma}$.

Fix $\bs V, \bs W \in T_\Sigma$.  Let
$\epsilon = \min_{\varnothing \neq B \subset S} \epsilon_B$.  For
$B\subset S$, define
\begin{gather*}
U_B = \{x\in \Sigma : \forall i \in B, x_i > 0, \|x\|_{S\setminus B}
< \epsilon/3 \}\,, \\
V_B = \{x\in \Sigma : \forall i \in B, x_i > 0, \|x\|_{S\setminus B} <
\epsilon \}\,, 
\end{gather*}
so that $U_B \subset V_B$.  By Lemma \ref{lem:pile},
$F = (\gamma^*_B F)|_{\Sigma_B}$ on $V_B$. Noting that
$\Vert x\Vert_\varnothing =0$, we can easily see that the sets $U_B$,
$B\subset S$, form an open cover of $\Sigma$.  Let
$G_B\colon U_B \to \bb R$ be given by
$$
G_B(x) \,:=\,  (\nabla_{\bs \gamma_B(\bs V)}
\nabla_{\bs \gamma_B(\bs W)} F|_{\Sigma_B} ) (\gamma_B(x))\,.
$$
This value is well defined since $\gamma_B(x)_i \ge x_i >0 $ for all
$i \in B$ and $F|_{\mathring{\Sigma}_B} \in C^2(\mathring{\Sigma}_B)$.
Note that $G_S = \nabla_{\bs V} \nabla_{\bs W} F$ on
$\mathring{\Sigma}$ by definition. The function
$G_B$ is continuous
because $F$ belongs to $\mc E_S$.

As $U_B$, $B\subset S$, forms an open cover of $\Sigma$, to complete
the proof, it remains to show that
$$ G_B = G_C\; \text{ on }\; U_B \cap U_C,\;\; B,C \subset S.$$
Fix $x\in U_B \cap U_C$. By definition,
$$ x_i > 0 \text{ for all } i \in B \cup C, \;\;
\text{ and } \;\; \|x\|_{S\setminus B} < \epsilon/3, \;\;
\|x\|_{S\setminus C} < \epsilon/3.$$
By \eqref{eq:probrep} and \eqref{eq:proj},
$\|x\|_{S\setminus B} < \epsilon/3$ implies that
$$\|x - \gamma_B(x)\|_S = \sum_{j\in B}
\sum_{k\in S\setminus B} u^B_j(k) x_k
+ \sum_{j\in S\setminus B} x_j
= 2 \sum_{k\in S\setminus B} x_k < 2\epsilon/3.$$
Therefore,
$$ \|\gamma_B(x)\|_{S\setminus C}
\le \|x-\gamma_B(x)\|_{S\setminus C}
+ \|x\|_{S\setminus C} < 2\epsilon/3 + \epsilon/3 = \epsilon\,.
$$
Since $[\gamma_B(x)]_i=0$ for $i\not\in B$,
$\|\gamma_B(x)\|_{S\setminus (B\cap C)} = \|\gamma_B(x)\|_{(S\setminus
B) \cup (S\setminus C)} = \|\gamma_B(x)\|_{S\setminus C}$. Thus, by
the previous estimate,
$\|\gamma_B(x)\|_{S\setminus (B\cap C)} < \epsilon$, so that
$\gamma_B(x) \in V_{B\cap C}$.

By Lemma \ref{lem:pile}, $F= \gamma^*_{B\cap C} F|_{\Sigma_{B\cap C}}$
on $V_{B\cap C}$, so that
$F|_{\Sigma_B} = (\gamma^*_{B\cap C} F|_{\Sigma_{B\cap
C}})|_{\Sigma_B}$ on $V_{B\cap C} \cap \mathring{\Sigma}_B$.  As
$F\in \ms K^{2^S} \subset \mc E_S$,
$F|_{\mathring{\Sigma}_{B\cap C}} \in C^2(\mathring{\Sigma}_{B\cap
C})$.  Thus, by the chain rule, for any
$y\in V_{B\cap C} \cap \mathring{\Sigma}_B$ and
$\bs X, \bs Y \in T_{\Sigma_B}$,
$$
\nabla_{\bs X} \nabla_{\bs Y} F|_{\Sigma_B}(y)
= \nabla_{\bs X} \nabla_{\bs Y}
(\gamma^*_{B\cap C} F|_{\Sigma_{B\cap
C}}) (y)
= \nabla_{\bs \gamma_{B\cap C}(\bs X)}
\nabla_{\bs \gamma_{B\cap C}(\bs Y)} F|_{\Sigma_{B\cap
C}}(\gamma_{B\cap C}(y))\,.
$$
Since $x\in U_B\cap U_C$, $\gamma_B(x) \in \mathring{\Sigma}_B$. On
the other we proved above that $\gamma_B(x) \in V_{B\cap C}$, and so
$\gamma_B(x) \in \mathring{\Sigma}_B \cap V_{B\cap C}$.  Hence, by the
previous identity for $y=\gamma_B(x) $, and 
Lemma \ref{lem:recursive}, 
$$ G_B(x) = \nabla_{\bs \gamma_B(\bs V)}
\nabla_{\bs \gamma_B(\bs W)} F|_{\Sigma_B}(\gamma_B(x))
= \nabla_{\bs \gamma_{B\cap C}(\bs V)}
\nabla_{\bs \gamma_{B\cap C}(\bs W)}
F|_{\Sigma_{B\cap C}}(\gamma_{B\cap C}(x))\,.
$$
In particular,  $G_B(x) = G_C(x)$ on $U_B \cap U_C$, which proves
Claim 1.

\smallskip\noindent{\bf Claim 2:} $F \in \mc D_i $ for all $i\in S$. \smallskip

It is enough to show that for any $i\in S$, $x\in \Sigma$ with
$x_i = 0$, and a sequence $x_n \to x$ with $(x_n)_i > 0$ for all $n$,
$$
\frac{m_i}{(x_n)_i} \nabla_{\bs v_i} F(x_n) \to 0 \,.
$$
Since $F \in \ms K^{2^S}$, and $S\setminus \{i\} \in 2^S$,
$ \nabla_{\bs v_i} F(y) = 0$ for all $y$ with
$y_i < \epsilon_{S\setminus \{i\}}$. This completes the proof of the
lemma.
\end{proof}

For two piles $\ms F_1, \ms F_2$, we say
$\ms F_1 \triangleleft \ms F_2$ if there exists $B\subset S$ such that
$\ms F_2 = \ms F_1 \cup \{B\}$. The next proposition is the key result
of  this section.

\begin{proposition}
\label{prop:piledense}
Fix $\mu \in \mc M(\Sigma)$. For any pair of piles satisfying
$\ms F_1 \triangleleft \ms F_2$, $\ms K^{\ms F_2}$ is $\mu$-dense in
$\ms K^{\ms F_1}$.
\end{proposition}

The proof of this proposition requires to explicitly approximate a
function in $\ms K^{\ms F_1}$ by a function in $\ms K^{\ms F_2}$.  To
do so, we need to use a smooth cutoff function technique.  Consider a
smooth cutoff function $\chi: \bb R_{\geq 0} \to \bb R_{\geq 0}$
satisfying:
$$\chi(x) = 1 \text{ for } x \leq 1, \;\; \chi(x)
= 0 \text{ for } x \geq 2, \;\; \chi \text{ is decreasing.}$$
For $\epsilon>0$, define
$\chi_\epsilon: \bb R_{\geq 0} \to \bb R_{\geq 0}$ by
$\chi_\epsilon(x) = \chi(x/\epsilon)$.  Note the following properties
of the cutoff $\chi_\epsilon$:
\begin{enumerate}[leftmargin=1cm]
\item[($\chi1$)] There exists $C>0$ such that
$\sup_{\epsilon>0}\sup_{x\geq 0}\chi_\epsilon(x) \leq C $.

\item[($\chi2$)] There exists $C>0$ such that
$\sup_{\epsilon>0}\sup_{x\geq 0} x\chi_\epsilon'(x) \leq C $.

\item[($\chi3$)] There exists $C>0$ such that
$\sup_{\epsilon>0}\sup_{x\geq 0} x^2\chi_\epsilon''(x) \leq C $.
\end{enumerate}

Fix two piles $\ms F_1, \ms F_2$, and assume that
$\ms F_2 = \ms F_1 \cup \{B\}$.  Let $A = S\setminus B$.  Fix a
function $F\in \ms K^{\ms F_1}$. Let
$F_\epsilon\colon \Sigma \to \bb R$ be the function defined by
$$
\cb{F_\epsilon}\,:=\,  [\chi_\epsilon \circ \mf I_A] \, \gamma^*_B
(F|_{\Sigma_B}) + [(1-\chi_\epsilon) \circ \mf I_A] \, F.
$$
It follows from the next lemma that $F_\epsilon \rightarrow F$ in the
$\mu$-topology.

\begin{lemma}
\label{lem:FFe}
There exists $\delta>0$ that depends on $F$ such that for all small
enough $\epsilon>0$, 
$$\textup{supp}(F - F_\epsilon) \subset
\{x\in \Sigma : \mf I_A(x) \leq 2\epsilon \text{ and } \min_{j\in B}
x_j \geq \delta \}.$$ Moreover, for small enough $\epsilon>0$,
$F_\epsilon \in \ms K^{\ms F_2}$, and
\begin{enumerate}[leftmargin=*]
\item[\textup{(1)}] $(F-F_\epsilon)|_{\Sigma_B} = 0$ for all
$\epsilon>0$.

\item[\textup{(2)}]
$\{x\in \Sigma : \mf L^{\mc E} F(x) \neq \mf L^{\mc E} F_\epsilon(x)
\}$ shrinks to $\varnothing$ as $\epsilon \to 0$, i.e., the limsup of
the sequence of sets is empty.

\item[\textup{(3)}] As $\epsilon \to 0$,
$\|\mf L^{\mc E} (F- F_\epsilon)\|_{\infty}$ is uniformly bounded.
\end{enumerate}
\end{lemma}

\begin{proof}
For  $\epsilon$, $\delta>0$, let
$$
\cb{\Lambda^B_{\delta, 2\epsilon} }\,:=\,  \{x\in \Sigma : \mf I_A(x) \leq
2\epsilon \text{ and } \min_{j\in B} x_j \geq \delta \}.$$ By
definition, $F = F_\epsilon$ when
$\mf I_A(x) \geq 2\epsilon$.

\smallskip\noindent {\it Assertion A:} There exist $\delta>0$ and
$\epsilon_0>0$, which only depend on $F$, such that $F = F_\epsilon$
in the domain
$$\{x\in \Sigma : \mf I_A(x) < 2\epsilon \text{ and } \min_{j\in B}
x_j \leq \delta \}.$$ for all $0<\epsilon<\epsilon_0$.

As $\ms F_2$ is a pile and $B\in \ms F_2$, $\ms F_2$
contains all subsets of $B$. Since $\ms F_1$ and $\ms F_2$ differ only
by the set $B$, $\ms F_1$ contains all subsets of $B$.  Thus, as $F$
belongs to $\ms K^{\ms F_1}$, by Lemma \ref{lem:pile}, there exists
$\delta'>0$, which only depends on $F$, so that for all
$\varnothing \subsetneq B'\subsetneq B$, and $x\in \Sigma$ with
$\|x\|_{{B'}^{c}} < \delta'$,
\begin{equation}
\label{06}
F(x) = [\gamma^*_{B'} (F|_{\Sigma_{B'}} ) ] (x)\,.
\end{equation}
        
We claim that there exists $\delta''>0$ such that for all $j\in B$,
$x\in \Sigma$ with $\|x\|_A < \delta''$ and $x_j \leq \delta''$, we
have the following properties:
\begin{enumerate}[leftmargin=1cm]
\item[($\delta''1$)] $\|x\|_{(B\setminus\{j\})^c} < \delta'$,
\item[($\delta''2$)]
$\|\gamma_B(x)\|_{(B\setminus \{j\})^c} < \delta'$.
\end{enumerate}

Indeed, fix $j\in B$, $x\in \Sigma$ with $\|x\|_A < \delta''$ and
$x_j \leq \delta''$. On the one hand,
\begin{equation*}
\|x\|_{(B\setminus\{j\})^c}
= \|x\|_{A \cup \{j\}} \leq \|x\|_A + \|x\|_{\{j\}} \leq 2\delta''\,.
\end{equation*}
On the other hand,
\begin{align*}
\|\gamma_B(x)\|_{(B\setminus\{j\})^c}
= \|\gamma_B(x)\|_{A \cup \{j\}} \leq \|x\|_{A \cup \{j\}}
+ \|\gamma_B(x)-x\|_{A \cup \{j\}}
\leq 2\delta'' + \|\gamma_B(x)-x\|_{A\cup \{j\}} \,.
\end{align*}
By  Lemma \ref{lem:RAlinmap},
$$
\|\gamma_B(x)-x\|_{A\cup \{j\}} \leq \|\gamma_B(x)-x\|
\leq  \sum_{i\in A} [L_A(x_A)]_i \|\bs v_i\| \leq c\,  \|x\|_A
\le c\, \delta''
$$
for some $c>0$ independent of $j$. To complete the proof of the claim
it remains to choose $\delta''$ so that $ (2+c) \delta'' <
\delta'$. We may assume that $\delta''\le \delta'$.

We turn to the proof of Assertion A.  Fix $j \in B$ and $x\in \Sigma$
such that $\|x\|_A < \delta''$, $x_j \leq \delta''$, and let
$B' = B\setminus \{j\}$. By the previous claim,
$\|x\|_{{B'}^c}, \|\gamma_{B}(x)\|_{{B'}^c} < \delta'$. Thus, by
\eqref{06} and Lemma \ref{lem:recursive}.
\begin{equation*}
F(x) = (\gamma^*_{B'} F|_{\Sigma_{B'}}) (x)
= F(\gamma_{B'}(x)) = F(\gamma_{B'}(\gamma_{B}(x)))
= (\gamma^*_{B'} F|_{\Sigma_{B'}}) (\gamma_{B}(x))\,.
\end{equation*}
Since $\|\gamma_{B}(x)\|_{{B'}^c}$ is also bounded by $\delta'$, by
\eqref{06}, 
\begin{equation*}
(\gamma^*_{B'} F|_{\Sigma_{B'}}) (\gamma_{B}(x))
= F(\gamma_{B}(x)) = (\gamma^*_B F|_{\Sigma_B}) (x)\,.
\end{equation*}
This proves Assertion A with $\delta = \delta''$ and $\epsilon_0$ small
to ensure that $\|x\|_A < \delta''$ if $\mf I_A(x)  < \epsilon_0$.

It follows from Assertion A and the first observation of the proof
that, for small enough $\epsilon>0$,
\begin{equation}
\label{29}
\text{supp}(F - F_\epsilon) \,\subset\,
\Lambda^B_{\delta'', 2\epsilon}\,.
\end{equation}
This proves the first statement of the Lemma.
        
We turn to the proof that $F_\epsilon$ belongs to $\ms K^{\ms F_2}$,
which is divided in several steps.  By definition of $F_\epsilon$ and
the first assertion of the lemma,
\begin{gather}
F_\epsilon = F \text{ on } \{x\in \Sigma : \mf I_A(x)
\geq 2\epsilon \text{ or } \min_{j\in B} x_j \leq \delta \},
\label{region1} \\
F_\epsilon = \gamma_B^* F|_{\Sigma_B}
\text{ on } \{\mf I_A (x) \leq \epsilon\}\,.
\label{region2}
\end{gather}

\smallskip
\noindent
\textbf{Step 1:  $F_\epsilon$ satisfies condition $\mathfrak{E2}$.}  By
Lemma \ref{app:6}, $F_\epsilon \in C^1(\Sigma)$.  Since $F\in\mc E_S$,
it remains to show that 
$(F - F_\epsilon)|_{\mathring{\Sigma}_C}\in
C^2_b(\mathring{\Sigma}_C)$ for all $C\subseteq S$ with $|C|\geq
2$. By definition of $F_\epsilon$,
\begin{equation}
\label{eq:FFepsilon}
F - F_\epsilon = (\chi_\epsilon
\circ \mf I_A) [F - \gamma^*_B (F|_{\Sigma_B})]\,.
\end{equation}

We consider two situations:

\vspace{1pt}
\noindent
(1-a) $A\cap C = \varnothing$.  In this case, $C\subseteq
B$. Therefore, $F = F_\epsilon$ on $\mathring{\Sigma}_C$ since
$\mf I_A = 0$ on $\mathring{\Sigma}_C$.  Therefore,
$(F - F_\epsilon)|_{\mathring{\Sigma}_C}\in
C^2_b(\mathring{\Sigma}_C)$.

\vspace{1pt}
\noindent
(1-b) $A\cap C \neq \varnothing$.  In this case, there exists
$i\in A\cap C$, so that $\mf I_A(x) > 0$ since $x_i \neq 0$ on
$\mathring{\Sigma}_C$.  Therefore, $\chi_\epsilon \circ \mf I_A$ is
smooth on $\mathring{\Sigma}_C$. On the other hand, by Lemma
\ref{lem:extregularity}, $F - \gamma^*_B (F|_{\Sigma_B})$ satisfies
condition $\mathfrak{E2}$. Thus,
$(F - F_\epsilon)|_{\mathring{\Sigma}_C}\in C^2(\mathring{\Sigma}_C)$.

It remains to bound the second derivatives of
$(F - F_\epsilon)|_{\mathring{\Sigma}_C}$ to show that it is in
$C^2_b(\mathring{\Sigma}_C)$. We cover the domain
$\mathring{\Sigma}_C$ with two open sets:
\begin{align*}
\mathring{\Sigma}_C = \{x\in \mathring{\Sigma}_C :
\mf I_A(x) < \epsilon\}
\cup \{x\in \mathring{\Sigma}_C : \mf I_A(x) > \epsilon/2\}.
\end{align*}
On the domain $\{x\in \mathring{\Sigma}_C : \mf I_A(x) < \epsilon\}$,
$F_\epsilon = \gamma^*_B (F|_{\Sigma_B})$ by the definition of
$F_\epsilon$. Lemma \ref{lem:extregularity} provides a bound for the
second derivatives of $F-F_\epsilon$ in this case.
        
On the domain
$\{x\in \mathring{\Sigma}_C : \mf I_A(x) > \epsilon/2\}$, \eqref{joh2}
provides a bound for the second derivatives of $\mf I_A$.  In
particular, a bound of the second derivatives of $F-F_\epsilon$ on
this set follows from \eqref{eq:FFepsilon} and Lemma
\ref{lem:extregularity}.

\noindent \textbf{Step 2: $F_\epsilon$ satisfies condition
$\mathfrak{E1}$.}  By \eqref{region1}, for $i\in B$, the map
$x \mapsto (1/x_i) \nabla_{\bs v_i} F_\epsilon$ is bounded.  We turn
to the case $i\in A$. As $F$ belongs to $\mc E_S$, it is enough to
show that the map 
\begin{equation}
\label{region3}
x \,\mapsto \, 
\frac{1}{x_i} \nabla_{\bs v_i}
(F-F_\epsilon) \;\text{ is bounded on }\; \mathring{\Sigma} \,. 
\end{equation}
By \eqref{eq:FFepsilon}, 
\begin{equation}
\label{eq8223}
\frac{1}{x_i} \nabla_{\bs v_i} (F - F_\epsilon)
= \frac{1}{x_i} (\nabla_{\bs v_i}(\chi_\epsilon \circ
\mf I_A))\,  [F - \gamma^*_B (F|_{\Sigma_B})]
+ (\chi_\epsilon \circ \mf I_A)
\frac{1}{x_i}\nabla_{\bs v_i} [F - \gamma^*_B (F|_{\Sigma_B})]\,.
\end{equation}
In order to bound the first term, write
\begin{equation*}
\nabla_{\bs v_i}(\chi_\epsilon \circ \mf I_A) =
(\chi_\epsilon' \circ \mf I_A) (\nabla_{\bs v_i}\mf I_A) =
(\mf I_A \cdot \chi_\epsilon'\circ \mf I_A)( \frac{1}{\mf I_A}
\nabla_{\bs v_i}\mf I_A)\,. 
\end{equation*}
The first term on the right-hand side of \eqref{eq8223} is thus equal
to
\begin{equation*}
(\mf I_A \cdot \chi_\epsilon'\circ \mf I_A)( \frac{\mf
I_A}{x_i} \nabla_{\bs v_i}\mf I_A)\Big[\, \frac{F -
\gamma^*_B (F|_{\Sigma_B})}{\mf I_A^2}\, \Big]\,.
\end{equation*}
By property ($\chi 2$), Lemma \ref{lem:normmapbdd}, property (b) of
Lemma \ref{johp}, and Lemma \ref{lem:estEA}.(1), this expression is
bounded.

We turn to the second term of \eqref{eq8223}. By definition of
$\gamma^*_B (F|_{\Sigma_B})$ and Lemma \ref{app:4},
$$ \frac{1}{x_i}\nabla_{\bs v_i} [F - \gamma^*_B
(F|_{\Sigma_B})] (x) = \frac{1}{x_i}\nabla_{\bs v_i} F (x) -
\frac{1}{x_i} \, \nabla_{\bs \gamma_B(\bs v_i)} F_B (\gamma_B (x)) \,
. $$
By \eqref{eq:projofvect}, as $i\in A$, $\gamma_B (\bs
v_i)=0$. Thus the second term vanishes. The first one is bounded
because $F\in \mc E_S$.

\noindent
\textbf{Step 3: $F_\epsilon$ is contained in $K^{\ms F_1}$.}  Fix
$C\in \ms F_1$, and let $D = S\setminus C$. We need to show that there
exists $\epsilon_C > 0$ such that
\begin{equation}
\label{28}
\nabla_{\bs v_j}F_\epsilon (x) = 0 \;\; \text{for all}\;\;
j \in D,\; x \in \Sigma \text{ with } \|x\|_D <\epsilon_C \,.
\end{equation}

We claim that $D\cap B \neq \varnothing$. Indeed, suppose, by
contradiction, that $D \cap B = \varnothing$, so that $B\subset
C$. Thus, $B\in \ms F_1$ because $C\in \ms F_1$ and $\mc F_1$ is a
pile, which is a contradiction with the hypothesis that
$B\not \in \ms F_1$.

It is enough to prove \eqref{28} for $F-F_\epsilon$.  By
\eqref{region1}, we have to show that there exists $\epsilon_C > 0$
such that
$$
\nabla_{\bs v_j} (F - F_\epsilon) (x) = 0  \;\; \text{for
all}\;\; j \in D\,, \, x \in \Lambda^B_{\delta, 2\epsilon} \text{ with
} \|x\|_D <\epsilon_C\,.
$$
Since $D\cap B \neq \varnothing$, taking $\epsilon_C = \delta$ yields
that
$$
\Lambda^B_{\delta, 2\epsilon} \cap \{\|x\|_D <\epsilon_C\}
=\varnothing\,.
$$
Because this set is empty, the condition naturally holds.
        
\noindent
\textbf{Step 4: $F_\epsilon$ is contained in $K^{\ms F_2}$.}  Recall from
Lemma \ref{johp}-(b), the definition of the constant $c_2$. Clearly,
$\mf I_A (x) \le \epsilon$ if $\| x\|_A \le \epsilon/c_2$.  Let
$\epsilon_B = \epsilon/c_2$. By \eqref{region2}, the choice of
$\epsilon_B$, and Lemma \ref{app:4}, on $\{\| x\|_A< \epsilon_B\}$, 
\begin{equation*}
\nabla_{\bs v_j} F_\epsilon(x) \,=\, \nabla_{\bs v_j}
\gamma^*_B (F|_{\Sigma_B}) (x)
= \nabla_{\bs \gamma_B(\bs v_j)} F_B (\gamma_B (x)) 
\end{equation*}
for all $j\in A$. This quantity vanishes because, by
\eqref{eq:projofvect}, $\gamma_B (\bs v_j) =0$ for $j\in A$. Thus,
$\nabla_{\bs v_j} F_\epsilon(x) = 0$ for all $j\in A$ and
$x\in \Sigma \text{ with } \|x\|_A < \epsilon_B$. This proves that
$F_\epsilon$ belongs to $K^{\ms F_2}$.
\smallskip

To complete the proof of the lemma, it remains to prove the assertions
(1)--(3).  The property (1) is obvious from the definition of
$F_\epsilon$.  By \eqref{29} and since $F_\epsilon=F$ on $\Sigma_B$, 
\begin{equation}
\label{eq:Lsupp}
\{x\in \Sigma : \mf L^{\mc E} F(x)
\neq \mf L^{\mc E} F_\epsilon(x) \} \subset \Lambda^B_{\delta,
2\epsilon} \setminus \Sigma_B.
\end{equation}
The property (2) follows from this fact.

It remains to show (3).  Let $G = F - \gamma^*_B F|_{\Sigma_B}$, so
that $G|_{\Sigma_B} = 0$. By Lemma \ref{lem:extregularity},
$G\in \mc E_A$.  By \eqref{eq:Lsupp} we only need to bound the term on
the domain $\Lambda^B_{\delta, 2\epsilon}\setminus \Sigma_B$ by some
constant independent of $\epsilon$.  Fix
$x\in \Lambda^B_{\delta, 2\epsilon}\setminus \Sigma_B$.  Let
$C = \{ i\in S : x_i \neq 0 \}$, so that $\mf I_A (x)>0$ and
$B\subsetneq C$.  By \eqref{eq:FFepsilon}, and the definition of
$\mf L^{\mc E}$ given in \eqref{eq:eop},
\begin{align*}
&\mf L^{\mc E} (F-F_\epsilon)(x)
= \mf L^{\mc E} ((\chi_\epsilon \circ \mf I_A) G)(x)
= \mf L^C ((\chi_\epsilon \circ \mf I_A) G)(x) \,.
\end{align*}
By  \eqref{eq:LB}, this expression is equal to
\begin{align}
\mf L^{C} (\chi_\epsilon \circ \mf I_A) \cdot G(x) +
\mf L^{C} G \cdot (\chi_\epsilon \circ \mf I_A)(x)
+ \sum_{i,j\in C} m_i r^C(i,j)
(\partial_{x_i} - \partial_{x_j}) (\chi_\epsilon \circ \mf I_A)
(\partial_{x_i} - \partial_{x_j}) G(x)\,.
\label{eq:expansion}
\end{align}
        
By property $(\chi2)$ and Lemma \ref{lem:estEA}-(2), the third term in
\eqref{eq:expansion} is bounded in
$\Lambda^B_{\delta, 2\epsilon}\setminus \Sigma_B$, uniformly in
$\epsilon >0$. We turn to the second. Note that
\begin{equation*}
(\mf L^{C} G)(x)
= b \,  \sum_{i\in A\cap C}
\frac{m_i (\nabla_{\bs v^C_i} G)(x)  }{x_i}
+ b \,  \sum_{j\in B}
\frac{m_j(\nabla_{\bs v^C_j} G)(x)  }{x_j}
+ \frac{1}{2} \sum_{i,j\in C}  m_i\, 
r^C(i,j) \,  [(\partial_{x_i} - \partial_{x_j}) G (x) ]^2\,.
\end{equation*}
By Lemma \ref{lem:resvan}, the first term is equal to
$$
b \,  \sum_{i\in A\cap C} \frac{ m_i [\nabla_{\bs v^C_i}
(F-\gamma_B^*F|_{\Sigma_B}) ](x)}{x_i}
= b \,  \sum_{i\in A\cap C}
\frac{ m_i (\nabla_{\bs v^C_i} F)(x)  }{x_i},
$$
which is bounded by the fact that $F \in \mc E_S$. On the other hand,
as $F\in C^1(\Sigma)$, on the set $\Lambda^B_{\delta, 2\epsilon}$,
there exists a finite constant $C_0$ such that
\begin{equation}
\label{eq:Bbdd}
b \,  \sum_{j\in B} \frac{m_j(\nabla_{\bs v^C_j} G)(x) }{x_j}
\leq \frac{C_0}{\delta} \; \text{ on } \;
\Lambda^B_{\delta, 2\epsilon}.
\end{equation}
The last term in the decomposition of $\mf L^{C} G $ is clearly
bounded. This proves that the second term in \eqref{eq:expansion} is
bounded in $\Lambda^B_{\delta, 2\epsilon}$, uniformly in $\epsilon>0$.
    
It remains to consider the first term of \eqref{eq:expansion}. It is
equal to
\begin{equation*}
(\chi_\epsilon' \circ \mf I_A) (x)\, (\mf L^C \mf I_A)(x)\,  G(x)
+ \frac{1}{2} \, (\chi_\epsilon''\circ \mf I_A)(x)\, 
\sum_{i,j\in C} m_i \, r^C(i,j) \,
[(\partial_{x_i}-\partial_{x_j})\mf I_A]^2 \, G(x) \,.
\end{equation*}
Since $G\in \mc E_A$ and $G|_{\Sigma_B}=0$, by Lemma
\ref{lem:estEA}-(1), we may rewrite this sum as
\begin{align*}
& [\, \mf I_A\, (\chi_\epsilon'\circ \mf I_A) (x) \,] \,
[\, \mf I_A (\mf L^C \mf I_A) (x) \,]  \, H(x)
\\
&\quad + [(\chi_\epsilon''\circ \mf I_A) \, \mf I_A^2 (x)]\, 
\frac{1}{2} \, \sum_{i,j\in C}  m_i \, r^C(i,j)\,
[(\partial_{x_j} - \partial_{x_i})\mf I_A (x)]^2 H(x)
\end{align*}
for some $H\in BC(\Sigma,\Sigma\setminus \Sigma_B)$.  By property
$(\chi3)$ and \eqref{joh2}, the second term is bounded in
$\Lambda^B_{\delta, 2\epsilon}\setminus \Sigma_B$, uniformly in
$\epsilon >0$.

We turn to the first. By property $(\chi2)$,
$\mf I_A (x) \, (\chi_\epsilon'\circ \mf I_A) (x)$ is uniformly
bounded in $\epsilon>0$. Fix $i\in C$. By \eqref{eq:projofvect},
$\nabla_{\bs v^C_i}\mf I_A = \nabla_{\gamma_C(\bs v_i)} \mf I_A
=\nabla_{\bs v_i} \mf I_A + \nabla_{\gamma_C(\bs v_i) - \bs v_i} \mf
I_A$. By Lemma \ref{lem:lincomb}, $\gamma_C(\bs v_i) - \bs v_i$ is a
linear combination of the vectors $\bs v_k$, $k\in C^c$. By Lemma
\ref{lem:normmapvanish}, $(\nabla_{\bs v_k}\mf I_A )(x) =0$ for all
$k\not\in C$. This implies
$\nabla_{\bs v^C_i} \mf I_A = \nabla_{\bs v_i} \mf I_A$. Thus,
\begin{align*}
\mf I_A (x)\, (\mf L^C \mf I_A) (x) &
= b \,  \sum_{i\in C}
\frac{ m_i (\nabla_{\bs v^C_i}\mf I_A)(x)}{x_i}
\, \mf I_A (x)
+ \frac{1}{2} \sum_{i,j\in C} m_i \, r^C(i,j)\,
[(\partial_{x_i}- \partial_{x_j})^2 \mf I_A (x)] \, \mf I_A(x) \\
&=
b \,  \sum_{i\in C}
\frac{ m_i (\nabla_{\bs v_i}\mf I_A)(x)}{x_i}
\, \mf I_A (x)
+ \frac{1}{2} \sum_{i,j\in C} m_i \, r^C(i,j)\,
[(\partial_{x_i}- \partial_{x_j})^2 \mf I_A (x)] \, \mf I_A(x)\,.
\end{align*}
By \eqref{joh2}.  the second term is bounded. The first one can
be rewritten as
$$
b \,  \sum_{i\in A\cap C}
\frac{ m_i (\nabla_{\bs v_i}\mf I_A)(x)}{x_i}
\, \mf I_A (x)
+ b \,  \sum_{i\in B}
\frac{ m_i (\nabla_{\bs v_i}\mf I_A)(x)}{x_i}
\, \mf I_A (x) \,.
$$
By Lemma \ref{lem:normmapbdd}, the first sum is bounded. The second
one can be estimated with the same arguments used for \eqref{eq:Bbdd}
using the bound \eqref{joh2}. This completes the proof of the lemma.
\end{proof}

\begin{proof}[Proof of Proposition \ref{prop:piledense}]
We claim that $F_\epsilon \rightarrow F$ in the $\mu$-topology.
By Lemma \ref{lem:FFe}, 
$$
\int_\Sigma \mf L^{\mc E} F_\epsilon \;d\mu
\rightarrow \int_\Sigma \mf L^{\mc E} F \; d\mu.
$$
It remains to show $F_\epsilon \to F$ in $L^\infty$.
By \eqref{eq:FFepsilon} and the definition of $\chi_\epsilon$, 
$$
|F-F_\epsilon|_\infty
\le \sup_{\mf I_A(x) \le 2\epsilon} |F- \gamma^*_B (F|_{\Sigma_B})|\,.
$$
As $F$ is continuous, the right-hand side converges to $0$ as
$\epsilon \rightarrow 0$. This completes the proof of the claim.

By Lemma \ref{lem:FFe}, $F_\epsilon \in \ms K^{\ms F_2}$ for small
enough $\epsilon>0$. Therefore, $\ms K^{\ms F_2}$ is $\mu$-dense in
$\ms K^{\ms F_1}$.
\end{proof}

\begin{corollary}
For all $\mu \in \mc M(\Sigma)$, $\ms K^{2^S}$ is $\mu$-dense in
$\mc E_S$.
\end{corollary}

\begin{proof}
As $2^S$ is the collection of all subsets of $S$, $2^S$ is a pile.  By
definition, $\ms K^{\varnothing} = \mc E_S$.  The assertion is thus a
consequence of Proposition \ref{prop:piledense}.
\end{proof}

\begin{proof}[Proof of Proposition \ref{prop:mudense}]
By Lemma \ref{lem:2SsubsetD}, $\ms K^{2^S} \subset \mc D_S$.
Thus, the statement follows from the previous corollary.
\end{proof}

\section{Proof of Theorem \ref{thm:abs}}
\label{sec:proofthmabs}

In this section we show that any solution of the
$(\mf L^{\mc E}, \mc E_S)$ martingale problem is absorbed at the
boundary. In the first subsection, we introduce the natural candidates
(in view of the form of the generator) to prove absorption. More
precisely, a family of positive functions which are superharmonic away
from the boundary.  These natural candidates do not belong to the
domain $\mc E_S$ and need to be regularized close to the boundary. In
the following two subsection by using the regularized versions we
prove absorption by considering the associated Dynkin's martingales.

\subsection{A superharmonic function}

As the title suggest, we construct in this subsection a superharmonic
function in three steps starting from the functions
$F_A\colon \Sigma \to \bb R_+$ introduced below in \eqref{31}, which
does not belong to the domain $\mc E_S$. At each step we improve the
regularity of the function keeping its essential properties, until
obtaining in Lemma \ref{lem:funcext} a function in $\mc E_S$.  As the
function $F_A$, introduced in \eqref{31}, belongs to $\mc E_A$, it
satisfies condition $\mf E_2$. Therefore, $F_A|_{\mathring{\Sigma}_D}$
belongs to $C^2_b(\mathring{\Sigma}_D)$ for all $D\subset S$ with
$|D|\ge 2$. In consequence, $\mf L^{\mc E} F_A$ is well defined as a
differential operator.

\begin{lemma}
\label{lem:supharm}
Fix a proper nonempty subset $B$ of $S$ and let $A = B^c$. For
$\gamma \in (0,1)$, let $F_A\colon \Sigma \to \bb R_+$ be given by
\begin{equation}
\label{31}
F_A(x) = \prod_{k\in A} x_k^{1+b}  (1-x_k^{\gamma}), \quad x\in \Sigma\,.
\end{equation}
Then, $F_A$ belongs to $\mc E_A$. Moreover, for each nonempty subset $D$
of $B$ and $\epsilon>0$, there exists
$\lambda^{A\cup D,A}(\epsilon)>0$ such that
$$
\mf L^{\mc E} F_A(z) \leq 0 \;\; \text{ for all } \;\;\ z\in \Big\{ x \in
\Sigma_{A\cup D}
: \max_{k\in A} x_k \leq \lambda^{A\cup D, A}(\epsilon) \text{ and }
\min_{i\in  D} x_i \geq \epsilon \Big\}\,.
$$
\end{lemma}

\begin{proof}
We start verifying that the conditions $\mathfrak{E1}(A)$ and
$\mathfrak{E2}$ hold.  For $\mathfrak{E1}(A)$, let $ i \in A$. By
assumption, $x_i^2$ divides $F_A(x)$, i.e., $F_A(x)=x_i^2 G(x)$ for
some smooth function $G$ Therefore, the quotient
$F_A(x)/x_i^2= G(x)$ remains bounded near $x_i=0$, so the map in
\eqref{eq:condC1} is bounded.

For $\mathfrak{E2}$, we note that $F_A \in C^2(\Sigma)$, by
assumption. Thus, all second derivatives exist and are continuous,
satisfying the smoothness requirements of $\mathfrak{E2}$.
    
It remains to show the existence of $\lambda^{A\cup D, A}(\epsilon)>0$ with
the given property.  Decompose $\Sigma_{A\cup D}$ as
$$
\Sigma_{A\cup D} \,=\, \bigcup_{\substack{C\subset A\cup D, \\ |C|\geq 2}}
\mathring{\Sigma}_C \cup \bigcup_{i\in A\cup D} \Sigma_{\{i\}}\,.
$$
Since $\mf L^{\mc E}$ vanishes on $\Sigma_{\{i\}}$, it is enough to
show that for all $C \subset A\cup D$ with $|C|\geq 2$, there exists
$\lambda^{A\cup D, A}_C(\epsilon)>0$ such that
\begin{equation}
\label{eq:supharm2}
\mf L^{\mc E} F_A(x) \leq 0 \;\;
\text{ on } \;\;\left\{ x \in \mathring{\Sigma}_C : \max_{k\in A} x_k
\leq \lambda^{A\cup D, A}_C (\epsilon) \text{ and } \min_{i\in  D} x_i \geq
\epsilon \right\},
\end{equation}
and then set
$\lambda^{A\cup D, A}(\epsilon)
= \min_{\substack{C\subset A\cup D \\ |C|\geq 2}}
\lambda^{A\cup D, A}_C(\epsilon)$.

Existence of $\lambda^{A\cup D, A}_C(\epsilon)$ is obvious if
$A \nsubseteq C$ since in this case $F_A(x) = 0$ for
$x\in \mathring{\Sigma}_C$.  Assume that $A \subseteq C$, and compute
$\mf L^{\mc E} F_A(x)$ for $x\in \mathring{\Sigma}_C$. By definition,
\begin{equation}
\label{eq:supharm1}
\mf L^{\mc E} F_A(x)=\mf L^C
F_A(x)=\sum_{j\in A} \sum_{i\in C} {m}_i \, {\bs v}^C_i\cdot {\bs
e}_j\left( b \,  \frac{\partial_{x_j} F_A(x)}{x_i} - \partial_{x_i} \partial_{x_j}
F_A(x) \right) \,.
\end{equation}
Rewrite this expression as $\sum_{j\in A} I_j(x) +\sum_{j\in A}
L_j(x)$, where 
\begin{equation*}
I_j(x) = {m}_j\, {\bs v}^C_j\cdot {\bs e}_j
\left( b \,  \frac{\partial_{x_j} F_A(x)}{x_j}
- \partial^2_{x_j} F_A(x) \right)
+ b \,  \sum_{i\in C\setminus A} {m}_i{\bs v}^C_i\cdot {\bs e}_j
\left( \frac{\partial_{x_j} F_A(x)}{x_i}\right)
\end{equation*}
and
\begin{equation*}
L_j(x) = \sum_{i\in A\setminus \{j\}} {m}_i\,
{\bs v}^C_i\cdot {\bs e}_j\left( b \,  \frac{\partial_{x_j} F_A(x)}{x_i}
- \partial_{x_i} \partial_{x_j} F_A(x) \right).
\end{equation*}

By \eqref{eq:supharm1}, it is enough to show that there exists
$\lambda^{A\cup D, A}_C(\epsilon)>0$ such that $I_j(x)\le 0$ and
$L_j(x)\le 0$ on the domain in \eqref{eq:supharm2}.  To define
$\lambda^{A\cup D, A}_C(\epsilon)$, fix a constant $M>0$ such that
\begin{equation}
\label{eq:supharm3} 0\leq
-\left(\frac{1+b} {\gamma\, (\gamma+b+1)} \right)
\sum_{i\in C\setminus A} \frac{{m}_i
\, {\bs v}^C_i \cdot {\bs e}_j}{{m}_j\, {\bs v}^C_j \cdot
{\bs e}_j} \leq M \quad\text{for all}\;\; j\in A\,.
\end{equation}
Let $\lambda^{A\cup D, A}_C(\epsilon)>0$ be such that
\begin{equation}
\label{eq:supharm4}
0\leq x \leq
\lambda^{A\cup D, A}_C(\epsilon) \;\; \text{ implies } \;\; Mx^{1-\gamma}
\leq \epsilon \;\; \text{ and } \;\; (\gamma+1) \, x^\gamma \leq 1.
\end{equation}

For simplicity, let $f(x) = x^{1+b} (1-x^{\gamma})$. On the one hand,
\begin{equation*}
I_j(x) = {m}_j\, {\bs v}^C_j\cdot {\bs e}_j \left(\gamma(\gamma+b+1)
\, x_j^{\gamma+b-1}  \prod_{k\in A\setminus\{j\}} f(x_k)\right) \left(
1 + \left(
\frac{f'(x_j)\,x_j^{-\gamma- b+1}}
{\gamma(\gamma+b+1)}\right)\sum_{i\in
C\setminus A}\left(\frac{1}{x_i}\right) \frac{{m}_i \, {\bs v}^C_i
\cdot {\bs e}_j}{{m}_j \, {\bs v}^C_j \cdot {\bs e}_j} \right).
\end{equation*}
By using \eqref{eq:supharm3} and \eqref{eq:supharm4} and the fact that
$f'(x)\le (1+b) \, x^b$, on the set introduced in \eqref{eq:supharm2}
\begin{equation*}
-\left( \frac{f'(x_j)\,x_j^{-\gamma-b+1}}{\gamma(\gamma+b+1)}\right)
\sum_{i\in C\setminus A}\left(\frac{1}{x_i}\right)
\frac{{\bs m}_i {\bs v}^C_i \cdot
{\bs e}_j}{{\bs m}_j {\bs v}^C_j \cdot {\bs e}_j}
\le \frac{M\,x_j^{1-\gamma}}{\epsilon} \le 1.
\end{equation*}
This shows that $I_j(x)\le 0$ because ${\bs v}^C_j\cdot {\bs e}_j
<0$. On the other hand,
\begin{equation*}
L_j(x) = \sum_{i\in A \setminus \{j\}} {m}_i \, {\bs v}^C_i \cdot
{\bs e}_j \,
f'(x_j)\left( b\, \frac{f(x_i)}{x_i} - f'(x_i) \right) \,
\prod_{k\in A \setminus \{i,j\}} f(x_k),
\end{equation*}
where the last product equals one if
$A\setminus\{i,j\}=\varnothing$. Finally, as $x_j\le 1$,
\eqref{eq:supharm4} assures that
\begin{equation*}
f'(x)\ge 0 \quad {\rm and} \quad \frac{f(x)}{x} - f'(x) \le 0.
\end{equation*}
This completes the proof of the lemma.
\end{proof}
    
Let $\pi_D:\bb R^S \to \bb R$ be the map defined as
$$
\pi_D(x) = \prod_{k\in D} x_k, \quad x\in \bb R^S\,.
$$
Note that
\begin{equation}
\label{32}
(\mf L^{\mc E} F_D) (x) \,=\,0 \quad\text{for all $x\in \Sigma$ such
that $\pi_D(x) =0$}\,.
\end{equation}
Indeed, fix $x\in \Sigma$ such that $\pi_D(x) =0$. Thus $x_i=0$ for
some $i\in D$. Let $C=\{j\in S : x_j >0\}$ so that $i\not\in C$. By
the definition \eqref{eq:eop} of the generator $\mf L^{\mc E}$,
$(\mf L^{\mc E} F_D)(x) = [\mf L^{C} (F_D\big|_{\mathring{\Sigma}_C} )
] (x)$. By \eqref{eq:LB}, the variables $x_\ell$, $\ell\not \in C$,
act as constants for the operator $\mf L^{C}$. Thus, as $x_i=0$,
$[\mf L^{C} (F_D\big|_{\mathring{\Sigma}_C} ) ] (x) =0$, as claimed.

Note that the function $F_{A,D}$ introduced below in \eqref{07} is not
the function $F_D$ defined in \eqref{31} because their domain are
different. 

\begin{lemma}
\label{lem:auxfunc}
Fix a nonempty subset $B$ of $S$ and let $A = B^c$.  For
$\varnothing \subsetneq D \subsetneq B \subset S$ and
$\gamma \in (0,1)$, let $F_{A,D}:\Sigma_B \to \bb R$ be
\begin{equation}
\label{07}
F_{A,D}(x) = \prod_{k\in D} x_k^{1+b} (1-x_k^{\gamma}), \quad x\in \Sigma_B.
\end{equation}
Then there exists a function $F\colon \Sigma \to \bb R$ in
$\mc E_{A\cup D}$ satisfying $F(x) = F_{A,D}(x), x\in \Sigma_B,$ and
$\mf L^{\mc E} F(x) =  0$ for all $x\in \Sigma$ with $\pi_D(x) = 0$.
\end{lemma}

\begin{proof}
The proof presented here is the extension of the Lemma 4.3 from
\cite{BJL}.  For the sake of completeness, we provide a detailed
construction.  If $A=\varnothing$, $F$ has to be equal to $F_{A,D}$
since $\Sigma_B=\Sigma$. Moreover, $F_{A,D}=F_D$, where $F_D$ has been
introduced in the previous lemma.  By this result, $F_{D}$ belongs to
$\mc E_D$ and, by \eqref{32}, $\mf L^{\mc E} F_{D}(x) = 0$ for all
$x\in \Sigma$ such that $\pi_D(x) = 0$. This shows that the assertion
of the lemma holds if $A=\varnothing$.

Asume, from now on, that $A$ is nonempty.  Let
$$
\cb{\beta}\,:=\,  \frac{4}{c_1}\,, 
$$
where $c_1$ is the constant given in property $b)$ of Lemma
\ref{johp}. This choice of $\beta$ is made so that the following
inequality holds: If $x_j < \|x\|_A$ for some $j\in D$, then
\begin{equation}
\label{bddR}
\frac{\beta \mf I_A(x)}{[\gamma_B(x)]_j}
\ge \frac{\beta c_1 \|x\|_A}{x_j
+ \sum_{k\in A} u^B_j(k) x_k}
\ge \frac{\beta c_1 \|x\|_A}{x_j + \|x\|_A} > 2\,.
\end{equation}

Assume that the function $F_{A,D}$ introduced in \eqref{07} is defined in
$\bb R^S$, and not only on $\Sigma_B$. Let $\mc V$ be the open subset
given by
$\color{blue} \mc V := \{ x\in \Sigma : \pi_D(\gamma_B(x)) > 0\}$, and
let $\cb{\Psi} \colon \bb R \to \bb R$ be a non-increasing function in
$C^2(\bb R)$ which is equal to 1 on $(-\infty, 0]$ and 0 on
$[1, \infty)$. Denote by $R\colon \mc V \to \bb R_+$ the function
given by
$$
\cb{R(x)} \,:=\,  \beta \, \frac{\mf I_A(x)}{ \pi_D(\gamma_B(x)) }\,\cdot
$$
By \eqref{bddR},
\begin{equation}
\label{bddR1}
R(x) \,>\, 2 \quad\text{if $x_j < \|x\|_A$ for some $j\in D$}\;.
\end{equation}

Denote by $F: \bb R^S \to \bb R$ the function given by
$$
\cb{ F(x)} \,:=\,
\begin{cases}
F_{A,D} (\gamma_B (x)) \, 
\Psi (R(x)-1) \,,  & x\in \mc V  \,, 
\\
0, & \text{ otherwise },
\end{cases}
$$
where $\gamma_B:\bb R^S \to \bb R^S$ has been introduced in
\eqref{eq:proj} and $\mf I_A$ in \eqref{mfi}.  We assert  that $F$
restricted to $\Sigma$ satisfies the conditions of the lemma. The
proof of this statement is divided in a series of claims.

\smallskip \noindent{\it Claim A:} For $x\in \Sigma$, let
$C = C_x = \{i\in S: x_i > 0\}$. If $x\notin \mc V$, then
$\Sigma_C \cap \mc V = \varnothing$. \smallskip

Fix $x\in \Sigma \setminus \mc V$. Clearly
$x \in \mathring{\Sigma}_C$. Let $i\in D$. Since
$D\subset B$, by Lemma \ref{lem:splxrest}, $[\gamma_B(y)]_i = 0$ for
all $y\in \mathring{\Sigma}_C$ if, and only if,
$[\gamma_B(x)]_i=0$. Therefore, $x\notin \mc V$ implies
$ \mathring{\Sigma}_C\subset \mc V^c$. Since $\Sigma_C$ is
the closure of $\mathring{\Sigma}_C$ and $\mc V^c$ is closed, 
Claim A is proved. 

\smallskip \noindent{\it Claim B:}   $F$ belongs to
$C^1(\Sigma)$. \smallskip 

The proof of this assertion is identical to the one of \cite[Lemma
4.3.A]{BJL}.  We first prove that for $x \in \Sigma \setminus \mc V$,
$F$ is differentiable at $x$ and $\nabla F (x) = 0$. To prove this, it
is enough to show that there exists $C>0$ such that
\begin{equation}
\label{eq:diffF}
|F(x) - F(y)| \leq C \, \|x-y\|^2
\;\; \text{ for all } x\in \Sigma \setminus \mc V, y\in \bb R^S \text{
with } \|x-y\| \leq 1\,.
\end{equation}
Recall that $x \in \Sigma \setminus \mc V$.  By the definition of $F$,
if $y \in \Sigma \setminus \mc V$, $F(y) = F(x)=0$. Assume that
$y\in \mc V$. In this case,
$$ |F(x) - F(y)| = |F(y)| = |\,F_{A,D}(\gamma_B(y)) \,
\Psi(R(y)-1)\,| = \pi_D(\gamma_B(y))^2 \Phi(y),
$$
for some continuous function $\Phi$. Since
$$
\pi_D(\gamma_B(y))^2  \,=\, [\, \pi_D(\gamma_B(y))-
\pi_D(\gamma_B(x)) \,]^2 \, ,
$$
we easily obtain \eqref{eq:diffF}.

The functions $\Psi(R-1)$ and $F$ are in $C^2(\mc V)$. In particular,
to prove that $F$ belongs to $C^1(\Sigma)$, it only remains to examine
the behavior of the derivatives of $F$ close to the boundary of
$\mc V$.
    
We claim that there exists a finite constant $C_1>0$ such that
\begin{equation}
\label{08}
\|\nabla F(x)\| \leq C_1\,  \big\{\, \|\nabla F_{A,D}(w)\|
+ \pi_D(w) \,\big\}, \quad x\in \mc V,
\end{equation}
where $\cb{w = \gamma_B(x)}$.

Recall the definition of the functions $u^B_k$, $k\in B$, introduced
in \eqref{10}. An elementary computation yields that for any smooth
function $G\colon \bb R^B\to \bb R$, $j\in S$,
\begin{equation}
\label{09}
\partial_{x_j}   (G \circ \gamma_B) (x) =
\sum_{k\in B} (\partial_{x_k}   G) ( \gamma_B (x)) \,
u^B_k(j)\,,\;\; \text{so that}\quad  (\nabla G)(x) =  u^{B,*} \,
\nabla G (\gamma_B (x) )\,,
\end{equation}
where $u^{B,*}$ is the matrix given by $u^{B,*}(j,k) = u^B_k(j)$,
$k\in B$, $j\in S$.
    
Let
$$
\cb{\mc W_0}\, :=\, \{ x\in \mc V : R(x) < 1 \}\,, \quad 
\cb{\mc W_1} \,:=\,  \{ x\in \mc V : \|x\|_A>0\}.
$$
By definition of $R(\cdot)$, $R(x) =0$ if $\|x\|_A=0$. Thus,
$\mc V = \mc W_0 \cup \mc W_1$, and it is enough to show that
\eqref{08} holds in each set $\mc W_0$, $\mc W_1$.

In $\mc W_0$, $\Psi(R-1)\equiv 1$. Hence, by definition of
$\Psi(\cdot)$, $F(\cdot)$, for $j\in S$, by \eqref{09},
\begin{equation}
\label{11}
(\nabla F)(x) =  u^{B,*} \, \nabla F_{A,D} (w)\,.
\end{equation}
This proves \eqref{08} on $\mc W_0$.

On $\mc W_1$, we calculate $\nabla \Psi(R-1)$. By \eqref{09}, it is
given by
\begin{align}
\label{eq:FC1}
\nabla\Psi(R-1)(x) = \Psi'(R-1)(x) \left[
-\, R(x) \, \frac{ u^{B,*}\nabla \pi_D(w)}{\pi_D(w)} + \beta \nabla
\mf I_A(x) \frac{1}{\pi_D(w)}\right].
\end{align}
Since $\Psi(R-1)\equiv 0$ on $R>2$, $|\Psi'(R-1)(x) R(x) |$ is
bounded, uniformly in $x\in \mc V$. Thus, by \eqref{joh2},
\begin{equation}
\label{Psi1}
\|\nabla \Psi(R-1) (x)\| \leq
\frac{C}{\pi_D(w)}\;\; \text{ on }\;\; \mc W_1 \,.
\end{equation}
Hence, by definition of $F_{A,D}$,
$\| F_{A,D} (x)\, \nabla \Psi(R-1) (x)\| \leq C \pi_D(w)$ for
$x\in \mc W_1$. As $\Psi$ is bounded, \eqref{08} also holds on
$\mc W_1$.
    
By \eqref{08} and the definition of $F_{A,D}$, $(\nabla F)(x)$ converges
to $0$ as $x$ approaches the boundary of $\mc V$.  Therefore,
$\nabla F$ is a well-defined continuous function on $\Sigma$. In
particular, $F$ is $C^1(\Sigma)$.

\smallskip \noindent{\it Claim C:} $F$ belongs to $C^2(\mc V)$ and has
bounded second derivatives. \smallskip

We claim that for all $j, k \in S$, and all $x\in \mc V$,
$$\|(\partial^2_{x_j x_k})F(x)\| \leq C_0$$
for some finite constant $C_0>0$. This is clear on the set $\mc W_0$
because $\Psi\equiv 1$ on $\mc W_0$. Taking a second derivative in
\eqref{eq:FC1} yields that
\begin{equation}
\label{Psi2}
|\partial^2_{x_j, x_k} \Psi(R-1)(x)|\leq
\frac{C}{\pi_D(w)^2}\;\; \text{ on } \mc W_1.
\end{equation}
The claim follows from \eqref{Psi1}, \eqref{Psi2} and a
straightforward computation.

\smallskip \noindent{\it Claim D:} $F$ satisfies condition
$\mathfrak{E2}$. \smallskip
    
Fix $C\subset S$ with $|C|\geq 2$.  By Claim A,
$$
\mathring{\Sigma}_C \subset \mc V \text{ or } \Sigma_C
\cap \mc V = \varnothing.
$$
If $\Sigma_C \cap \mc V = \varnothing$, then $F(x) = 0$ for all
$x\in \Sigma_C$, and condition $\mathfrak{E2}(C)$ holds.  If
$\mathring{\Sigma}_C \subset \mc V$, then condition $\mathfrak{E2}(C)$
follows from Claim C.  This proves Claim D.

\smallskip \noindent{\it Claim E:}
$F$ satisfies condition $\mathfrak{E1}(j)$ for all $j\in D$. \smallskip 

The proof is identical to the one of Lemma 4.3.D in \cite{BJL}.  By
definition, $\nabla F (x) = 0$ for $x\in \Sigma\setminus \mc V$. We
turn to the set $\mc V$.  Fix $j\in D$. By \eqref{bddR1}, and the
definition of $R(\cdot)$,
$$
F\equiv 0 \text{ on the open subset } \{x\in \Sigma : x_j <
\|x\|_A\}.
$$
Thus, by \eqref{08}, there exists a finite constant $C_1$ such that 
$$
| \bs v_j \cdot \nabla F (x) | \,\le\,  C_1 \, \big\{\,
\|\nabla F_{A,D}(w)\|_B + \pi_D(w) \,\big\}\, \mtt  1\{\|x\|_A \le x_j\},
\quad x \in \mc V\, .
$$ 
For $x\in \mc V$ with $\|x\|_A \le x_j$,
$$ w_j \,:=\, [\gamma(x)]_j \le x_j + \|x\|_A \,\le\,  2 \, x_j\, .$$
Therefore,
\begin{align*}
\frac{|\, \nabla_{\bs v_j} F(x)\,|}{x_j}
\,\le \, C_1 \, \left(\frac{\|\nabla F_{A,D}(w)\|_B}{w_j} +
\frac{\pi_D(w)}{w_j} \right)
\,\le\,  C_1
\end{align*}
for some finite constant $C_1$ whose value may have changed from line
to line. This proves Claim E.

\smallskip \noindent{\it Claim F:}   The fucntion $F$ satisfies condition
$\mathfrak{E1}(i)$ for all $i\in A$. \smallskip 

Fix $k \in A$.  By \eqref{09},
\begin{equation*}
\nabla_{\bs v_k} (F_{A,D}\circ \gamma_B) (x) \,=\,
\bs v_k \cdot \nabla  (F_{A,D}\circ \gamma_B) (x) \,=\,
\bs v_k \cdot u^{B,*} \nabla  F_{A,D} (\gamma_B (x) )\,.
\end{equation*}
Thus, by \eqref{12} and \eqref{eq:projofvect}, as
$k\in A$,
\begin{equation}
\label{13}
\nabla_{\bs v_k} (F_{A,D}\circ \gamma_B) (x) \,=\,
\gamma_B(\bs v_k) \cdot \nabla  F_{A,D} (\gamma_B (x) )
\,=\, 0\,.
\end{equation}

Recall the definition of the sets $\mc W_0$, $\mc W_1$ introduced in
Claim B.  On the open set $\mc W_0$, $F(x) = F_{A,D} (\gamma_B  (x))$. Hence,
by \eqref{13}, on this set, $\nabla_{\bs v_k} F =0$.
We turn to the set $\mc W_1$. By \eqref{13}, $\nabla_{\bs v_k}
(F_{A,D}\circ \gamma_B) (x) = \nabla_{\bs v_k} (\pi_D\circ \gamma_B)
(x)=0$, so that by \eqref{eq:FC1},
\begin{align*}
(\bs v_k \cdot \nabla F)(x)
\,=\,  & F_{A,D} (\gamma_B(x)) \,
\Psi'(R(x)-1) \, \Big\{\, \beta\, \nabla_{\bs v_k} \mf I_A (x)
\frac{1}{\pi_D(\gamma_B(x))}\,\Big\}
\\
=\; & \frac{\beta^2}{R(x)} \, \Psi'(R(x)-1)\,
\frac{F_{A,D}(\gamma_B(x))} {\pi_D(\gamma_B(x))^2}\, 
(\nabla_{\bs v_k}\mf I_A)  (x) \, \mf I_A(x)\,.
\end{align*}
Therefore,
$$
\frac{(\nabla_{\bs v_k } F)(x)}{x_k} \,=\,
\frac{\beta^2}{R (x)} \, \Psi'(R(x)-1)\, \frac{F_{A,D}(\gamma_B(x))}
{\pi_D(\gamma_B(x))^2}\,
\frac{\nabla_{\bs v_k}\mf I_A (x)}{x_k} \, \mf I_A(x)\,.
$$
By definition of $F_{A,D}$, $F_{A,D}(\gamma_B(x))/\pi_D(\gamma_B(x))^2$ is
bounded in $\Sigma$. By Lemma \ref{lem:normmapbdd} and Lemma
\ref{johp}-(b). $(\nabla_{\bs v_k} \mf I_A (x) /x_k)\, \mf I_A(x)$ is
bounded in $\mc W_1$. Finally, $R (x)^{-1} \, \Psi'(R(x)-1)$ is
bounded. This completes the proof of the claim.

\smallskip \noindent{\it Claim G:}   $\mf L^{\mc E} F = 0$ if $\pi_D(x) = 0$. \smallskip 

Fix $x\in \Sigma$ such that $\pi_D(x) = 0$. Let
$C = \{ i\in S: x_i \neq 0\}$, so that $x\in \mathring{\Sigma}_C$.  If
$\|x\|_A = 0$, $C\subset B$, and $\pi_D(\gamma_B(x)) = \pi_D(x) =
0$. In particular, $x\not\in\mc V$.  By Claim A,
$\Sigma_C \subset \Sigma \setminus \mc V$. Thus, by definition, $F$
vanishes on $\Sigma_C$. Therefore,  $(\mf L^{\mc E} F)(x) =  (\mf
L^{\mc C} F|_{\mathring{\Sigma}_C} )(x) = 0$.  

Suppose that $\|x\|_A \neq 0$. As $\pi_D(x) = 0$, $x_i=0$ for some
$i\in D$. Thus, $x_i < \|x\|_A$ and, by \eqref{bddR1}, $R(x)>2$. This
implies that $F(y)=0$ for $y\in \mathring{\Sigma}_C$, $y$ near $x$,
which in turn yields that
$(\mf L^{\mc E} F)(x) = (\mf L^{\mc C} F|_{\mathring{\Sigma}_C} )(x) =
0$. This completes the proof of the lemma.
\end{proof}

The next lemma is a modification of \cite[Lemma 4.4]{BJL}.

\begin{lemma}
\label{lem:funcext}
Fix a nonempty, proper subset $A$ of $S$ and a function $F$ in
$\mc E_A$.  Let $B = S\setminus A$. For every $\epsilon >0$ there
exists a function $H=H_\epsilon\colon \Sigma \to \bb R$ in $\mc E_S$
such that
\begin{equation}
\label{14}
F(x) = H(x) \text{ and } \mf L^{\mc E} F(x) = \mf L^{\mc E} H(x)
\text{ for all $x\in\Sigma$ such that } \min_{j\in B} x_j \geq \epsilon\,.
\end{equation}
\end{lemma}

\begin{proof}
By Lemma \ref{phi}, there exist a function $G : \Sigma \to \bb R$ and
$0<\delta<\epsilon$ such that 
\begin{enumerate}[leftmargin=*]
\item $G \in \mc D_S$,

\item
$G(x) = 1, \text{ for all } x\in \Sigma \text{ such that } \min_{j\in
B} x_j \geq \epsilon,$

\item $G(x) = 0$ for all
$x\in \Sigma \text{ such that } \min_{j\in B} x_j \leq \delta$.
\end{enumerate}
Define $H(x) = F(x)G(x)$. We claim that $H$ is the desired function.

\smallskip \noindent{\it Claim A:} The function $H$ belongs to
$\mc E_S$. \smallskip

As $F$ and $G$ satisfy condition $\mathfrak{E2}$, so does $H$.  We
turn to condition $\mathfrak{E1}$.  For $i \in B$, $G(x)=0$ for
$x_i \le \delta$. So condition $\mathfrak{E1}(i)$ is trivial.  For
$i \in A$, $F$ and $G$ satisfy condition $\mathfrak{E1}(i)$, and so
does $H$.  This proves that $H$ belongs to $\mc E_S$.

Since $G(x) = 1$ for all $x\in \Sigma$ such that
$\min_{j\in B} x_j \geq \epsilon$, $H$ fulfills condition \eqref{14}.
\end{proof}

We may finally construct the superharmonic function in the domain
$\mc E_S$.  Fix a nonempty subset $A$ of $S$ and let $B=S\setminus A$.
For $\epsilon>0$, set
$$
\cb{\bs a_0(\epsilon)} \,:=\, \min\{ \lambda^{B\cup D, D}(\epsilon) :
\varnothing \subsetneq D \subset A\},
$$
and let
\begin{equation}
\label{Kep}
\cb{ K_\epsilon} \,:=\,
\big\{ x \in \Sigma : \max_{k\in A} x_k \leq \bs
a_0(\epsilon) \text{ and } \min_{i\in B} x_i \geq \epsilon
\big\}\, .
\end{equation}
Fix $0<\gamma<1$, a subset $\varnothing \subsetneq D \subset A$,
and recall from \eqref{07} the definition of the function
$F_{A\setminus D,D} \colon \Sigma_{D\cup B} \to \bb R$.  Apply Lemma
\ref{lem:auxfunc} to extend the function $F_{A\setminus D,D}$ to a function
$G_D\colon \Sigma \to \bb R$ which belongs to $\mc E_A$ and such that
\begin{equation}
\label{eq:FD}
\begin{gathered}
G_D  (x) \ge 0\,, \quad 
G_D(x) = F_{A\setminus D,D} (x), \quad x\in \Sigma_{B\cup D},
\\
\mf L^{\mc E} G_D (x) = 0 \text{ if } \pi_D(x) = 0. 
\end{gathered}
\end{equation}
Apply Lemma \ref{lem:funcext} to the function $G_D\in \mc E_A$ to
obtain a function $H_D^{\epsilon}\colon \Sigma \to \bb R$ in $\mc E_S$
such that
\begin{equation}
\label{eq:HD}
H^\epsilon_D (x) = G_D(x) \text{ and }
\mf L^{\mc E} H^\epsilon_D (x) = \mf L^{\mc E} G_D(x),
\text{ for all $x\in \Sigma$ such that} \min_{j\in B} x_j \geq
\epsilon\,. 
\end{equation}
We summarize the properties of the function $H^\epsilon_D$. It
belongs to the domain $\mc E_S$, and 
\begin{equation}
\label{eq:HD2}
\begin{gathered}
H_D^\epsilon(z) = 0 \text{ for } z\in \Sigma_B
\text{ such that } \min_{j\in B} z_j \geq \epsilon\,,
\\
H_D^\epsilon(x) = F_{A\setminus D,D}(x) \geq 0  \text{ for }
x\in \Sigma_{B\cup D} \text{ such that } \min_{j\in B} x_j \geq
\epsilon\,.
\end{gathered}
\end{equation}

\subsection{First time interval}

Recall from Section \ref{sec:An absorbed diffusion} the definition of
the sequence of stopping times $(\sigma_n: n\ge 0)$. The main result
of this subsection states that the process remains absorbed at the
boundary of $\Sigma$ in the time-interval $[0,\sigma_1)$.

\begin{proposition}
\label{prop:firsttimeinterval}
Fix $z\in \Sigma$, and let $A = \{i\in S : z_i = 0\}$,
$B = S\setminus A$. Assume that $A$ is nonempty. Then,
$$
\bb P_z\big[\, \|X_t\|_A=0 \,,\,  0\leq t< \sigma_1 \,\big]
\,=\,    1\,.
$$
\end{proposition}

The proof of this result is divided into several steps.

\begin{lemma}
\label{lem:LHDbdd}
For all $\epsilon >0$ there exists a constant $C(\epsilon)>0$ such
that 
$$ \mf L^{\mc E} H^\epsilon_D(x) \leq C(\epsilon) \, \mtt 1
\big\{\, \pi_D(x)>0, \|x\|_{A\setminus D} > 0\, \big \}
$$
for all $x\in K_\epsilon$ and nonempty proper subset $D$ of $A$.
\end{lemma}

\begin{proof}
The proof is an adaptation of the proof of \cite[Lemma 5.5]{BJL}.  Fix
$\epsilon > 0$. Since each function
$\mf L^{\mc E} H^\epsilon_D, \varnothing \subsetneq D \subset A$, is
bounded on $K_\epsilon$,
\begin{equation}
\label{15}
C(\epsilon) \, :=\, \sup\{\|\mf L^{\mc E} H^\epsilon_D(x)\| :
x\in \Sigma, \varnothing \subsetneq D \subset A\} < \infty\,.
\end{equation}

By definition, $\min_{j\in B} x_j \geq \epsilon$ for all $x\in
K_\epsilon$. Thus, by the third property in \eqref{eq:FD} and by
\eqref{eq:HD},
\begin{equation}
\label{16}
\mf L^{\mc E} H^\epsilon_D(x) = \mf L^{\mc E} G_D(x) \,\mtt 1
\{\pi_D(x)>0\}
\quad\text{so that}\;\;
\mf L^{\mc E} H^\epsilon_D(x) = \mf L^{\mc E} H^\epsilon_D(x) \,\mtt  1
\{\pi_D(x)>0\}
\end{equation}
for all sets $\varnothing \subsetneq D\subset A$ and $x\in
K_\epsilon$.

On the other hand, if $\pi_D(x) > 0$ and $\|x\|_{A\setminus D} = 0$
for some $x\in K_\epsilon$, then $x\in \mathring{\Sigma}_{D\cup B}$.
Therefore, by \eqref{eq:HD}, by the second property in
\eqref{eq:FD}, and since $F_{A\setminus D,D} = F_D$ on
$\mathring{\Sigma}_{D\cup B}$, 
\begin{gather*}
\mf L^{\mc E} H^\epsilon_D(x) = \mf L^{\mc E} G_D(x)
= \mf L^{\mc E} F_{A\setminus D,D} (x)
= \mf L^{\mc E} F_{D} (x)
= \mf L^{B \cup D} F_{D} (x) \,.
\end{gather*}
This quantity is negative by the definition of $\bs a_0(\epsilon)$ in
\eqref{Kep} and Lemma \ref{lem:supharm}.  Therefore, by \eqref{16},
$$\mf L^{\mc E} H^\epsilon_D(x) \leq
\, \mtt 1 \big\{\pi_D(x)>0, \|x\|_{A\setminus D} > 0\big\}\, 
\mf L^{\mc E} H^\epsilon_D(x)\,,  \quad x\in K_\epsilon.
$$
This completes the proof of the lemma in view of \eqref{15}. 
\end{proof}

For every $\epsilon>0$, define $\tau_\epsilon$ as the exit time from
the set $K_\epsilon$:
$$
\tau_\epsilon \,:=\, \inf\{t\geq 0: X_t \notin K_\epsilon\}.
$$

\begin{lemma}
\label{lem:prodbdd}
Fix $z\in \mathring{\Sigma}_B$.  For all
$0<\epsilon< \min_{j\in B} z_j$ and nonempty subset $D$ of $A$,
$$
\bb P_z  \big[\, \pi_D (X_t) = 0 \,, 0\leq t \leq \tau_\epsilon \,\big] =
1 \,.
$$
\end{lemma}

\begin{proof}
Fix $z\in \mathring{\Sigma}_B$, $0<\epsilon< \min_{j\in B} z_j$. We
first prove the lemma for $D=A$.  Recall from \eqref{31} the
definition of the function $F_A\in \mc E_A$. By Lemma
\ref{lem:funcext}, there exists $H^\epsilon_A\colon \Sigma\to\bb R$ in
$\mc E_S$ with
\begin{equation}
\label{eq:HA}
F_A(x) = H^\epsilon_A(x) \text{ and } \mf L^{\mc E} F_A(x)
= \mf L^{\mc E} H^\epsilon_A(x), \text{ for all $x\in \Sigma$ such
that } \min_{j\in B} x_j \geq \epsilon\,.
\end{equation}

Fix $t>0$. Since $H^\epsilon_A \in \mc E_S$, 
$$
\bb E_z [\, H^\epsilon_A(X_{t\wedge \tau_\epsilon})\,]
\,=\, H^\epsilon_A(z) \,+\,
\bb E_z \Big[\, \int_{0}^{t\wedge \tau_\epsilon}
\mf L^{\mc E} H^\epsilon_A(X_s) ds\,\Big]\,.
$$
By the definition of $\tau_\epsilon$, by \eqref{eq:HA}, and by Lemma
\ref{lem:supharm}, the expectation term of the right hand side in the
above equation is negative. Therefore,
$$
\bb E_z\left[H^\epsilon_A(X_{t\wedge \tau_\epsilon})\right] \leq
H^\epsilon_A(z)\,.
$$
By \eqref{eq:HA}, we may replace $H^\epsilon_A$ by $F_A$, so that
$\bb E_z\left[F_A(X_{t\wedge \tau_\epsilon})\right] \leq 0$.  This
implies that
$$
\bb E_z\left[\pi_A(X_{t\wedge \tau_\epsilon})\right] = 0\,.
$$
By considering a countable dense set of times in $\bb R_+$, we
complete the proof of the lemma in the case $D=A$.

For $\varnothing \subsetneq D\subsetneq A$, we use an induction
argument on $|A|-|D|$. Fix $0\le n < |A|-1$, and assume that the
assertion of the lemma holds for all sets $C\subset A$ with
$|C|\ge |A|-n$.  Consider a subset $D'\subset A$ such that
$|D'|=|A| - n - 1$. By the first part of the proof.
\begin{equation}
\label{eq:rechyp}
\bb P_z\big[\, \pi_{D'}(X_{s\wedge\tau_\epsilon})>0\,,
\|X_{s\wedge\tau_\epsilon}\|_{A\setminus D'}>0\,\big] = 0
\end{equation}
for all $s\ge 0$. Fix $t \ge 0$. Recall from \eqref{eq:HD} the
definition of the function $H^\epsilon_{D'}$.  Since
$H^\epsilon_{D'} \in \mc E_S$,
$$\bb E_z\left[H^\epsilon_{D'}(X_{t\wedge \tau_\epsilon})\right]
= H^\epsilon_{D'}(z) + \bb E_z\Big[\, \int_{0}^{t\wedge \tau_\epsilon}
\mf L^{\mc E} H^\epsilon_{D'}(X_s) \, ds \, \Big] \,.
$$
Thus, as $\epsilon < \min_{j\in B} z_j$, by the first property in
\eqref{eq:HD2} and by Lemma \ref{lem:LHDbdd},
$$
\bb E_z\left[H^\epsilon_{D'}(X_{t\wedge \tau_\epsilon})\right] \le
C(\epsilon)\, \bb E_z\left[\int_0^{t\wedge {\tau_\epsilon}} \mtt 
1\{\pi_{D'}(X_s)>0, \|X_s\|_{A\setminus D'}>0\}\, ds\right].
$$
By \eqref{eq:rechyp}, the right-hand side of the previous expression
vanishes.  Hence, by the second property of \eqref{eq:HD2}, 
$$
\bb E_z\left[\mtt 1\{\|X_{t\wedge\tau_\epsilon}\|_{A\setminus D'}=0\}
F_{A\setminus D', D'}(X_{t\wedge\tau_\epsilon})\right]
\,\le \, \bb E_z[F_{A\setminus D', D'}(X_{t\wedge
\tau_\epsilon)})]= \bb E_z\left[H^\epsilon_{D'}(X_{t\wedge
\tau_\epsilon})\right] \le 0\,.
$$
As $F_{A\setminus D',D'} = F_{D'}$ on $\Sigma_{B\cup D'}$, and
$X_{t\wedge \tau_\epsilon} \in \Sigma_{B}$ [because
$\|X_{t\wedge\tau_\epsilon}\|_{A\setminus D'}=0$],
$$
\bb P_z\left[\|X_{t\wedge \tau_\epsilon}\|_{A\setminus D'} =0,
\pi_{D'}(X_{t\wedge \tau_\epsilon})>0 \right] = 0.
$$
Combining this identity with \eqref{eq:rechyp} yields that
$$
\bb P_z\left[\pi_{D'}(X_{t\wedge \tau_\epsilon} > 0)\right] = 0 .
$$
Finally, by considering a countable dense subset of times in
$\bb R_+$, we establish that the assertion of the lemma holds for
$D'$, which concludes the proof.
\end{proof}

\begin{proof}[Proof of Proposition \ref{prop:firsttimeinterval}]
Applying Lemma \ref{lem:prodbdd} for $D=\{j\}$, $j\in A$, yields that for
any $\epsilon < \min_{j\in B} z_j$,
$$ \bb P_z \left[\|X_t\|_A = 0 \text{ for all } 0\le t \le \tau_\epsilon \right] = 1 .$$
Since $\tau_\epsilon$ is the first time in which either $\max_{k\in A} x_k > \bs a_0(\epsilon) $ or $\min_{i\in B} x_i <\epsilon$,
$$ \bb P_z \left[\|X_t\|_A = 0 \text{ for all } 0\le t \le h_B(\epsilon) \right] = 1 ,$$
where, $h_B(\epsilon)$ is the exit time of the domain $\min_{i\in B} x_i \ge \epsilon$:
$$h_B(\epsilon) \,:=\, \inf\{t\ge 0 : \min_{i\in B} x_i < \epsilon\} .$$
Letting $\epsilon \downarrow 0$, we obtain Proposition
\ref{prop:firsttimeinterval}.
\end{proof}

\subsection{Absorption at the boundary}

Recall from Section \ref{sec:An absorbed diffusion} the definition of
$\sigma_n$ and $\ms B_n$, $n\ge 0$.  Using regular probability
distributions, as in \cite[Section 5.2]{BJL}, yields the following
proposition.

\begin{proposition}
\label{p1}
 For all $x\in \Sigma$, $n\geq 0$,
    $$ \bb P_x\big[\,
    \sigma_n = 0 \text{ or } \ms B_n = \ms B(X_t) \text{ for all }
    t\in [\sigma_n, \sigma_{n+1}) \,\big] = 1\,.$$
    \end{proposition}

\begin{proof}[Proof of Theorem \ref{thm:abs}]
The assertions is a direct consequence of Proposition \ref{p1}.
\end{proof}

\section{Proof of Theorem \ref{thm:mainthm}}
\label{sec8}

In addition to proving the theorem, this section presents some
properties of the boundary dimension-decaying diffusion process
characterized by the martingale problem introduced in the Definition
\ref{def1}.

\subsection*{Existence of a solution}

As the proof of Theorem \ref{thm:mainthm} is identical to the one of
\cite[Section 6 and 7]{BJL}, we only give a brief sketch.  The
existence part consists of two steps. We start with the tightness.

\begin{proposition} \textup{(\cite[Proposition 7.6]{BJL})} For any
sequence $x_N\in \Sigma_N, N\geq 1$, the sequence of laws
$\{ \bb P^N_{x_N} : N\geq 1 \}$ is tight. Moreover, every limit point
of the sequence is concentrated on continuous trajectories.
\end{proposition}

Next result asserts that any limit point is a solution of the martingale
problem for $(\mf L, \mc D_S)$.

\begin{proposition} \textup{(\cite[Proposition 7.7]{BJL})} Let
$x_N \in \Sigma_N, N\geq 1$, be a sequence converging to some
$x\in \Sigma$, and denote by $\wt{\bb P}$ a limit point of the
sequence $\bb P^N_{x_N}$.  Under $\wt{\bb P}$, for any
$H \in \mc D_S$,
    \begin{equation*}
        H(X_t) - H(X_0) - \int_0^t \mf L H(X_s) ds,
    \end{equation*}
    is a martingale.
\end{proposition}

The proofs of these results presented in \cite{BJL} for the
supercritical regime $b > 1$ apply to the critical case $b = 1$.  In
particular, these results guarantee the existence of a solution
$\bb P_x, x\in \Sigma$, of the martingale problem for the generator
$(\mf L, \mc D_S)$.

\subsection*{Uniqueness. An alternative martingale problem}

Before showing the uniqueness of solutions for the $(\mf L, \mc D_S)$
martingale problem, we show that a solution $\bb P$ of a
$(\mf L, \mc D_S)$-martingale problem also solves an alternative
martingale problem.

Let $D_0(\Sigma)$ be the set of functions $F:\Sigma \to \bb R$ such
that, for all $B\subset S$ with at least two elements, $F|_{\Sigma_B}$
belongs to $C^2(\Sigma_B)$ and has compact support contained in
$\mathring{\Sigma}_B$. For $F\in D_0(\Sigma)$, we define
$\mc L F:\Sigma \to \bb R$ as follows: For $x\in \Sigma$, let
$B = \{i\in S \;:\; x_i \neq 0\}$. Then
\begin{equation*}
    \mc L F(x) = \begin{cases}
        (\mf L^B F)|_{\Sigma_B} (x), & \text{if } |B| \geq 2, \\
        0, & \text{otherwise}.
    \end{cases}
    \end{equation*}
    
    Recall from Section \ref{sec:An absorbed diffusion} the definition
    of the sequence of stopping times $(\sigma_n)_{n\ge 0}$.
    Consider the jump process
$$ N_t \,:=\, \sup\{n\geq 0 \;:\; \sigma_n \leq t\}, \quad t\geq 0,$$
and define $N^S_t \,:=\, N_t \wedge |S|, t\geq 0$. Clearly, since $\bb P$ is absorbing,
$$ \bb P[N_t = N^S_t, \text{ for all } t\geq 0] = 1. $$

\begin{theorem} \label{thm:altmart} \textup{(\cite[Theorem 2.5]{BJL})} Suppose that $\bb P$ is a solution of the martingale problem for $(\mf L, \mc D_S)$.
    For each $x\in \Sigma$ and any $F \in D_0(\Sigma)$,
    \begin{equation*}
    F(X_t) - \int_0^t \mc L F(X_s) ds - \int_0^t F(X_s) dN^S_s, \quad t\geq 0,
    \end{equation*}
    is a $\bb P$-martingale.
\end{theorem}

This new martingale is referred to as a $\mc L$-martingale.  The following
proposition gives the uniqueness of a solution of the alternative
martingale problem.

\begin{prop} \label{prop:altuni} \textup{(\cite[Proposition 6.1]{BJL})}
    For each $x\in \Sigma$, there exists at most one absorbing solution of the $\mc L$-martingale problem starting at $x$.
\end{prop}

Theorem \ref{thm:mainthm} is a direct consequence of this result,
Theorem \ref{thm:abs}, Theorem \ref{thm:altmart}, and Proposition
\ref{prop:altuni}.

Proposition \ref{prop:altuni} also gives Proposition \ref{prop:restriction}.
Fix $x\in \Sigma$ and assume that $\ms A(x) = \{j \in S: x_j = 0\} \neq \varnothing$.
Let $B = \ms A(x)^c$. By Theorem \ref{thm:abs}, the measure $\bb P_x^B$ of Proposition \ref{prop:restriction} is a well-defined probability measure that solves $\mf L$-martingale problem.
By Theorem \ref{thm:altmart}, this also solves $\mc L$-martingale restricted to $\Sigma_B$.
The uniqueness property established in Proposition \ref{prop:altuni}
immediately yields the desired conclusion.

\subsection*{Additional properties}
According to \cite[Section 7.3]{BJL}, the solution $\{\bb P_x : x\in \Sigma\}$ of the martingale problem satisfies three additional properties.
Also for the critical case $b = 1$, these properties are satisfied by the solution of the martingale problem for $(\mf L, \mc D_S)$, defined in Theorem \ref{thm:mainthm}, and the exact same proof applies.

The first property states that the solution has the Feller continuity property.
\begin{proposition} \label{prop:feller} \cite[Proposition 7.10]{BJL}
    Let $(x_n)_{n\geq 1}$ be a sequence in $\Sigma$ converging to some $x\in \Sigma$. Then $\bb P_{x_n} \to \bb P_x$ in the sense of weak convergence of measures on $C(\bb R_+, \Sigma)$.
\end{proposition}

The second property tells us that the solution satisfies the strong Markov property.
\begin{proposition} \label{prop:strongmarkov} \cite[Proposition 7.11]{BJL}
    Fix $x\in \Sigma$. Let $\tau$ be a finite stopping time and $\{\bb P^\tau_\omega\}$ be a regular conditional probability distribution of $\bb P_x$ given $\mc F_\tau$.
    Then, there exists a $\bb P_x$-null set $\mc N \in \mc F_\tau$, such that
    $$ \bb P^\tau_\omega \circ \theta^{-1}_{\tau(\omega)} = \bb P_{X_\tau(\omega)},\;\; \omega \in \mc N^c,$$
    where we recall $(\theta_t)_{t\geq 0}$ is the semigroup of time translations.
\end{proposition}

Proposition \ref{prop:feller} and \ref{prop:strongmarkov} together
imply that the solution is actually a Feller process.  The last
property provides a uniform bound on the expected value of the
absorption time $\sigma_1$ for all initial points $x \in \Sigma$.

\begin{proposition}
\label{ft}
\cite[Proposition 7.12]{BJL} Let $z\in \Sigma$ be such that
$z\neq \bs e_j$, $j\in S$. For any $q > b$,
    $$ \bb E_z[\sigma_1] \leq \frac{|B|^{(q-1)\vee 1}}{(q+1)\, (q-b)\, d(B)},$$
    where $B = \{i\in S : z_i \neq 0\}$ and
    $d(B) = \min_{j\in B} \frac{1}{2}\sum_{k\neq j} \left(m_jr(j,k)+
    m_k r(k,j)\right)$. In particular,
    $\bb P_z[\sigma_1 < \infty] = 1$.
    \end{proposition}
    
From the proposition, we can conclude that as time flow, the process
successively absorbs into the decreasing subsimplices and eventually
reaches the vertices $\Sigma_j$, $j\in S$, in a time which has
finite expectation. Note that we do not prove that the dimension
decays only by one at each step, though we believe that this happens.

\appendix

\section{Properties of $C^1$ functions on $\Sigma$}

We recall from the definition of tangents vectors of $\Sigma$, $T_\Sigma$, $C^1(\mathring{\Sigma})$, and $C^1(\Sigma)$ from Section 2. 

\begin{lemma} \label{app:1}
    Suppose we have a $F \in C(\Sigma)$, $V \in C(\Sigma, T_\Sigma)$. Then the following are equivalent:
    \begin{enumerate}[leftmargin=*]
        \item $F \in C^1(\Sigma)$ and $\nabla^\Sigma F = V$.
        \item For all $x, y \in \Sigma$, we have
        $$ F(y) - F(x) = \int_0^1  V(x + t(y-x)) \cdot (y-x) dt, $$
        where $\cdot$ is the standard inner product in $\bb R^S$.
    \end{enumerate}
\end{lemma}

\begin{proof}
    (1) $\Rightarrow$ (2): Since $F \in C^1(\Sigma)$, we have $F|_{\mathring{\Sigma}} \in C^1(\mathring{\Sigma})$.
    Therefore, for $p, q \in \Sigma$, we have
    $$ F(q) - F(p) = \int_0^1 \nabla F(p + t(q-p)) \cdot (q-p) dt, $$
    For arbitrary $x, y \in \Sigma$, we take a sequence $p_n \to x$ and $q_n \to y$ with $p_n, q_n \in \mathring{\Sigma}$, the equation
    $$ F(q_n) - F(p_n) = \int_0^1 \nabla F(p_n + t(q_n-p_n)) \cdot (q_n-p_n) dt$$
    converges to the equation
    $$ F(y) - F(x) = \int_0^1 \nabla F(x + t(y-x)) \cdot (y-x) dt.$$
    Here, we used the uniform continuity of $\nabla F$ on $\Sigma$.

    (2) $\Rightarrow$ (1): Fix $x \in \mathring{\Sigma}$. The equation implies
    $$ F(y) - F(x) - (y-x) \cdot V(x) =  \int_0^1 (V(x + t(x-y)) - V(x)) \cdot (x-y) dt .$$
    Using the uniform continuity of $V$, we have
    $$ \lim_{y\to x} \frac{ F(y) - F(x) - (y-x) \cdot V(x)}{|y-x|} = 0.$$
    This implies that $F$ is differentiable at $x$ and $\nabla F(x) = V(x)$.
    This completes the proof.
\end{proof}

\begin{lemma} \label{app:4}
    Fix $F \in C^1(\Sigma_B)$. Then $\gamma^*_B F \in C^1(\Sigma)$ and for any $ \bs V \in T_\Sigma$, we have
    $$ \nabla_{\bs V} (\gamma^*_B F) (x) = \nabla_{\gamma_B(\bs V)} F (\gamma_B(x)). $$
\end{lemma}

\begin{proof}
    It is enough to show that the equlity holds for $x\in \mathring{\Sigma}$. Fix $x \in \mathring{\Sigma}$.
    From \eqref{eq:proj}, we also have $\gamma_B(x) \in \mathring{\Sigma}_B$. Consider small enough $\epsilon>0$ such that $x + \epsilon \bs V \in \mathring{\Sigma}$ and $\gamma_B(x) + \epsilon \gamma_B(\bs V) \in \mathring{\Sigma}_B$.
    The left hand side is equal to
    $$ \lim_{\epsilon \to 0} \frac{\gamma^*_B F(x + \epsilon \bs V) - \gamma^*_B F(x)}{\epsilon} = \lim_{\epsilon \to 0} \frac{F(\gamma_B(x) + \epsilon \gamma_B(\bs V)) - F(\gamma_B(x))}{\epsilon} = \nabla_{\gamma_B(\bs V)} F(\gamma_B(x)).$$
\end{proof}

Recall the definition of $C^2_b(\mathring{\Sigma})$. Here is an easier criteria to check whether $F\in C^1(\Sigma)$.

\begin{lemma} \label{app:6}
    Suppose $F\in C(\Sigma)$ and $F|_{\mathring{\Sigma}}\in C^2(\mathring{\Sigma})$. Then, $F \in C^1(\Sigma)$.
\end{lemma}

\begin{proof}
    Take $\nabla F : \mathring{\Sigma} \to T_{\Sigma}$.
    For $x \in \Sigma$, and any sequence $x_n \to x$ with $x_n \in \mathring{\Sigma}$, it is enough to show that the sequence
    $ \nabla F (x_n)$ is cauchy. Observe that
    \begin{align*}
    \nabla F (x_n) - \nabla F (x_m) &= \int_0^1 \nabla_{x_n-x_m} \nabla F ( (1-t)x_n + t x_m ) dt \\
    &= |x_n - x_m| \int_0^1 \nabla_{\frac{x_n-x_m}{|x_n-x_m|}} \nabla F ( (1-t)x_n + t x_m ) dt.
    \end{align*}
    Since $F\in C^2_b(\mathring{\Sigma})$, we have $|\nabla F(x_n) - \nabla F(x_m) | \le C |x_n - x_m| $ for some $C>0$ from the above equation.
    This completes the proof.
    \end{proof}

\section*{Acknowledgements}
C. L. has been partially supported by FAPERJ CNE E-26/201.117/2021, by
CNPq Bolsa de Produtividade em Pesquisa PQ 305779/2022-2.

\end{document}